\documentclass[11pt]{article}

\usepackage{epsfig,graphics}
\usepackage{amssymb,amsmath,amsthm,bm}
\usepackage{epstopdf}
\usepackage{graphicx}
\usepackage{xcolor}
\usepackage{tikz}
\graphicspath{ {images/} }
\usepackage{color}
\usepackage{float}
\usepackage{bbm}
\usepackage{dsfont}

\newtheorem{theo}{Theorem}
\newtheorem{lem}{Lemma}

\newtheorem{cor}{Corollary}
\newtheorem{mrem}{Remark}

\newtheorem{example}{Example}

\def\R{{\mathbb R}}

\def\RR{{\mathbb R}}

\def\ga{\alpha}
\def\gga{\gamma}
\def\gth{\theta}

\def\gb{\beta}

\def\gs{\sigma}
\def\gl{\lambda}
\def\gL{\Lambda}
\def\wt{\widetilde}

\def\gp{{\prime}}

\def\ep{\epsilon}
\def\vep{\varepsilon}

\def\gt{\triangle}

\def\b0{{\bf 0}}
\def\1{{\bf 1}}

\def\bd{{\bf d}}

\def\cE{\mathcal E}
\def\cG{\mathcal G}
\def\cH{\mathcal H}

\def\rem{{\rm rem}}

\def\err{{\rm res}}

\def\terr{{\rm err}}
\def\tErr{{\rm Err}}

\def\bd{{\rm bd}}
\def\Bd{{\rm Bd}}

\def\wh{\widehat}

\def\wt{\widetilde}

\makeatletter
\def\widebreve{\mathpalette\wide@breve}
\def\wide@breve#1#2{\sbox\z@{$#1#2$}%
     \mathop{\vbox{\m@th\ialign{##\crcr
\kern0.08em\brevefill#1{0.8\wd\z@}\crcr\noalign{\nointerlineskip}%
                    $\hss#1#2\hss$\crcr}}}\limits}
\def\brevefill#1#2{$\m@th\sbox\tw@{$#1($}%
  \hss\resizebox{#2}{\wd\tw@}{\rotatebox[origin=c]{90}{\upshape(}}\hss$}
\makeatletter
\def\wb{\widebreve}

\usepackage{scalerel,stackengine}
\stackMath
\newcommand\reallywidecheck[1]{%
\savestack{\tmpbox}{\stretchto{%
  \scaleto{%
    \scalerel*[\widthof{\ensuremath{#1}}]{\kern-.6pt\bigwedge\kern-.6pt}%
    {\rule[-\textheight/2]{1ex}{\textheight}}%WIDTH-LIMITED BIG WEDGE
  }{\textheight}% 
}{0.5ex}}%
\stackon[1pt]{#1}{\scalebox{-1}{\tmpbox}}%
}
\def\wc{\reallywidecheck}

\begin{document}

\title{Signal Separation Based on Adaptive Continuous Wavelet Transform  and Analysis\thanks{
This work is partially supported by the Hong Kong Research Council, under Projects $\sharp$ 12300917 and $\sharp$ 12303218, and HKBU Grants $\sharp$ RC-ICRS/16-17/03 and $\sharp$ RC-FNRA-IG/18-19/SCI/01, the Simons Foundation, under Grant $\sharp$ 353185, and the National Natural Science Foundation of China, 
under Grants $\sharp$ 62071349, $\sharp$ 61972265  and  $\sharp$ 11871348, by National Natural Science Foundation of Guangdong Province of China,  under Grant $\sharp$ 2020B1515310008, by Educational Commission of Guangdong Province of China,  under Grant $\sharp$ 2019KZDZX1007, and by Guangdong Key Laboratory of Intelligent Information Processing, China.
}
}

\author{Charles K. Chui${}^{1}$,  Qingtang Jiang${}^2$, Lin Li${}^3$, and  Jian Lu${}^{4}$%\thanks{Corresponding author.  E-mail: jianlu@szu.edu.cn.}
}

\date{}
%\date{\today}

\maketitle

%\centerline{Draft}

\bigskip
%{\small 1. College of Mathematics \& Statistics, Shenzhen University, Shenzhen 518060, China}

{\small 1. Department of Mathematics, Hong Kong Baptist University, Hong Kong.}

{\small 2. Department of Mathematics \& Statistics, 
%Dept. of  Math \& CS, 
University of Missouri-St. Louis, St. Louis,  MO 63121, USA.} %  E-mail: jiangq@umsl.edu}

{\small 3. School of Electronic Engineering, Xidian University, Xi${}'$an 710071, China.}  
%E-mail: lilin@xidian.edu.cn}

{\small 4. Shenzhen Key Laboratory of Advanced Machine Learning and Applications, College of Mathematics \& Statistics, Shenzhen University, Shenzhen 518060, China.} % E-mail: jianlu@szu.edu.cn}

\begin{abstract}
In nature and the technology world, acquired signals and time series are usually affected by multiple complicated factors and appear as multi-component non-stationary modes.   In many situations it is necessary to separate these signals or time series to a finite number of mono-components to represent the intrinsic modes and underlying dynamics implicated in the source signals. 
Recently the synchrosqueezed transform (SST) was developed as an empirical mode decomposition (EMD)-like tool to enhance the time-frequency resolution and energy concentration of a multi-component non-stationary signal and provides more accurate component recovery.  To recover individual components,  the SST method consists of two steps. First the  instantaneous frequency (IF) of a component is estimated from the SST plane. Secondly,  after IF is recovered, the associated component is computed by a definite integral along the estimated IF curve on the SST plane.  The reconstruction accuracy for a component depends heavily on the accuracy of the IFs estimation carried out in the first step. More recently, a direct method of the time-frequency approach, called signal separation operation (SSO), was introduced for %solving the inverse problem of 
multi-component signal separation. While both SST and SSO are mathematically rigorous on IF estimation, SSO avoids the second step of the two-step SST method in component recovery (mode retrieval).  The SSO method is based on some variant of the short-time Fourier transform. In the present paper, we propose a direct  method of signal separation based on the adaptive continuous wavelet-like transform (CWLT) by introducing two models of the adaptive CWLT-based approach for signal separation: the sinusoidal signal-based model and the linear chirp-based model, which are derived respectively from sinusoidal signal approximation and the linear chirp approximation at any time instant. A more accurate component recovery formula is derived from linear chirp local approximation.  We present the theoretical analysis of our approach. For each model, we establish the error bounds for  IF estimation and component recovery. 

\end{abstract}

\section{Introduction}
Real-world signals  and time series are mostly non-stationary and multi-component, given by
\begin{equation}
\label{AHM0}
x(t)=A_0(t)+\sum_{k=1}^K x_k(t), \qquad x_k(t)=A_k(t)  \cos \left(2\pi \phi_k(t)\right), 
\end{equation}
with $A_k(t), \phi_k'(t)>0$, where  $A_0(t)$ is the trend, and $A_k(t), 1\le k\le K$, are called the instantaneous amplitudes  and $\phi'_k(t)$ the instantaneous frequencies (IFs).  Therefore, modeling a non-stationary signal  $x(t)$ as in \eqref{AHM0}
is instrumental to extract information hidden in $x(t)$.
 
In this regard, the empirical mode decomposition (EMD) scheme  introduced in \cite{Huang98} is a popular method to decompose a non-stationary signal into the form of  \eqref{AHM0}.
There are many articles that study the properties of EMD and propose its variants, see e.g. 
\cite{Cicone20, HM_Zhou16, Flandrin04, LCJJ19, HM_Zhou09,Oberlin12a, Rilling08,van20,Y_Wang12,Wu_Huang09, Xu06}. However EMD could lead to mode mixtures or artifacts \cite{Li_Ji09} and cannot be used to recover the actual component as modelled by \eqref{AHM0}.

On the other hand, the objective of signal separation is to solve an inverse problem of recovering signal component $x_k(t)$ in \eqref{AHM0}. For stationary signals, there is the classic work of De Prony (called Prony's method) \cite{De Prony}, and its improvements to the well-known MUSIC \cite{MUSIC} and ESPRIT \cite{ESPRIT} algorithms, based on the mathematical model of exponential sums (in terms of constant frequencies), for first extracting the frequencies, from which the sub-signals are recovered. We may call this the time-frequency approach for signal separation. 

The objective of EMD is to decompose the given multi-component signal $x(t)$ into a finite number of components, called intrinsic mode functions (IMFs), and then to reformulate each IMF $y_k(t)$ as an amplitude-frequency modulated signal, $y_k(t)=B_k(t)\cos \big(2\pi\gth_k(t)\big)$ artificially, by taking the real part of the polar formulation of the analytic extension of $y_k(t)$ via the Hilbert transform. On the contrary, the time-frequency approach is to solve the inverse problem of recovering the signal components $x_k(t)=A_k(t)\cos \big(2\pi\phi_k(t)\big)$ from the blind-source data $x(t)$, assuming that $x(t)$ is governed by the signal model defined by \eqref{AHM0}, by first extracting the IFs $\phi^\gp_k(t)$ and then recovering the signal component $x_k(t)=A_k(t)\cos \big(2\pi\phi_k(t)\big)$ by using the recovered IF, for each $k$. In this regards, the first time-frequency approach for resolving the inverse problem for non-stationary signals is the synchrosqueezed transform (SST), introduced in \cite{Daub_Maes96} and 
discussed in the Princeton Ph.D. dissertation \cite{Wu_thesis} by using both the 
continuous wavelet transform (CWT) and the short-time Fourier transform (STFT). 
Further development based on the CWT is the pioneering paper  \cite{Daub_Lu_Wu11}, followed by another paper \cite{Thakur_Wu11} based on STFT (see also \cite{MOM14}). 
Of course, we may consider SST as an alternative to EMD and its variants, that overcomes some limitations of the EMD scheme \cite{Flandrin_Wu_etal_review13}. 
Other types of SST, such as SST with vanishing moment wavelets with stacked knots \cite{Chui_Lin_Wu15}, 
a hybrid EMD-SST computational scheme \cite{Chui_Walt15}, 
matching demodulation transform-based SST \cite{Li_Liang12,Wang_etal14, Jiang_Suter17,WCSGTZ18}, synchrosqueezed curvelet transform \cite{Yang14}, 
synchrosqueezed wave packet transforms \cite{Yang15}, the SST based on S-transform \cite{S_transform_SST15}, the multitapered SST \cite{Daub_Wang_Wu16}, 
 the 2nd-order SST %and high-order SST  
 \cite{MOM15,OM17, BMO18, Pham17b},  % Pham17, Lin_Cubic19}, 
 the adaptive SST \cite{Wu17, Saito17, LCHJJ18, LCJ18, CJLS18,LJL20}, have been proposed and studied. 

To recover the individual components $x_k(t)$,  the SST method consists of two steps. First the IF $\phi'_k(t)$ of each $x_k(t)$ is estimated from the SST plane. Secondly,  after IF is recovered, $x_k(t)$ is computed by a definite integral along each estimated IF curve on the SST plane.   
The reconstruction accuracy for $x_k(t)$ depends heavily on the accuracy of the IFs estimation carried out in the first step. On the other hand,  a direct time-frequency approach, called signal separation operation or operator (SSO), was introduced in \cite{Chui_Mhaskar15} for multi-component signal separation. The difference of the SSO approach is that the components are reconstructed simply  by substituting the time-frequency ridge to SSO. The component recovery formula in \cite{Chui_Mhaskar15} is derived based on sinusoidal signal approximation. When considered as a decomposition scheme, to overcome the limitations of EMD, a hybrid EMD-SSO computational scheme is developed in \cite{CMW16}. 
Recently the authors of \cite{LCJ20} show that  the SSO is related to the adaptive short-time Fourier transform (STFT). With the adaptive STFT, they then obtain a more accurate component recovery formula derived from the  linear chirp (also called linear frequency modulation signal) approximation at any local time. 
%In addition, they also proposed a recovery scheme to extract the signal components one by one with the time-varying window updated for each component.
Most recently the authors of \cite{CJLL20_adpSTFT} carry out 
a theoretical analysis of the component recovery formula in \cite{LCJ20} which is derived from linear chirp local approximation. In addition, linear chirp local approximation-based SSO approach has been extended recently in \cite{LHJC20} to 3-dimensional case with variables of time, frequency and chirp rate to recover components with crossover IFs.

The SSO approach in \cite{Chui_Mhaskar15,LCJ20,LHJC20, CJLL20_adpSTFT} are based on some variant of STFT. In this paper  we introduce and develop a signal separation method based on the adaptive continuous wavelets of the form:  
\begin{equation}
\label{wavelet_general}
\psi_\gs(t):= \overline{\frac 1\gs g(\frac t\gs)} e^{i2\pi \mu t}, 
\end{equation}
where $\gs>0, \mu>0$,  $g\in L_2(\R)$, and  $\mu$ is a fixed positive parameter. 
The parameter $\gs$  in $\psi_\gs(t)$ is also called the window width in the time-domain of wavelet $\psi_\gs(t)$.
The CWT of $x(t)$ with a time-varying parameter considered in \cite{LCJ18} is defined by
\begin{eqnarray}
%\label{def_CWT_para}
\nonumber \wt W_x(a, b)\hskip -0.6cm &&:= 
\int_{-\infty}^\infty x(t) \frac 1{a}\overline{\psi_{\gs(b)} \big(\frac{t-b}a\big)} dt\\
&&\label{def_CWT_para}
=\int_{-\infty}^\infty x(b+at) \frac 1{\gs(b)} g\Big(\frac t{\gs(b)}\Big)e^{-i2\pi \mu t} dt, 
\end{eqnarray}
where $\gs=\gs(b)$ is a positive function of $b$. Since in this paper, $\psi_\gs$ given by \eqref{wavelet_general} is not required to satisfy the so-called admissible condition, we call $\wt W_x(a, b)$  given by \eqref{def_CWT_para} the adaptive continuous wavelet-like transform (CWLT) of $x(t)$ with respect to $\psi_\gs$.

For simplicity of our presentation, we consider the complex version of 
\eqref{AHM0} with the trend $A_0(t)$ being removed, namely,
\begin{equation}
  \label{AHM}
  x(t)=\sum_{k=1}^K x_k(t)=\sum_{k=1}^K A_k(t) e^{i2\pi\phi_k(t)}
  \end{equation}
 with $A_k(t), \phi_k'(t)>0$. In addition, we assume that $\phi_{k-1}'(t)<\phi_k'(t), t\in \RR$ for $2\le k\le K$. The reader is referred to \cite{Chui_Mhaskar15} for the methods to remove the trend $A_0(t)$. We assume the following conditions hold: 
 \begin{eqnarray}
\label{cond_basic0}&&A_k(t)\in L_\infty(\R), \; A_k(t)>0, t\in \RR,\\
\label{cond_basic}&&  \phi_k(t)\in C^2(\R), \; \inf_{t\in \R} \phi_k'(t)>0, \; \sup_{t\in \R} \phi_k'(t)<\infty, \\
 %\label{cond_phi_2nd_der}&& |A'_k(t)|\le \vep \phi_k'(t),  \;  |\phi''_k(t)|\le \vep\phi_k'(t), \; t\in \R, \;M''_k:=\sup_{t\in \R} |\phi''_k(t)|<\infty, \\
\label{condition1} && |A_k(t+\tau)-A_k(t)|\le \vep_1 |\tau| A_k(t),  \; t\in \R, \; 1\le k\le K, 
\end{eqnarray}
for $\vep_1>0$. When $\vep_1$ is small, \eqref{condition1} means $A_k(t)$ changes slowly. 

In this paper we consider two models of the adaptive CWLT-based signal separation method: the sinusoidal signal-based model and the linear chirp-based model. For the sinusoidal signal-based model,  
we assume that the $\phi_k$ satisfy  
\begin{eqnarray}
\label{condition1b}&& |\phi''_k(t)|\le \vep_2, \; t\in \R, \; 1\le k\le K, 
\end{eqnarray}
where $\vep_2>0$ is a small number.  In this model, 
 each component $x_k(t)$ is well approximated locally by a sinusoidal signal at any local time $b$ in the sense that 
\begin{eqnarray*}
\wt W_{x_k}(a, b)\hskip -0.6cm &&= \int_{-\infty}^\infty A_k(b+at) e^{i2\pi\phi_k(b+at)}  \frac 1{\gs(b)} g\Big(\frac t{\gs(b)}\Big)e^{-i2\pi \mu t} dt \\
&& \approx  \int_{-\infty}^\infty A_k(b) e^{i2\pi(\phi_k(b)+\phi'_k(b) at)}  \frac 1{\gs(b)} g\Big(\frac t{\gs(b)}\Big)e^{-i2\pi \mu t} dt \\
&&= x_k(b)\wh g\big(\gs(b)(\mu-a \phi_k'(b))\big). 
\end{eqnarray*}
Under certain condition similar to the following well-separated condition for conventional SST  
\begin{equation}
\label{freq_resolution}
\frac{\phi_k'(t)-\phi'_{k-1}(t)}{\phi_k'(t)+\phi'_{k-1}(t)}\ge \gt, \;  t\in \R, \; 2\le k\le K, 
\end{equation}
%The condition \eqref{freq_resolution} is called the well-separated condition with resolution $\gt$.  
where $0<\gt <1$, the set $\{a: \; |\wt W_x(a, b)|>\wt \ep_1\}$, where $\wt \ep_1>0$, can be  expressed as non-overlapping union of $\cG_{b, k}$ with $\frac \mu{\phi'_k(b)}\in \cG_{b, k}$, $1\le k\le K$. 
Denote 
\begin{equation}
\label{def_max_eta}
%\wh a_0:=0, \; 
\wh a_k =\wh a_k(b):={\rm argmax}_{a \in\mathcal{G}_{b, k}  }
|\wt W_x(a,b)|, ~~ k=1,\cdots, K.
\end{equation}
 Then we can use $\mu/{\wh a_k(b)}$ to approximate $\phi'_k(b)$:
 \begin{equation}
 \label{IF_approx1}
\phi'_k(b) \approx   \frac {\mu}{\wh a_k(b)}; 
 \end{equation}
  and most importantly,  we may reconstruct each component $x_k(b)$ by simply substituting 
$\wh a_k(b)$ to $a$ in $\wt W_x(a, b)$:  
\begin{equation}
\label{comp_xk_est0}
x_k(b) \approx \wt W_x(\wh a_k(b), b).
\end{equation}

The sinusoidal signal-based model requires that the instantaneous frequency $\phi'_k$ changes slowly, namely \eqref{condition1b} holds for a small $\vep_2>0$. To separate a multi-component signal with significantly changing  instantaneous frequency components, we propose the linear chirp-based 
%linear frequency modulation (LFM) signal-based 
model. In this model, we assume that the $\phi_k$ satisfy the following conditions:   
\begin{eqnarray}
 \label{cond_basic_2nd} 
 && %A_k(t)\in C^2(\R)\cap L_\infty(\R),
  \phi_k(t)\in C^3(\R),  
%\phi'_k(t), 
\phi''_k(t) \in L_\infty(\R), \\
%\\ \label{cond_phi_3nd_der}&& |A'_k(t)|\le \vep, \; |A''_k(t)|\le \vep,  \;  
\label{condition2} && |\phi^{(3)}_k(t)|\le \vep_3, \; t\in \R, \; 1\le k\le K,  
\end{eqnarray}
% \begin{equation}
%\label{condition2}
%|\phi_k^{(3)}(t)|\le \vep_3, \; t\in \R, \; 1\le k\le K, 
%\end{equation}
%for some positive number $\vep_3$, where $1\le k\le K$.  
where $\vep_3>0$ is a small number.  In this case $\phi_k(t)$ is not required to satisfy \eqref{condition1b}, which means IF $\phi_k'(t)$ is allowed to change rapidly.  
In this linear chirp-based model, each component $x_k(t)$ is well approximated by a linear chirp at any local time  $b$ in the sense that 
\begin{eqnarray*}
\wt W_{x_k}(a, b)\hskip -0.6cm &&=\int_\R x_k(b+a t)\frac 1{\gs(b)}g\Big(\frac t{\gs(b)}\Big) e^{-i2\pi \mu t}dt\\
%\label{CWT_approx0}
&&\approx \int_\R A_k(b)e^{i2\pi (\phi_k(b)+\phi_k'(b) a t +\frac12\phi''_k(b) a^2t^2)}\frac 1{\gs(b)}g\Big(\frac t{\gs(b)}\Big) e^{-i2\pi \mu t}dt\\
%&&=\int_\R x_k(b)e^{i2\pi (\phi_k'(b) a t +\frac12\phi''_k(b) a^2t^2)}\frac 1{\gs(b)}g\Big(\frac t{\gs(b)}\Big) e^{-i2\pi \mu t}dt \\
&&%\label{CWT_approx}
= x_k(b)  %G_{0, k}(a, b)
G_k\big(\gs(b)(\mu-a \phi_k'(b) ), a, b\big), 
\end{eqnarray*}
where for any given $a, b$, we use $G_k(\xi, a, b)$ to denote the Fourier transform of $e^{i\pi \phi''_k(b)a^2\gs^2(b) t^2}g(t)$, namely,
\begin{equation}
\label{def_Gk_general}
G_k(\xi, a, b):=%{\cal F}\Big(e^{i\pi \phi''_k(b) a^2 \gs^2(b)t^2}g(t)\Big)\big(\xi)=
\int_{\R} e^{i\pi \phi''_k(b) a^2 \gs^2(b)t^2}g(t) e^{-i2\pi \xi t}dt.
\end{equation}
%with ${\cal F}$ denoting the Fourier transform. 
%Note that $G_k(\xi)$ depends on $a, b$ also if $\phi''_k(b)\not=0$. We drop $a, b$ in $G_k$ for simplicity. 

Under certain well-separated conditions, the set $\{a: \; |\wt W_x(a, b)|>\wt \ep_1\}$
can be  expressed as non-overlapping union of another group of sets $\cH_{b, k}$ with $\frac \mu{\phi'_k(b)}\in \cH_{b, k}$, $1\le k\le K$. Denote 
\begin{equation}
\label{def_max_eta_2nd}
\wc a_k :=\wc a_k(b):={\rm argmax}_{a \in\mathcal{H}_{b, k}  }\Big|\frac{\wt W_x(a, b)}{G_k(0,  a, b)}\Big|, ~~ k=1,\cdots, K.
\end{equation}
Then $\mu/{\wc a_k(b)}$ and $\wt W_x(\wc a_k(b), b)$ are the reconstructed instantaneous frequency and components:

\begin{equation}
 %\label{IF_approx2}
 \label{comp_xk_est_2nd1}
 \phi'_k(b)\approx \frac\mu{\wc a_k(b)}, \quad  x_k(t) \approx \frac 1{G_k(0, \wc a_k, b)}\wt W_{x}(\wc a_k, b).
\end{equation} 
Again, the reconstructed component of $x_k(b)$ is obtained by simply substituting 
$\wc a_k(b)$ to $a$ in $\wt W_x(a, b)$.   Note that there is a factor 
 $\frac 1{G_k(0, \wc a_k, b)}$ %or $\sqrt{1-i2\pi \phi''_\ell(b)\wc a_\ell^2\gs^2(b)}$ 
in \eqref{comp_xk_est_2nd1}. Hence, the linear chirp-based model is different from the  sinusoidal signal-based counterpart. 

We will present the theoretical analysis of our approach in the next two sections, Sections 2 and 3. More precisely, we will establish the error bounds for $|\mu-{\wh a_k(b)} \phi'_k(b)|$ and 
$|\wt W_x(\wh a_k(b), b)-x_k(b)|$ in Section 2, and  
obtain such error bounds for $|\mu-{\wc a_k(b)} \phi'_k(b)|$ and 
$|\wt W_x(\wc a_k(b), b)-x_k(b)|$ in Section 3. We will provide some experimental results in Section 4, the last section of this paper. 
Compared with \cite{Chui_Mhaskar15} and very recent papers on the CWT-based SSO \cite{Chui_Han20,Chui_Mhaskar20},  this paper proposes both the sinusoidal signal-based and the linear chirp-based adaptive CWLT models, while \cite{Chui_Mhaskar15,Chui_Han20,Chui_Mhaskar20} consider the sinusoidal signal-based model only in univariable and multivariable cases or with CWT of high varnishing moments.

Observe that in \eqref{IF_approx1} and \eqref{comp_xk_est_2nd1}, we use 
$\mu/{\wh a_k}$ and $\mu/{\wc a_k}$ to approximate instantaneous frequency $\phi'_k$. Thus $\xi= \frac\mu a$ represents the frequency variable.  Actually when $\xi= \frac\mu a$ and the window function $g$  is the Gaussian function defined by   
\begin{equation}
\label{def_g}
%\wh g(\xi)=e^{-2\pi^2 \xi^2},
g(t)=\frac 1{\sqrt {2\pi}} e^{-\frac {t^2}2}, 
\end{equation}
then  the adaptive CWLT $\wt W_x(a, b)$ of $x(t)$  defined by \eqref{def_CWT_para} is 
$$%\begin{eqnarray*}
%\label{def_CWT_para}
\wt W_x(\frac \mu \xi, b) = 
\frac 1{\sqrt {2\pi}} 
\int_{-\infty}^\infty x(t) \frac{\xi}\mu  \frac 1{\gs(b)} e^{-\frac {\xi^2}{\mu^2}\big(\frac{t-b}{\gs(b)}\big)^2}
e^{-i2\pi \xi (t-b)} dt. 
$$ %\end{eqnarray*}
We remark that the following transform 
$$
S_x(\xi, b):=\frac 1{\sqrt {2\pi}} 
\int_{-\infty}^\infty x(t) |\xi|  e^{-\xi^2( t-b)^2}
e^{-i2\pi \xi t} dt 
$$
is called the S-transform of $x(t)$, see \cite{S_transform1996}. Thus our adaptive CWLT $\wt W_x(\frac \mu \xi, b)$ is a (generalized) S-transform with a time-varying parameter $\gs(b)$.

Before moving on to the next section, we also remark that in practice, for a particular signal $x(t)$, its CWLT $W_x(a,b)$ and adaptive CWLT $\wt W_x(a,b)$ lie in a region of the scale-time plane:
$$
\{(a, b): \; a_1(b)\le a\le a_2(b), b\in \RR\}
$$ 
for some $0<a_1(b),  a_2(b) <\infty$. That is $W_x(a,b)$ and $\wt W_x(a,b)$ are negligible for  $(a, b)$ outside this region. Thus we need not to worry about whether $\psi_\gs$ in \eqref{wavelet_general} satisfies the so-called \lq\lq{}admissible condition\rq\rq{} or not. Throughout this paper we assume for each $b\in \RR$, the scale $a$ is in the interval:
\begin{equation}
\label{a_interval} 
a_1(b)\le a\le a_2(b). 
\end{equation}
In addition, we will use the following notations:  
\begin{equation}\label{def_M_u}
\nu(b):= \min_{1\leq k\leq K} A_{k}(b), \; M(b):=\sum_{k=1}^K A_{k}(b).
%,  \; M_\ell(b):=\sum_{k\not= \ell} A_{k}(b).  
	%B=B(t):= \sup_{0\leq k\leq K} |\phi'_k(t)|,
\end{equation}
Throughout this paper, $\sum_{k\not= \ell}$ denotes $\sum_{\{k: ~ k\not= \ell, 1\le k\le K\}}$.    
Furthermore,  by a window function $g(t)$, we mean that it is a function in $ L_2(\R)$ with certain decay at $\infty$, and    
$$
\int_\R g(t) dt =1.  
$$
%$\ga$ is defined by \eqref{def_ga_general}  for a given small $\tau_0>0$.

\section{Sinusoidal signal-based method}
%We assume
%\begin{equation}
%\label{freq_resolution_adp}d\rq{}:=\min_{k\in \{1, \cdots, K\}}\min_{t\in \R}\frac{\phi_k'(t)-\phi'_{k-1}(t)}{\phi_k'(t)+\phi'_{k-1}(t)}> 0.
%\end{equation}
%Thus $x(t)$ satisfies the well-separated condition  \eqref{freq_resolution} with resolution $=d\rq{}/2$.
%However, the value $d\rq{}$ may be very small. In this section a window function $g$ satisfies the conditions in Assumption 1. 

%For $\vep_1>0, \vep_2>0$ and $0<\gt<1$, let ${\cal C}_{\vep_1, \vep_2}$ denote the set of 

In this section we study the sinusoidal signal-based approach. 
We assume the multi-component signals of \eqref{AHM} satisfy conditions \eqref{cond_basic0}-\eqref{condition1b}. First we show that under conditions \eqref{condition1} and \eqref{condition1b}, a  multi-component signal is well-approximated by sinusoidal signals at any local time provided that $\vep_1, \vep_2$ are small. 
%For $1\le k\le K$, let ${\cal Z}_k$ be the zone in the scale-time  plane defined by
%\begin{equation}
%\label{def_Zk0}{\cal Z}_k:=\{(a, b): \; |1 - a\phi_k'(b)|< \gt\}.
%\end{equation}
%Then the well-separated condition \eqref{freq_resolution} implies that ${\cal Z}_k, 1\le k\le K$ are non-overlapping.
More precisely, write $x_k(b+a t)$ as
\begin{eqnarray*}
x_k(b+a t)\hskip -0.6cm &&=A_k(b+a t)e^{i2\pi \phi_k(b+a t)}=
A_k(b)e^{i2\pi (\phi_k(b)+\phi_k'(b) at) }+\big(x_k(b+a t)-A_k(b)e^{i2\pi (\phi_k(b)+\phi_k'(b) at) }\big). 
%\big(A_k(b+a t)-A_k(b)\big)e^{i2\pi \phi_k(b+a t)}\\
%&&\qquad +x_k(b)e^{i2\pi \phi_k'(b) at }\big(e^{i2\pi (\phi_k(b+a t)-\phi_k(b)-\phi_k'(b) at)}-1\big).
\end{eqnarray*}
Note that as a function of $t$, $A_k(b)e^{i2\pi (\phi_k(b)+\phi_k'(b) at) }$ is a sinusoidal function.  Then the adaptive CWLT  $\wt W_x(a, b)$ of $x(t)$ defined by \eqref{def_CWT_para} with $g$ can be expanded as 
\begin{eqnarray}
\nonumber
\wt W_x(a, b)\hskip -0.6cm &&=\sum_{k=1}^K \int_\R x_k(b+a t)\frac 1{\gs(b)}g\Big(\frac t{\gs(b)}\Big) e^{-i2\pi \mu t}dt\\
%\label{CWT_approx0}
\nonumber
&&=\sum_{k=1}^K \int_\R x_k(b)e^{i2\pi \phi_k'(b) at}\frac 1{\gs(b)}g\Big(\frac t{\gs(b)}\Big) e^{-i2\pi \mu t}dt +\rem_0,
\end{eqnarray}
or
\begin{equation}
\label{CWT_approx_1st}
\wt W_x(a, b)
=\sum_{k=1}^K x_k(b)  \wh g\big(\gs(b)(\mu-a \phi_k'(b))\big) +\rem_0,
\end{equation}
where $\rem_0$ is the remainder for the expansion of  $\wt W_x(a, b)$ given by
\begin{eqnarray}
%\label{def_rem0}
\nonumber \rem_0\hskip -0.6cm &&:=\sum_{k=1}^K \int_\R \Big\{x_k(b+a t)-A_k(b)e^{i2\pi (\phi_k(b)+\phi_k'(b) at) }\Big\}\frac 1{\gs(b)}g\Big(\frac t{\gs(b)}\Big) e^{-i2\pi \mu t}dt\\
\label{def_rem0} &&=\sum_{k=1}^K \int_\R \Big\{
%x_k(t)e^{i2\pi \phi_k'(b) at} \big(
 %e^{i\pi \phi''_k(\xi_{k,2}) t^2}-1\big)+A'(\xi_{k, 1})t e^{i2\pi \phi_k(b+a t)}
\big(A_k(b+a t)-A_k(b)\big)e^{i2\pi \phi_k(b+a t)}
\\
\nonumber  &&\qquad \qquad
+x_k(b)e^{i2\pi \phi_k'(b) at }\big(
e^{i2\pi (\phi_k(b+a t)-\phi_k(b)-\phi_k'(b) at)}-1\big)
\Big\}\frac 1{\gs(b)}g\Big(\frac t{\gs(b)}\Big) e^{-i2\pi \mu t}dt.
\end{eqnarray}
First we have the following lemma about the bound of $\rem_0$. 

\begin{lem}
%Let $x(t)\in {\cal C}_{\vep_1, \vep_2}$ 
Suppose $x(t)$ is a multi-component signal of \eqref{AHM} satisfying \eqref{cond_basic0}-\eqref{condition1b} for some $\vep_1>0, \vep_2>0$. Let $\wt W_x(a, b)$ be its adaptive CWLT of $x(t)$ with a window function $g$. Then 
$\wt W_x(a, b)$ can be written as \eqref{CWT_approx_1st} with 
\begin{equation}
\label{rem0_est}
|\rem_0|\le \lambda_0(a, b):=M(b)a\gs(b)\big(\vep_1  I_1  +\pi  \vep_2 I_2 a \gs(b)\big),
\end{equation}
where
\begin{equation}
\label{def_In}
I_n:=\int_\R  | t^n g(t)| dt, \;  n=1, 2, \cdots. 
\end{equation}
\end{lem}

\begin{proof}
With $|A_k(b+a t)-A_k(b)|\le  \vep_1 a |t| A_k(b)$ and
$$
%|e^{i\pi \phi''_k(\xi_{k,2}) t^2}-1|\le |\pi \phi''_k(\xi_{k,2}) t^2|
|e^{i2\pi (\phi_k(b+a t)-\phi_k(b)-\phi_k'(b) at)}-1|\le 2\pi |\phi_k(b+a t)-\phi_k(b)-\phi_k'(b) at|
\le
\pi\vep_2 a^2|t|^2,
$$
we have
\begin{eqnarray*}
|\rem_0|\hskip -0.6cm && \le \sum_{k=1}^K \int_\R  \vep_1 a | t| A_k(b) \frac 1{\gs(b)}\Big|g\Big(\frac t{\gs(b)}\Big)\Big| dt+ \sum_{k=1}^K A_k(b)\int_\R  \pi \vep_2 a^2 | t|^2 \frac 1{\gs(b)}\Big|g\Big(\frac t{\gs(b)}\Big)\Big|
 dt   \\
 &&=\vep_1 I_1 a \gs(b) \sum_{k=1}^K A_k(b) +\pi  \vep_2 I_2 a^2 \gs^2(b) \sum_{k=1}^K A_k(b)=\gl_0(a, b),
\end{eqnarray*}
as desired. 
\end{proof}

If $\vep_1$ and $\vep_2$ are small, then the remainder $\rem_0$ in \eqref{CWT_approx_1st} is negligible. Hence,  the term $x_k(b)\wh g\big(\gs(b)(\mu-a \phi_k'(b))\big)$  in \eqref{CWT_approx_1st} determines the scale-time zone of the adaptive CWLT $\wt W_{x_k}(a, b)$ of the $k$th component $x_k(t)$ of $x(t)$. More precisely, if $g$ is band-limited, %that is  $\wh g$ is compactly supported,
to say supp($\wh g)\subset [-\ga, \ga]$ for some $\ga>0$, then $x_k(b)\wh g\big(\gs(b)(\mu-a \phi_k'(b))\big)$ lies  within the zone of the scale-time plane:
 $$
Z_k:=\Big\{(a, b): |\mu - a \phi_k'(b)|< \frac {\ga}{\gs(b)}, b\in \R\Big\}.
 $$
The upper and lower boundaries of $Z_k$ are respectively
$$
\mu - a \phi_k'(b)= \frac {\ga}{\gs(b)}\; \hbox{and} \;  \mu - a \phi_k'(b)=-\frac {\ga}{\gs(b)}
$$
or equivalently
$$
a=(\mu + \frac {\ga}{\gs(b)})/ \phi_k'(b)\; \hbox{and} \;  a=(\mu - \frac {\ga}{\gs(b)})/ \phi_k'(b).
$$
Thus $Z_{k-1}$ and $Z_k$ do not overlap (with $Z_{k-1}$ lying above $Z_k$ in the time-scale plane) if
\begin{equation}
\label{separated_cond_1st0}
(\mu + \frac {\ga}{\gs(b)})/ \phi_k'(b) \le (\mu - \frac {\ga}{\gs(b)})/ \phi_{k-1}'(b),
\end{equation}
or equivalently
\begin{equation}
%\label{separated_cond_1st_compact}
\label{separated_cond_1st}
\gs(b)\ge \frac {\ga}\mu  \frac { \phi_k'(b)+\phi_{k-1}'(b)}{ \phi_k'(b)-\phi_{k-1}'(b)}, \; b\in \RR.
\end{equation}
Therefore the multi-component signal $x(t)$ is well-separated (that is $Z_k\cap Z_{\ell}=\emptyset, k\not=\ell$), provided that $\gs(b)$ satisfies \eqref{separated_cond_1st} for $k=2, \cdots, K$. 
%\begin{equation}
%%\label{separated_cond_1st_compact}
%\label{separated_cond_1st}
%\gs(b)\ge \frac {\ga}\mu  \frac { \phi_k'(b)+\phi_{k-1}'(b)}{ \phi_k'(b)-\phi_{k-1}'(b)}, \; b\in \R, k=2, \cdots, K.
%\end{equation}
Observe that our well-separated condition \eqref{separated_cond_1st} is different from that in \eqref{freq_resolution} considered in \cite{Daub_Lu_Wu11}.
%\cite{Wu_thesis} and \cite{MOM14}.

If $\wh g$ is not compactly supported,
we consider the ``support'' of  $\wh g$ outside which $\wh g(\xi)\approx 0$.
More precisely, for a given small positive threshold $\tau_0$, if
%$|\wh g(\xi)|/\max_{\xi}|\wh g(\xi)|\le \tau_0$ 
$|\wh g(\xi)|\le \tau_0$
for $|\xi|\ge \ga$ for some $\ga>0$, then we say $\wh g(\xi)$ is {\it essentially supported} in $[-\ga, \ga]$.
When $|\wh g(\xi)|$ is even and (strictly) decreasing for $\xi\ge 0$,   
then $\ga$ is obtained by solving
\begin{equation}
\label{def_ga_general}
|\wh g(\ga)|=\tau_0.
\end{equation}
For example, when  $g$ is the Gaussian function defined by \eqref{def_g},  
then, with $\wh g(\xi)=e^{-2\pi^2 \xi^2}$,  the corresponding $\ga$ is given by
\begin{equation}
\label{def_ga}
\ga=\frac 1{2\pi}\sqrt{2\ln (1/{\tau_0 })} .
\end{equation}

When $\wh g$ is not compactly supported, let $\ga$ be the number defined by \eqref{def_ga_general}, namely assume  $\wh g(\xi)$ is {\it essentially supported} in $[-\ga, \ga]$.
Then $x_k(b)\wh g\big(\gs(b)(\mu-a \phi_k'(b))\big)$
 lies within the scale-time zone $Z_k$ defined by %(refer to \eqref{f_zone})
\begin{equation}
\label{def_Zk}
Z_k:=\Big\{(a, b): |\wh g\big(\gs(b)(\mu-a \phi_k'(b))\big) |>\tau_0, b\in \R\Big\}=
\Big\{(a, b): |\mu -a \phi_k'(b)|< \frac {\ga}{\gs(b)}, b\in \R\Big\}.
\end{equation}
Thus if the remainder $\rem_0$ in \eqref{CWT_approx_1st} is small,
$\wt W_{x_k}(a, b)$  {\it essentially lies} within $Z_k$ and hence,
the multi-component signal $x(t)$ is well-separated provided that  $\gs(b)$ satisfies \eqref{separated_cond_1st} for $2\le k\le K$.
When \eqref{freq_resolution} holds, a simple choice of $\gs(b)\equiv \frac \ga{\mu \gt}$ satisfies \eqref{separated_cond_1st}. % holds for some $\gs(b)$.
Refer to \cite{LCHJJ18} for other choices of  $\gs(b)$.  In this section we assume $\gs(b)$ is such a function that  \eqref{separated_cond_1st} holds. 
In the following, for $\vep_1>0, \vep_2>0$ and $\ga>0$, we let ${\cal C}_{\vep_1, \vep_2}$ denote the set of 
 the multi-component signals of \eqref{AHM} satisfying \eqref{cond_basic0}-\eqref{condition1b} and \eqref{separated_cond_1st} for some $\gs(b )>0$.

Here we remark that in practice $\phi'_k(t), 1\le k\le K$ are unknown.  
However both the condition in \eqref{freq_resolution} considered in the seminal paper \cite{Daub_Lu_Wu11} on SST and that in \eqref{separated_cond_1st} involve $\phi'_k(t)$. Like paper \cite{Daub_Lu_Wu11}, the main goal of our paper is to establish theoretical theorems which guarantee the recovery of components, namely, we provide conditions under which the components can be recovered. These conditions involve some properties 
of $x_k(t)$ including $\phi'_k(t)$ and even $\phi''_k(t)$ in the next section. 

%\bigskip 

From \eqref{rem0_est}, we have that
%for $(a, b)\in Z_k$,
\begin{equation}
\label{rem0_est_k}
{|\rem_0(a, b)|}\le \Lambda_k(b)M(b),  \; \hbox{for $(a, b)\in Z_k$}, 
\end{equation}
where 
\begin{equation}
\label{def_gLk}
\gL_k(b):=\vep_1  I_1 \frac{\mu \gs(b)+\ga}{\phi'_k(b)} +\pi  \vep_2 I_2
\Big(\frac{\mu \gs(b)+\ga}{\phi'_k(b)}\Big)^2. 
\end{equation}
Recall that we assume that the scale variable $a$ lies in the interval \eqref{a_interval}. 
Throughout this section, we may simply assume that
\begin{equation}
\label{def_a2}
a_1(b)=\frac{\mu-\ga/\gs(b)}{\phi_K\rq{}(b)}\le a\le a_2(b)=
\frac{\mu+\ga/\gs(b)}{\phi_1\rq{}(b)}.
\end{equation}
Hence we have for $a\in [a_1, a_2]$, 
\begin{equation}
\label{rem0_est_all}
{|\rem_0(a, b)|}\le M(b)\Lambda_1(b). 
\end{equation}

%\bigskip 
From \eqref{separated_cond_1st0}, we have 
$$
\frac {\phi'_k(b)}{\phi'_\ell(b)}\ge \Big(\frac {\mu\gs(b)+\ga}{\mu\gs(b)-\ga}\Big)^{k-\ell}, \; \hbox{if $\ell<k$}, 
$$
and 
$$
\frac {\phi'_k(b)}{\phi'_\ell(b)}\le \Big(\frac {\mu\gs(b)+\ga}{\mu\gs(b)-\ga}\Big)^{\ell-k}, \; \hbox{if $\ell>k$}.  
$$
Thus for $(a, b)\in Z_\ell$, that is $\mu \gs(b)-\ga < \gs(b) a\phi'_\ell(b) < \mu \gs(b)+\ga$, we have 
\begin{equation}
\label{ineq_Z_ell}
\begin{array}{l}
\gs(b)\big(a \phi'_k(b) -\mu\big)>  \frac {(\mu\gs(b)+\ga)^{k-\ell}}{(\mu\gs(b)-\ga)^{k-\ell-1}}-\mu \gs(b) \; \hbox{for $\ell<k$}, \\
\gs(b)\big(\mu-a \phi'_k(b)\big)> \mu\gs(b)- \frac {(\mu\gs(b)-\ga)^{\ell-k}}{(\mu\gs(b)+\ga)^{\ell-k-1}} \; \hbox{for $\ell>k$}. 
\end{array}
\end{equation}
Denote 
$$
\gamma_{\ell, k}(b):=\begin{cases}
\frac {(\mu\gs(b)+\ga)^{k-\ell}}{(\mu\gs(b)-\ga)^{k-\ell-1}}-\mu \gs(b) \; \hbox{if $\ell<k$}, \\
\mu\gs(b)- \frac {(\mu\gs(b)-\ga)^{\ell-k}}{(\mu\gs(b)+\ga)^{\ell-k-1}} \; \hbox{if $\ell>k$}. 
\end{cases}
$$
Then by \eqref{ineq_Z_ell}, we have 
\begin{equation}
\label{ineq_Z_ell_gga_ellk}
\big|\gs(b)(\mu -a \phi'_k(b))\big|> \gga_{\ell,  k}(b) \; \hbox{for any $(a, b)\in Z_\ell, k\not=\ell$}.   
\end{equation}

%Let $f(\xi)$ be the function defined by \eqref{def_f}. 

In this section we assume the window function $g$ satisfies the following condition. 

{\bf Assumption 1.} \;  {\it $|\wh g(\xi)|$ can be written as  
\begin{equation}
\label{def_f}
|\wh g(\xi)|=f(|\xi|), 
\end{equation}
where $f(\xi)$ is a positive and {\rm (}strictly{\rm )} decreasing function on $\xi\ge 0$. }

Denote 
 \begin{equation}
 \label{def_err}
 \terr_\ell(b):=M(b)\gL_\ell(b)+\sum_{k\ne \ell}A_k(b) f(\gga_{\ell, k}), 
 \end{equation}
where, as mentioned in Section 1, $\sum_{k\not= \ell}$ denotes $\sum_{\{k: ~ k\not= \ell, 1\le k\le K\}}$. 

\begin{lem}
Let $x(t)\in {\cal C}_{\vep_1, \vep_2}$ and $\wt W_x(a, b)$ be its adaptive CWLT with a window function $g$ satisfying Assumption 1. Then for any $(a, b)\in Z_\ell$, 
\begin{equation}
\label{est_xk_CWT}
 \big|\wt W_x(a, b)-x_\ell(b) \wh g\big(\gs(b)(\mu -a \phi'_\ell(b)) \big)\big|\le \terr_\ell(b). 
\end{equation}
 \end{lem}
 
 \begin{proof}
 For any $(a, b)\in Z_\ell$, from \eqref{CWT_approx_1st}, we have 
\begin{eqnarray*}
\nonumber && \big|\wt W_x(a, b)-x_\ell(b) \wh g\big(\gs(b)(\mu -a \phi'_\ell(b)) \big)|\\
&&=\Big|\sum_{k\ne \ell} x_k(b)  \wh g\big(\gs(b)(\mu -a \phi'_k(b))\big) +\rem_0\Big|\\
&&\le \big|\rem_0\big|+\sum_{k\ne \ell} A_k(b)  f\big(|\gs(b)(\mu -a \phi'_k(b))\big|\big) \\
&&\le  M(b)\gL_\ell(b)+\sum_{k\ne \ell} A_k(b) f(\gga_{\ell, k}), 
\end{eqnarray*}
where the last inequality follows from \eqref{rem0_est_k}, \eqref{ineq_Z_ell_gga_ellk} and the assumption that $f(\xi)$ is  decreasing for $\xi\ge 0$. 
\end{proof}

Note that $\gga_{\ell, k}\ge \ga$, where $\ga$ is defined by \eqref{def_ga_general}, that is $f(\ga)=\tau_0$. Thus we have that 
 \begin{equation}
\label{ineq_Z_ell_gga_ellk_g}
\big|\wh g\big(\gs(b)(\mu -a \phi'_k(b))\big)\big|\le  f(\gga_{\ell, k}) \le f(\ga)= \tau_0 \; \hbox{for any $(a, b)\in Z_\ell, k\not=\ell$};  
\end{equation}
and hence, 
\begin{equation}
 \label{err_ineq}
 \terr_\ell(b)\le M(b)\gL_\ell(b)+\tau_0 \sum_{k\ne \ell}A_k(b).
 \end{equation}

For a fixed $b$ and a positive $\wt \ep_1$ (possibly depending on $b$), 
we let $\cG_b$ and $\cG_{b, k}$ denote the sets defined by 
\begin{equation}
\label{def_cGk}
\cG_b:=\{a\in [a_1, a_2]: \; |\wt W_x(a, b)|>\wt \ep_1\}, 
\; \cG_{b, k}:=\{a \in \cG_b: \; |\mu -a \phi_k'(b)|< \frac {\ga}{\gs(b)}, b\in \R\}.  
\end{equation}
Note that $\cG_b$ and $\cG_{b, k}$ depend on $\wt \ep_1$, and for simplicity of presentation, we drop $\wt \ep_1$ from them. 

Let $\wh a_\ell$ be defined by \eqref{def_max_eta}. Observe that $\cG_{b, k}=\cG_b\cap \{a: \; (a, b)\in Z_ k\}$. Thus $\cG_{b, k}, 1\le k\le K$ are disjoint since 
 $Z_k$ are not overlapping. In addition, we will show in the next theorem that each $\cG_{b, k}$ is non-empty. 
 Thus the definition for $\wh a_\ell$ in \eqref{def_max_eta} makes sense.

Next we present our analysis results on adaptive CWLT-based IF estimation and component recovery derived from sinusoidal signal local approximation. % in Theorem \ref{theo:main_1st} below. 
\begin{theo}
\label{theo:main_1st} Let $x(t)\in {\cal C}_{\vep_1, \vep_2}$ for some $\vep_1, \vep_2>0$, $g$ be a window function satisfying Assumption 1 and  $\gs(b)>0$ be a function satisfying \eqref{separated_cond_1st}. 
Suppose $\vep_1, \vep_2, \tau_0$ are small enough such that 
\begin{equation}
\label{theo1_cond1}
2M(b)\big(\tau_0+\gL_1(b)\big)\le \nu(b). 
\end{equation} 
Let $\wt \ep_1=\wt \ep_1(b)>0$ be a function of $b$ satisfying   
\begin{equation}
\label{cond_ep1}
M(b)\big(\tau_0+\gL_1(b)\big)\le 
\wt \ep_1 \le \nu(b)-M(b)\big(\tau_0+\gL_1(b)\big).
\end{equation}
Then the following statements hold.
\begin{enumerate}
\item[{\rm (a)}] Let $\cG_b$ and $\cG_{b, k}$ be the sets defined by \eqref{def_cGk} for some $\wt \ep_1$ satisfying \eqref{cond_ep1}.   Then $\cG_b$ can be expressed as a disjoint union of exactly $K$ non-empty sets $\cG_{b, k}, 1\le k\le K$.

\item[{\rm (b)}] Let $\wh a_\ell$ be defined by \eqref{def_max_eta}.  Then for $\ell=1, 2, \cdots, K$, 
\begin{equation}\label{phi_est}
|\mu-\wh a_\ell(b)\phi_{\ell}'(b)|\le \bd_{1, \ell}:=\frac{1}{\gs(b)} f^{-1}\big(1-2 \; \terr_\ell(b)/A_\ell(b)\big), 
\end{equation}
where $\terr_\ell$ is defined by \eqref{def_err}. 
\item[{\rm (c)}]  For $\ell=1, 2, \cdots, K$,
\begin{equation}
\label{comp_xk_est}
\big|\wt W_x(\wh a_\ell, b)-x_\ell(b)\big|
\le\bd_{2, \ell}:= \terr_\ell(b)+2\pi I_1 A_\ell(b) f^{-1}\big(1-2 \; \terr_\ell(b)/A_\ell(b)\big). 
\end{equation}
\item[{\rm (d)}] 
 %If in addition the window function $g(t)\ge 0$ for $t\in \R$, then 
 For $\ell=1, 2, \cdots, K$,
\begin{equation}
\label{abs_IA_est}
\big| |\wt W_x(\wh a_\ell, b)|-A_{\ell}(b) \big|\le \terr_\ell(b). 
\end{equation} 
\end{enumerate}
\end{theo}

%\bigskip 

\begin{mrem}\label{rem1}
Here we remark that $\terr_\ell(b)< \frac 12 A_\ell(b)$ if   $2M(b)\big(\tau_0+\gL_\ell(b)\big)\le  \nu(b)$, and hence, $f^{-1}\big(1-2 \; \terr_\ell(b)/A_\ell(b)\big)$ is well defined. Indeed, from \eqref{err_ineq}, 
\begin{eqnarray*}
 \terr_\ell(b)\hskip -0.6cm&& \le M(b)\gL_\ell(b)+\tau_0 \sum_{k\ne \ell}A_k(b)
\le \frac 12 \nu(b)-M(b)\tau_0+\tau_0 \sum_{k\ne \ell}A_k(b)\\
&& < \frac 1 2 A_\ell(b)-M(b)\tau_0+\tau_0 \sum_{k=1}^KA_k(b)= \frac 1 2 A_\ell(b). 
\end{eqnarray*}
\hfill $\square$ 
\end{mrem}

From \eqref{comp_xk_est}, we know the  recovery formula for a complex signal is   \eqref{comp_xk_est0}, while  for a real-valued $x(t)$, the  recovery formula for component $x_\ell(t)$ will be  
\begin{equation}
\label{comp_xk_est_real}
x_\ell(t) \approx 2{\rm Re}\Big(\wt W_x(\wh a_\ell, b)\Big). 
\end{equation}

When $\wh g(\xi)$ is supported in $[-\ga, \ga]$, we can set $\tau_0$ in Theorem \ref{theo:main_1st} to be zero.  Thus  the condition in \eqref{theo1_cond1} is reduced to  $2M(b)\gL_1(b)\le \nu(b)$. In addition, the error $\terr_\ell(b)$ in \eqref{def_err} is simply $M(b)\gL_\ell(b)$.  To summarize, we have the following corollary, where 
$f$ is the decreasing function on $[0, \ga]$ with $f(\xi)=\wh g(|\xi|)$.   

\begin{cor}
\label{cor:main_1st} Let $x(t)\in {\cal C}_{\vep_1, \vep_2}$ for some $\vep_1, \vep_2>0$ and 
$g$ be a window function satisfying Assumption 1 and that supp($\wh g)\subseteq [-\ga, \ga]$. 
% in the Schwartz class $\cS$ with $g(0)\not=0$.  
Suppose $\gs(b)>0$ satisfies \eqref{separated_cond_1st}. 
If $\vep_1, \vep_2$ are small enough such that $2M(b)\gL_1(b)\le \nu(b)$, then we have the following statements.
\begin{enumerate}
\item[{\rm (a)}]  Let $\cG_b$ and $\cG_{b, k}$ be the sets defined by \eqref{def_cGk} for some $\wt \ep_1$ satisfying  $M(b)\gL_1(b)\le 
\wt \ep_1 \le \nu(b)-M(b)\gL_1(b)$.    Then $\cG_b$ can be expressed as a disjoint union of exactly $K$ non-empty sets $\cG_{b, k}, 1\le k\le K$. 

\item[{\rm (b)}] Let $\wh a_\ell$ be defined by \eqref{def_max_eta}.  Then for $\ell=1, 2, \cdots, K$,
\begin{equation}\label{cor_phi_est}
|\mu-\wh a_\ell(b)\phi'_\ell(b)|\le \frac{1}{\gs(b)} f^{-1}\big(1-\frac {2 \; M(b)\gL_\ell(b)}{A_\ell(b)}\big).  
\end{equation}
\item[{\rm (c)}] For $\ell=1, 2, \cdots, K$,
\begin{equation}
\label{cor_comp_xk_est}
\big|\wt W_x(\wh a_\ell, b)-x_\ell(b)\big|
\le M(b)\gL_\ell(b)+2\pi I_1 A_\ell(b) f^{-1}\big(1-\frac {2 \; M(b)\gL_\ell(b)}{A_\ell(b)}\big). 
\end{equation}
\item[{\rm (d)}] 
%If in addition the window function $g(t)\ge 0$ for $t\in \R$, then 
For $\ell=1, 2, \cdots, K$,
\begin{equation}
\label{cor_abs_IA_est}
\big| |\wt W_x(\wh a_\ell, b)|-A_{\ell}(b) \big|\le M(b)\gL_\ell(b). 
\end{equation} 
\end{enumerate}
\end{cor}

\bigskip 

\begin{example}\label{example1}
Let $g(t)$ be the Gaussian window function given in \eqref{def_g}. Then $\wh g(\xi)=e^{-2\pi^2 \xi^2}$.  Hence $f(\xi)=e^{-2\pi^2 \xi^2}$. 
Thus the terms $f(\gga_{\ell, k})$ in $\terr_\ell(b)$ defined by \eqref{def_err} are 
$$
e^{-2\pi^2 \gga_{\ell, k}^2}, 
$$
which are very small if $\ga\ge 1$. For this $g$, we have 
$$
f^{-1}(\xi)=\wh g(\xi)^{-1}= \frac1{\pi \sqrt 2} \sqrt{-\ln \xi}, \;  0<\xi <1. 
$$
Hence the error bound $\bd_{1, \ell}$ in \eqref{phi_est} is 
$$
\bd_{1, \ell}=\frac{1}{\gs(b)} f^{-1}\big(1-2 \; \terr_\ell(b)/A_\ell(b)\big)=
\frac{1}{\gs(b)\pi \sqrt 2} \sqrt{-\ln \big(1-2 \; \terr_\ell(b)/A_\ell(b)\big)}. 
$$
Assume $\vep_1, \vep_2, \tau_0$ are small enough such that $2 \; \terr_\ell(b)/A_\ell(b)\le c_0$ for some $0<c_0<1$. 
Then using the fact $-\ln(1-t)< \frac 1{1-c_0}t$ for $0<t\le c_0$, we have 
\begin{equation*}
\bd_{1, \ell} < \frac1{\gs(b)\pi  \sqrt{1-c_0}} \sqrt{{\terr_\ell(b)}/{A_\ell(b)}}.  
\end{equation*}

In this case the error bound $\bd_{2, \ell}$ in \eqref{comp_xk_est} for component recovery satisfies 
$$
\bd_{2, \ell}<\terr_\ell(b)+\frac{2I_1}{\sqrt{1-c_0}} \sqrt{A_\ell(b)\terr_\ell(b)}. 
$$
We see a  smaller  $\gs(b)$ results in a smaller error bound $\bd_{2, \ell}$ for component recovery. 
\hfill $\blacksquare$
\end{example}

\bigskip

\bigskip

Next we give the proof of Theorem \ref{theo:main_1st}. 
%Here we consider the case $g$ is non-bandlimited. The proof of Theorem \ref{theo:main_1st} when $g$ is bandlimited, which is Corollary \ref{cor:main_1st}, is the same but simpler. 

%\bigskip 

{\bf Proof  of  Theorem \ref{theo:main_1st}(a)}.  Clearly $\cup _{k=1}^K \cG_{b, k}\subseteq \cG_b$. 
Next we show  $\cG_b\subseteq \cup _{k=1}^K \cG_{b, k}$.  Let $a \in \cG_b$. 
Assume $a\not \in \cup _{k=1}^K \cG_{b, k}$. That is $(a, b)\not \in \cup _{k=1}^K Z_k$.  
Then by the definition of $Z_k$ in \eqref{def_Zk}, we have  
$|\wh g\big(\gs(b)(\mu-a\phi'_k(t))\big)|\le \tau_0$. Hence, by \eqref{CWT_approx_1st} and \eqref{rem0_est_all}, we have 
\begin{eqnarray*}
|\wt W_x(a, b)|\hskip -0.6cm &&\le \sum_{k=1}^K \big|x_k(b) \wh g\big(\gs(b)(\mu-a\phi'_k(b))\big)\big| + |\rem_0|
\\
&&\le \tau _0 \sum_{k=1}^K A_k(b)  +M(b)\gL_1(b)\le \wt \ep_1, 
\end{eqnarray*}
a contradiction to the assumption $|\wt W_x(a, b)|>\wt \ep_1$. Thus $(a, b)\in Z_\ell$ for some $\ell$. 
This shows that $a \in \cG_{b, \ell}$. Hence $\cG_b=\cup _{k=1}^K \cG_{b, k}$. 
Since $Z_k, 1\le k\le K$ are not overlapping, we know  $\cG_{b, k}, 1\le k\le K$ are disjoint.  

To show that $\cG_{b, \ell}$ is non-empty, 
it is enough to show $\frac\mu{\phi_\ell\rq{}(b)}\in \cG_b$ which implies $\frac\mu{\phi_\ell\rq{}(b)}\in \cG_{b, \ell}$ since 
$(\frac\mu{\phi_\ell\rq{}(b)}, b)\in Z_\ell$.   From \eqref{est_xk_CWT} with $a=\frac \mu{\phi'_\ell(b)}$, we have 
\begin{equation}
\label{est_xk_CWT1}
\big|\wt W_x(\frac \mu{\phi'_\ell(b)}, b)\big|\ge |x_\ell(b) \wh g(0)|-\terr_\ell(b) =A_\ell(b) - \terr_\ell(b),   
\end{equation}
since $\wh g(0)=1$.  This and \eqref{err_ineq} lead to 
\begin{eqnarray*}
|\wt W_x(\frac\mu{\phi_\ell\rq{}(b)}, b)|\hskip -0.6cm 
&&\ge  A_\ell(b) -M(b)\gL_\ell(b)-\tau_0 \sum_{k\ne \ell}A_k(b)\\
&& >  A_\ell(b)   -M(b)\gL_\ell(b)- M(b)\tau_0\\
&& \ge  \nu(b) - M(b)\big(\tau_0 +\gL_1(b)\big) \ge \wt \ep_1.  
\end{eqnarray*}
Thus $\frac\mu{\phi_\ell\rq{}(b)}\in \cG_b$. This completes the proof of  Theorem \ref{theo:main_1st}(a). 
\hfill $\square$

%\bigskip 

{\bf Proof  of  Theorem \ref{theo:main_1st}(b)}. 
By the definition of $\wh a_\ell$ and \eqref{est_xk_CWT}, we have 
$$
\big|\wt W_x(\frac \mu{\phi'_\ell(b)}, b)\big| \le \big|\wt W_x(\wh a_\ell, b)\big| \le 
\big|x_\ell(b) \wh g\big(\gs(b)(\mu -\wh a_\ell \; \phi'_\ell(b)) \big)\big|+ \terr_\ell(b)
$$
This, together with \eqref{est_xk_CWT1}, implies 
$$
A_\ell(b) - \terr_\ell(b)\le A_\ell(b) \big|\wh g\big(\gs(b)(\mu -\wh a_\ell \; \phi'_\ell(b)) \big)\big|+ \terr_\ell(b), 
$$
or equivalently
$$
0<1- 2 \; \terr_\ell(b)/A_\ell(b)\le  f\big(|\gs(b)(\mu -\wh a_\ell \; \phi'_\ell(b)) |\big).  
$$
Then \eqref{phi_est} follows from the above inequality and that $f(\xi)$ is decreasing on $(0, \infty)$. 
\hfill $\square$

%\bigskip 

{\bf Proof  of  Theorem \ref{theo:main_1st}(c)}.
From \eqref{est_xk_CWT}, we have 
\begin{eqnarray*}
&&\big|\wt W_x(\wh a_\ell, b)-x_\ell(b)\big|\\
&&\le \big|\wt W_x(\wh a_\ell, b)-x_\ell(b) \wh g\big(\gs(b)(\mu -\wh a_\ell \; \phi'_\ell(b)) \big)\big|
+\big |x_\ell(b) \wh g\big(\gs(b)(\mu -\wh a_\ell \; \phi'_\ell(b)) \big)-x_\ell(b)\big|
\\
&&\le \terr_\ell(b)+A_\ell(b) \Big| \int_\R g(t)\Big(e^{-i2\pi \gs(b)(\mu -\wh a_\ell \; \phi'_\ell(b) )t }-1\Big) dt \Big|\\
&&\le \terr_\ell(b)+A_\ell(b) \int_\R |g(t)| \; 2\pi \gs(b)\big|\mu -\wh a_\ell \; \phi'_\ell(b)\big | |t| dt\\
&&\le \terr_\ell(b)+A_\ell(b) 2\pi f^{-1}\big(1-2 \; \terr_\ell(b)/A_\ell(b)\big) \int_\R |g(t)|  |t| dt\\
&&=  \terr_\ell(b)+2\pi I_1 A_\ell(b) f^{-1}\big(1-2 \; \terr_\ell(b)/A_\ell(b)\big). 
\end{eqnarray*}
This shows \eqref{comp_xk_est}. 
\hfill $\square$
%\bigskip 

{\bf Proof of Theorem \ref{theo:main_1st}(d)}. By the definition of $\wh a_\ell$ and \eqref{est_xk_CWT1}, we have 
\begin{equation}
\label{est_xk_IA1}
\big|\wt W_x(\wh a_\ell, b)\big|\ge 
\big|\wt W_x(\frac \mu{\phi'_\ell(b)}, b)\big|\ge A_\ell(b) - \terr_\ell(b). 
\end{equation}
On the other hand, by Assumption 1, we have $|\wh g(\xi)|\le |\wh g(0)|\le 1$ for any $\xi\in \R$. This fact and 
\eqref{est_xk_CWT} imply 
 \begin{equation}
\label{est_xk_IA2}
\big|\wt W_x(\wh a_\ell, b)\big| \le 
\big|x_\ell(b) \wh g\big(\gs(b)(\mu -\wh a_\ell \; \phi'_\ell(b)) \big)\big|+ \terr_\ell(b)\le A_\ell(b) +\terr_\ell(b).
\end{equation}
\eqref{abs_IA_est} follows from \eqref{est_xk_IA1} and \eqref{est_xk_IA2}. This completes the proof  of  Theorem \ref{theo:main_1st}(d). 
\hfill $\square$

\section{Linear chirp-based method}
A signal of the form 
\begin{equation}
\label{def_chirp_At}
s(t)=A e^{i2\pi \phi(t)}=A e^{i2\pi (ct +\frac 12 r t^2)} 
\end{equation}
for some $c>0$ and non-zero number $r$ is called a linear chirp or linear frequency modulation (LFM) signal.   In this section we consider multi-component signals $x(t)$ of \eqref{AHM} with $A_k(t)$ satisfying \eqref{cond_basic0} and \eqref{condition1} for some $\vep_1>0$, 
 and $\phi_k(t)$ satisfying \eqref{cond_basic},  \eqref{cond_basic_2nd} and \eqref{condition2} for some $\vep_3>0$.

As shown below, conditions \eqref{condition1} and \eqref{condition2} imply that when $\vep_1, \vep_3$ are small,  each component $x_k(t)$ is well approximated locally by linear chirps of \eqref{def_chirp_At}. 
Indeed, write $x(b+at)$ as
\begin{equation}
\label{xm_xr}
x(b+a t)=x_{\rm m}(a,b,t)+x_{\rm r}(a,b,t),
\end{equation}
where
\begin{eqnarray*}
%\label{xm}
&&x_{\rm m}(a,b,t):=\sum_{k=1}^K x_k(b)e^{i2\pi (\phi_k'(b) at+\frac 12\phi''_k(b) (at)^2) }, \\
%\label{xr}
&&x_{\rm r}(a,b,t):=\sum_{k=1}^K \big(x_k(b+a t)-x_k(b)e^{i2\pi (\phi_k'(b) at+\frac 12\phi''_k(b) (at)^2) }\big)\; .
\end{eqnarray*}
%Next we consider the approximation of $\wt W_x(a, b)$ when $x(b+at)$ is approximated by $x_{\rm m}(a,b,t)$.
With \eqref{xm_xr}, we have
\begin{eqnarray}
\nonumber
\wt W_x(a, b)\hskip -0.6cm &&=\sum_{k=1}^K \int_\R x_k(b+a t)\frac 1{\gs(b)}g\Big(\frac t{\gs(b)}\Big) e^{-i2\pi \mu t}dt\\
%\label{CWT_approx0}
\nonumber 
&&=\sum_{k=1}^K \int_\R x_k(b)e^{i2\pi (\phi_k'(b) a t +\frac12\phi''_k(b) a^2t^2)}\frac 1{\gs(b)}g\Big(\frac t{\gs(b)}\Big) e^{-i2\pi \mu t}dt +\err_0\\
&&\label{CWT_approx}
=\sum_{k=1}^K x_k(b)  %G_{0, k}(a, b)
G_k\big(\gs(b)(\mu-a \phi_k'(b) ), a, b\big)+\err_0, 
\end{eqnarray}
where $G_k(\xi, a, b)$ is defined by \eqref{def_Gk_general}, and 
\begin{eqnarray}
\label{def_err0}
&&\err_0:= \int_\R x_{\rm r}(a,b,t)\frac 1{\gs(b)}g\Big(\frac t{\gs(b)}\Big) e^{-i2\pi \mu t}dt.
\end{eqnarray}

Next lemma provides an upper bound for $\err_0$. 

\begin{lem}\label{lem:basic}
Let $x(t)$ be a signal of \eqref{AHM} with $A_k(t)$ and $\phi_k(t)$ satisfying \eqref{condition1} and \eqref{condition2} respectively. 
 Then the  adaptive CWLT $\wt W_x(a, b)$ of $x(t)$ with a window function $g$  
 can be written as \eqref{CWT_approx} with 
 \begin{equation}
\label{err0_est}
|\err_0|\le  M(b)\Pi(a, b),
\end{equation}
where
$$
\Pi(a, b):=\vep_1  I_1 a \gs(b)+\frac \pi 3 \vep_3 I_3 a^3 \gs^3(b),  
$$
with $I_n$ defined by \eqref{def_In}. 
\end{lem}

\begin{proof} 
By condition \eqref{condition2}, we have 
$$
|e^{i2\pi (\phi_k(b+a t)-\phi_k(b)-\phi_k'(b) at- \frac 12\phi''_k(b) (at)^2)}-1|\le
2\pi \frac 16 \sup_{\eta \in \R}|\phi^{(3)}_k(\eta) (a t)^3|
\le \frac \pi 3 \vep_3 a^3 |t|^3.
$$
This, together with $|A_k(b+a t)-A_k(b)|\le  \vep_1 a |t| A_k(b)$, leads to  that 
\begin{eqnarray*}
%\label{xr_est}
|x_{\rm r}(a,b,t)|\hskip -0.6cm && =\Big|\sum_{k=1}^K \Big\{ (A_k(b+a t)-A_k(b))e^{i2\pi \phi_k(b+a t)}\\
\nonumber &&
+x_k(b)e^{i2\pi (\phi_k'(b) at+\frac 12\phi''_k(b)(at)^2) } \big(
e^{i2\pi (\phi_k(b+a t)-\phi_k(b)-\phi_k'(b) at- \frac 12\phi''_k(b) (at)^2)}-1\big)\Big\}\Big|\\
&&\le M(b)\big(\vep_1 a|t|+\frac \pi 3 \vep_3 a^3 |t|^3\big). 
\end{eqnarray*}
Thus, %By \eqref{xr_est}, we have the following estimate for $\err_0$:
\begin{eqnarray*}
|\err_0|\hskip -0.6cm&&\le   M(b)\int_\R  \big(\vep_1 a | t|+
 \frac \pi3 \vep_3 a^3  | t|^3\big) \frac 1{\gs(b)}\Big|g\Big(\frac t{\gs(b)}\Big)\Big|
 dt \\
 &&=M(b)\big(\vep_1 I_1 a \gs(b)  +\frac \pi 3 \vep_3 I_3 a^3 \gs^3(b)\big)=M(b)\Pi(a, b). 
\end{eqnarray*}
This completes the proof of Lemma \ref{lem:basic}. 
\end{proof}

Note that as a function of $t$, $x_k(b)e^{i2\pi (\phi_k'(b) at+\frac 12\phi''_k(b) (at)^2) }$ is a linear chirp. Hence $x_{\rm m}(a, b, t)$ is a linear combination of linear chirps with variable $t$, and it approximates $x(b+at)$ well, provided that  $\vep_1, \vep_3$ are small.

By \eqref{err0_est}, we know $|\err_0|$ is small  if $\vep_1, \vep_3$ are small enough.  Hence, in this case  $G_k\big(\gs(b)(\mu-a \phi_k'(b)), a, b\big)$ determines the scale-time zone for $\wt W_{x_k}(a, b)$. 
% In the following we  describe those  scale-time zones mathematically.
More precisely, let $0<\tau_0<1$ be a given small number as the threshold. % with which for a function $H(\xi)$
Denote
\begin{equation*}
%\label{def_Ok}
O'_k:=%\{(a, b): |G_{0,k}(a, b)|>\tau_0, b\in \R\}/\max_u |G_k(u)|=
\big\{(a, b): |G_k\big(\gs(b)(\mu-a\phi_k'(b)), a, b\big)|%/\max_u |G_k(u)|
>\tau_0, b\in \R\big\}.
\end{equation*}
If for each fixed $a$ and $b$,  $|G_k(\xi, a , b)|$ is even in $\xi$ and decreasing for $\xi\ge 0$, % for some $\xi_0>0$.
then $O'_k$ can be written as
\begin{equation}
\label{def_Ok2_old}
O'_k=\Big\{(a, b): |\mu -a \phi_k'(b)|< \frac {\ga_k}{\gs(b)}, b\in \R\Big\}.
\end{equation}
where $\ga_k$ is obtained by solving $|G_k(\xi, a, b)|%/\max_u |G_k(u)|
=\tau_0$ for $\xi$.  In general $\ga_k=\ga_k(a,b)$ depends on both $b$ and $a$, and it is hard to obtain the explicit expressions for the boundaries of $O'_k$.
As suggested in \cite{LCHJJ18}, in this paper, we assume $\ga_k(a,b)$ can be replaced by $\gb_k(a,b)$ with
$\ga_k(a,b)\le \gb_k(a,b)$ such that $O'_k$ defined by \eqref{def_Ok2_old} with  $\ga_k=\gb_k(a,b)$ can be written as
\begin{equation}
\label{def_Ok2}
O_k:=\big\{(a, b): l_k(b)<a<u_k(b), b\in \R\big\},
\end{equation}
for some $0<l_k(b)<u_k(b)$, and 
\begin{equation}
\label{def_Ok2_ineq}
|G_k\big(\gs(b)(\mu-a\phi_k'(b)), a, b\big)|\le \tau_0, \; \hbox{for $(a, b)\not\in O_k$}. 
\end{equation} 
In addition, we will assume the multi-component signal $x(t)$ is well-separated, that is there is $\gs(b)$ such that $u_{k-1}(b) \le l_k(b), b\in \RR, k=2, \cdots, K$, or equivalently
\begin{equation}
\label{cond_no_overlapping}
O_k\cap O_{\ell}=\emptyset, \quad k\not=\ell.
\end{equation}
Clearly, from \eqref{err0_est}, 
we have that 
\begin{equation}
\label{rem0_est_k_2nd}
{|\err_0(a, b)|}\le M(b)\Pi_k(b),  \; \hbox{for $(a, b)\in O_k$}, 
\end{equation}
where 
\begin{equation}
\label{def_gLk_2nd}
\Pi_k(b):=\Pi(u_k(b), b)=\vep_1  I_1 u_k (b)\gs(b)+\frac \pi 3 \vep_3 I_3 u_k^3(b) \gs^3(b). 
\end{equation}

%\bigskip 

Recall that we assume that the scale variable $a$ lies in the interval \eqref{a_interval}. Of course we assume $a_1(b), a_2(b)$ satisfy 
$$
a_1(b)\le l_K(b), \; u_1(b)\le a_2(b). 
$$
Throughout this section, we simply let %assume that
\begin{equation}
\label{def_a2_2nd}
a_1(b)=l_K(b), \; a_2(b)=u_1(b).
\end{equation}
Hence we have for $a\in [a_1, a_2]$, 
\begin{equation}
\label{rem0_est_all_2nd}
{|\err_0(a, b)|}\le M(b)\Pi_1(b). 
\end{equation}

\bigskip

As an example, let $g$ be the Gaussian function defined by  \eqref{def_g}.
% as an example to illustrate our approach.  One can obtain for this $g$ 
Then we have (see e.g. \cite{Gibson06,LCHJJ18})
\begin{equation}
%\label{CWT_LinearChip}
\label{def_Gk}
G_{k}(u, a, b)=\frac {1}{\sqrt{1-i2\pi \phi''_k(b)a^2\gs^2(b)}}\;
e^{-\frac{2\pi^2 u^2}{1+(2\pi \phi''_k(b)a^2\gs^2(b))^2} (1+i2\pi \phi''_k(b)a^2\gs^2(b))}.
%G_{0, k}(a, b)=\frac {1}{\sqrt{1-i2\pi \phi''_k(b)\gs^2(b)}}\;
%e^{-\frac{2\pi^2 \gs^2(b)}{1+\big(2\pi \phi''_k(b)\gs^2(b)\big)^2}\big(1+i2\pi \phi''_k(b)\gs^2(b)\big) \big(\mu -a \phi_k'(b)\big)^2}.
\end{equation}
Thus
\begin{equation}
\label{abs_Gk}
|G_k(u, a, b)|=\frac 1{\big(1+(2\pi \phi''_k(b)a^2\gs^2(b))^2\big)^{\frac 14}}\;
e^{-\frac{2\pi^2}{1+(2\pi \phi''_k(b)a^2 \gs^2(b))^2}u^2}.
\end{equation}
Therefore, in this case,
%$\ga_k$ obtained by solving
%$$|G_k(u)|/\max_u |G_k(u)|=e^{-\frac{2\pi^2}{1+(2\pi \phi''_k(b)a^2 \gs^2(b))^2}u^2}=\tau_0
%$$ is
assuming $\tau_0 (1+  (2\pi \phi''_k(b)a^2 \gs^2(b))^2)^{\frac 14}\le 1$ (otherwise, $|G_k(u, a, b)|< \tau_0$ for any $u$),
%\begin{equation}
%\label{def_zone}
 %O_k=\Big\{(a, b): |\mu -a \phi_k'(b)|< \sqrt{\frac 1{\gs^2(b)}+
 %(2\pi \phi''_k(b) \gs(b))^2} \; \frac 1{2\pi}\sqrt{-2\ln (\tau_0 \sqrt{1+
 %(2\pi \phi''_k(b) \gs^2(b))^2})}, t\in \R\Big\}.
%\end{equation}
\begin{equation*}
%\label{def_gak_Gaussian}
\ga_k=\ga \sqrt{1+(2\pi \phi''_k(b) a^2 \gs^2(b))^2}
\; \frac 1{2\pi}\sqrt{2\ln (\frac 1{\tau_0})-\frac 12 \ln(1+(2\pi \phi''_k(b)a^2 \gs^2(b))^2)}.
\end{equation*}
Authors of \cite{LCHJJ18} replaced $\ga_k$ by
\begin{equation}
\label{def_gbk}
\gb_k=\ga\big(1+ 2\pi |\phi''_k(b)|a^2 \gs^2(b)\big),
\end{equation}
where $\ga= \frac 1{2\pi}\sqrt{2\ln (1/{\tau_0})}$ as defined by \eqref{def_ga}. 
Since $\ga_k\le \gb_k$, we know \eqref{def_Ok2_ineq} holds.  
That is $\wt W_{x_k}(a, b)$ lies within the scale-time zone: % $O_k$: % in the scale-time plane:
   \begin{equation*}
%\label{def_zone_larger}
 %O_k=
 \Big\{(a, b): |\mu -a \phi_k'(b)|< \frac {\ga}{\gs(b)}
  \Big(1+ 2\pi |\phi''_k(b)| a^2 \gs^2(b)\Big), b\in \R\Big\},
\end{equation*}
which can be written as \eqref{def_Ok2} with (see \cite{LCHJJ18})
\begin{eqnarray*}
u_k(b)\hskip -0.6cm &&
\label{zone_k_eq_big_upper}
=\frac {2(\mu+\frac \ga{\gs(b)})}{\phi'_k(b)+\sqrt{\phi'_k(b)^2-8\pi \ga (\ga+\mu \gs(b))|\phi''_k(b)|}},\\
&&\nonumber \\
l_k(b)\hskip -0.6cm &&
\label{zone_k_eq_big_low}
=\frac {2(\mu-\frac \ga{\gs(b)})}{\phi'_k(b)+\sqrt{\phi'_k(b)^2+8\pi \ga (\mu \gs(b)-\ga)|\phi''_k(b)|}}.
\end{eqnarray*}
In addition, \cite{LCHJJ18} provides the well-separated conditions such that  \eqref{cond_no_overlapping} holds.

\bigskip
%Let $g$ be a window function in the Schwarz class $\cS$ with  $|\wh g(\xi)|$ is  even and decreasing for $\xi\ge 0$.

In the following we assume  $x(t)$ given by \eqref{AHM} satisfies  
\eqref{cond_basic0}-\eqref{condition1}, \eqref{cond_basic_2nd}, and \eqref{condition2} 
for some $\vep_1>0, \vep_3>0$.  In addition, we assume 
the adaptive CWLTs $\wt W_{x_k}(a,b)$ of its components  with a window function $g$ 
%(meaning $x_k(b)G_k\big(\gs(b)(\mu-a\phi'_k(b))\big)$)
lie within non-overlapping scale-time zones $O_k$ in the sense that
%$|G_k\big(\gs(b)(\mu-a\phi'_k(b))\big)|\le \tau_0$ for $(a,b)\in O_k$.
\eqref{def_Ok2_ineq} holds and each $O_k$ is  given by \eqref{def_Ok2} with \eqref{cond_no_overlapping} satisfied for some $\gs(b)$. Let ${\cal E}_{\vep_1, \vep_3}$ denote the set of such multi-component signals. 

Denote 
 \begin{eqnarray}
\label{def_h0} 
&& h_0(b):=\min_{1\le k \le K} |G_k(0, \frac{\mu}{\phi'_k(b)}, b)|,   %\hbox{{\bf Note: we do not know $G_k(0, \frac{\mu}{\phi'_k(b)}, b)$}}
\\
 \label{def_Err}
&& \tErr_\ell=\tErr_\ell(b):=M(b)\Pi_\ell(b)+\tau_0\sum_{k\ne \ell}A_k(b), 
 \end{eqnarray}
where $\Pi_\ell(b)$ is defined by \eqref{def_gLk_2nd}. 

%Thus the multi-component signal $x(t)$ is well-separated (that is $Z_k\cap Z_{\ell}=\O, k\not=\ell$.), provided that $\gs(t)$ satisfies \begin{equation}
%\label{separated_cond_1st_compact} \gs(t)\ge \frac {2\gt}{ \phi'_k(t)-\phi_{k-1}'(t)}, \; t\in \R, k=1, \cdots, K. \end{equation}
%Observe that our well-separated condition \eqref{separated_cond_1st_compact} is different from that in \eqref{freq_resolution} considered in \cite{Wu_thesis} and \cite{MOM14}.

\begin{lem}
Let $x(t)\in {\cal E}_{\vep_1, \vep_3}$ and $\wt W_x(a, b)$ be its adaptive CWLT with a window function $g$.  Then for any $(a, b)\in O_\ell$, 
\begin{equation}
\label{est_xk_CWT_2nd}
 \big|\wt W_x(a, b)-x_\ell(b) G_\ell\big(\gs(b)(\mu -a \phi'_\ell(b)), a, b\big)\big|\le \tErr_\ell(b). 
\end{equation}
 \end{lem}
 
 \begin{proof}
 For any $(a, b)\in O_\ell$, from \eqref{CWT_approx}, we have 
\begin{eqnarray*}
\nonumber && \big|\wt W_x(a, b)-x_\ell(b) G_\ell\big(\gs(b)(\mu -a \phi'_\ell(b)), a, b\big)|\\
&&=\Big|\sum_{k\ne \ell} x_k(b)  G_k\big(\gs(b)(\mu -a \phi'_\ell(b)), a, b\big) +\err_0\Big|\\
&&\le \big|\err_0\big|+\sum_{k\ne \ell} A_k(b)  \big| G_k\big(\gs(b)(\mu -a \phi'_\ell(b)), a, b\big)\big| \\
&&\le  M(b)\Pi_\ell(b)+\sum_{k\ne \ell} A_k(b)\tau_0 \; \quad \hbox{(by \eqref{rem0_est_k_2nd} and \eqref{def_Ok2_ineq})}\\
&&= \tErr_\ell(b), 
\end{eqnarray*}
as desired. 
%where the second last inequality follows from \eqref{rem0_est_k_2nd} and \eqref{def_Ok2_ineq}.
% and the assumption that $f(\xi)$ is decreasing for $\xi\ge 0$. 
\end{proof}

\bigskip

For a window function, define 
\begin{equation}
\label{def_PFT}
\wb g(\xi, \gl):={\cal F}\big(e^{-i\pi \gl t^2}g(t)\big)\big(\xi)=
\int_{\R} g(t) e^{-i2\pi \xi t-i\pi \gl t^2}dt.    
\end{equation}
%In some literature,  $\wb g(\eta, \gl)$ is called a polynomial Fourier transform of $g$, see \cite{Bi_Stankovic11,Stankovic13}.
In this section, we assume a window function satisfying the following assumption. 

{\bf Assumption 2.} \;  {\it $|\wb g(\xi, \gl)|$ can be written as  
\begin{equation*}
%\label{def_F}
|\wb g(\xi, \gl)|=F(|\xi|, \gl), 
\end{equation*}
where for each $\gl\in \R$, $F(\xi, \gl)$ is a positive and {\rm (}strictly{\rm )} decreasing function in $\xi$ for $\xi\ge 0$. }

\bigskip

Observe that $G_k(\xi, a, b)$ defined by \eqref{def_Gk_general} is $\wb g(\xi, -\phi''_k(b) a^2 \gs^2(b))$. 
Thus for fixed $a$ and $b$, $|G_k(\xi, a, b)|$ is a decreasing function in $\xi$ for $\xi\ge 0$.   
Denote 
\begin{equation}
\label{def_Fk}
F_{a, k}(\xi):=|G_k(\xi, a, b)|, \; \xi \ge 0. 
\end{equation}
Here we drop $b$ in $F_{a, k}(\xi)$ for simplicity of presentation of the paper. We let $F^{-1}_{a, k}(\xi)$ denotes the inverse function of $F_{a, k}(\xi), \xi\ge 0$. 
%Since $F_{a, k}(\xi)$  decreasing on $\xi\ge 0$, it has inverse

%\bigskip 
For a fixed $b$ and a positive number  $\wt \ep_1$ (possibly depending on $b$), 
we let $\cG_b$ denote the set defined by \eqref{def_cGk}, and we define 
%$\mathcal H_{t, k}$ denote the sets defined by 
\begin{equation}
\label{def_Hk}
%\cG_b:=\{\eta: \; |\wt V_x(t, \eta)|>\wt \ep_1\}, \; 
%\cH_{t, k}:=\{\eta \in \cG_b: \; |\eta-\phi'_k(t)|< \frac {\ga_k}{\gs(t)}\}.  
\cH_{b, k}=\cG_b\cap \{a: \; (a, b)\in O_ k\}. 
\end{equation}
Note that $\cH_{b, k}$ also depends on $\wt \ep_1$, and for simplicity of presentation, we drop 
$\wt \ep_1$ from it. %Also observe that $\cH_{t, k}=\cG_b\cap \{\eta: \; (t, \eta)\in O_ k\}$.  
When $\cH_{b, k}, 1\le k\le K$ are non-empty and non-overlapping, let 
 $\wc a_\ell$ be defined by \eqref{def_max_eta_2nd}. 
Observe that $|G_\ell(0,  a, b)|=|\wb g(0, -\phi''_k(b) a^2 \gs^2(b))|=F(0, -\phi''_k(b) a^2 \gs^2(b))$ is always positive by Assumption 2. Thus  $\wc a_\ell$ in \eqref{def_max_eta_2nd} is well defined. 

Next theorem provides our analysis results on adaptive CWLT-based IF estimation and component recovery derived from linear chirp local approximation. 
\begin{theo}
\label{theo:main_2nd} Let $x(t)\in {\cE}_{\vep_1, \vep_3}$ for some $\vep_1, \vep_3>0$ and $g$ be a window function satisfying Assumption 2.  
Suppose %for this $x(t)$, there is a function $\gs(b)>0$ such that \eqref{cond_no_overlapping} holds, and  
 $\vep_1, \vep_3, \tau_0$ are small enough such that 
\begin{eqnarray*}
%2\tErr_\ell<{|G_\ell(0, \wc a_\ell(b), b)|} A_\ell(b) \; \hbox{and} \; \;  
&&2M(b)\big(\tau_0+\Pi_1(b)\big)\le h_0(b)\nu(b), \\
&&\Big(\frac 1{\big|G_\ell\big(0,\wc a_{\ell}, b\big)\big|}+\frac 1{|G_\ell(0,  \frac\mu{\phi_\ell\rq{}(b)}, b)|}\Big)\frac {\tErr_\ell(b)}{A_\ell(b)}<1. 
\end{eqnarray*} 
Let $\wt \ep_1$ be a function of $b$ satisfying  
\begin{equation}
\label{cond_ep_2nd}
M(b)\big(\tau_0+\Pi_1(b)\big)\le 
\wt \ep_1 \le h_0(b) \nu(b)-M(b)\big(\tau_0+\Pi_1(b)\big).
\end{equation}
Then the following statements hold. 
\begin{enumerate}
\item[{\rm (a)}] Let $\cG_b$ and $\cH_{b, k}$ be the sets defined by \eqref{def_cGk}and \eqref{def_Hk} respectively for some $\wt \ep_1$ satisfying \eqref{cond_ep_2nd}.   Then $\cG_b$ can be expressed as a disjoint union of exactly $K$ non-empty sets $\cH_{b, k}, 1\le k\le K$. 
%In addition,  for each $\ell\in \{0, 1, \cdots, K\}$, we have $(t, \phi'_\ell(b))\in \cG_\ell$  for any $t$. 

\item[{\rm (b)}] Let $\wc a_\ell$ be defined by \eqref{def_max_eta_2nd}.  Then for $\ell=1, \cdots, K$, 
\begin{equation}\label{phi_est_2nd}
|\mu -\wc a_{\ell}(b)\phi_{\ell}^{\gp}(b)|\le \Bd_{1, \ell}:=\frac{1}{\gs(b)} F_{\wc a_{\ell}, \ell}^{-1}\Big(\big|G_\ell\big(0,\wc a_{\ell}, b\big)\big|-\frac {\tErr_\ell(b)}{A_\ell(b)}-\frac{\big|G_\ell\big(0,\wc a_{\ell}, b\big)\big|}{|G_\ell(0,  \frac\mu{\phi_\ell\rq{}(b)}, b)|}
\frac {\tErr_\ell(b)}{A_\ell(b)}
\Big), 
\end{equation}
where $\tErr_\ell$ is defined by \eqref{def_Err}, $F_{\wc a_{\ell}, \ell}^{-1}(\xi)$ is the inverse function 
$F_{\wc a_{\ell}, \ell}(\xi)=|G_\ell(\xi, \wc a_{\ell}, b)|$. 
\item[{\rm (c)}]  For $\ell=1, 2, \cdots, K$,
\begin{equation}
\label{comp_xk_est_2nd}
\big|\wt W_{x}(\wc a_\ell, b)-G_\ell(0, \wc a_\ell, b) x_\ell(b)\big|
\le \Bd_{2, \ell},  
\end{equation}
where 
 \begin{equation}
\label{def_B2_ell}
 \Bd_{2, \ell}:=\tErr_\ell(b)+2\pi I_1 A_\ell(b)  F_{\wc a_{\ell}, \ell}^{-1}
 \Big(\big|G_\ell\big(0,\wc a_{\ell}, b\big)\big|-\frac {\tErr_\ell(b)}{A_\ell(b)}-\frac{\big|G_\ell\big(0,\wc a_{\ell}, b\big)\big|}{|G_\ell(0,  \frac\mu{\phi_\ell\rq{}(b)}, b)|}
\frac {\tErr_\ell(b)}{A_\ell(b)}
\Big). 
 \end{equation}
\item[{\rm (d)}] For $\ell=1, 2, \cdots, K$,
\begin{equation}
\label{abs_IA_est_2nd}
\big| |\wt W_{x}(\wc a_\ell, b)|-|G_\ell(0, \wc a_\ell, b)| A_\ell(b)\big|
\le  \tErr_\ell(b) \max \Big\{1,\frac{\big|G_\ell\big(0,\wc a_{\ell}, b\big)\big|}{|G_\ell(0,  \frac\mu{\phi_\ell\rq{}(b)}, b)|}\Big\}. 
\end{equation} 
\end{enumerate}
\end{theo}

Note that to obtain $\wc a_k$ in \eqref{def_max_eta_2nd}, we need $\phi_k^{\gp\gp}(t)$ to calculate $|G_k(0,  a, b)|$. 
However in practice we do not know in general  $\phi_k^{\gp\gp}(t)$. Thus we need to estimate $\phi_k^{\gp\gp}(t)$. The reader refers to \cite{LCJ20,CJLL20_adpSTFT} for methods to find an approximation of $\phi_k^{\gp\gp}(t)$.  

\begin{mrem}
From Theorem \ref{theo:main_2nd}{\rm (c)}, we have the recovery formula $\eqref{comp_xk_est_2nd1}$. 
%\begin{equation}
%\label{comp_xk_est_2nd1}
 %x_\ell(t) \approx \frac 1{G_\ell(0, \wc a_\ell, b)}\wt W_{x}(\wc a_\ell, b). 
%\end{equation}
For a real-valued $x(t)$, the  recovery formula for component $x_\ell(t)$ is   
\begin{equation}
\label{comp_xk_est_2nd1_real}
 x_\ell(t) \approx 2{\rm Re}\Big(\frac 1{G_\ell(0, \wc a_\ell, b)}\wt W_{x}(\wc a_\ell, b)\Big). 
\end{equation}

In some cases $\wc a_\ell$ defined by \eqref{def_max_eta_2nd} is $\wh a_\ell$ defined by \eqref{def_max_eta}. %Thus in this case, the linear chirp-based model does not provide a more accurate IF approximation than the sinusoidal signal-based model. However, 
Compared with sinusoidal signal-based recovery formulas \eqref{comp_xk_est0} and \eqref{comp_xk_est_real}, the component recovery formulas derived from linear chirp local approximation have a factor 
 $\frac 1{G_\ell(0, \wc a_\ell, b)}$ %or $\sqrt{1-i2\pi \phi''_\ell(b)\wc a_\ell^2\gs^2(b)}$ 
 as shown in \eqref{comp_xk_est_2nd1} and \eqref{comp_xk_est_2nd1_real}. %\eqref{comp_xk_est_2nd1_gaussian} and \eqref{comp_xk_est_2nd1_real_gaussian}. 
 However, this simple factor results in a more accurate component recovery.  
 %, which distinguishes the linear chirp-based model  from the sinusoidal signal-based model. 
 \hfill $\blacksquare$
\end{mrem}

\bigskip 
\begin{example}\label{example2}
Let $g(t)$ be the Gaussian window function given by \eqref{def_g}. Then the corresponding $G_k(\xi, a, b)$ is given by \eqref{def_Gk}.  $|G_k(\xi, a, b)|$ can be written as $F_{a, k}(|\xi|)$ with $F_{a, k}$ given by (see \eqref{abs_Gk}) 
$$
F_{a, k}(\xi)=\frac 1{\big(1+(2\pi \phi''_k(b)a^2\gs^2(b))^2\big)^{\frac 14}}\;
e^{-\frac{2\pi^2}{1+(2\pi \phi''_k(b)a^2 \gs^2(b))^2}\xi^2}.
$$
 Then we have 
$$
F^{-1}_{a, \ell}(\xi)= \frac1{\pi \sqrt 2 F_{a, \ell}(0)^2} \Big(-\ln \frac {\xi}{F_{a, \ell}(0)}\Big)^{1/2}, \;  0<\xi <F_{a, \ell}(0),  
$$
where 
$$
F_{a, \ell}(0)=\big(1+(2\pi \phi''_\ell(b)\gs^2(b) a^2)^2\big)^{-\frac 14}.
$$ 
Note that $\big|G_\ell\big(0,\wc a_{\ell}, b\big)\big|=F_{\wc a, \ell}(0)$.  
Hence the error bound $\Bd_{1, \ell}$ in \eqref{phi_est_2nd} is 
\begin{eqnarray*}
\Bd_{1, \ell}\hskip -0.6cm && =\frac{1}{\gs(b)\pi \sqrt 2 F_{\wc a, \ell}(0)^2} \Big\{-\ln \Big(1-\big(\frac 1{\big|G_\ell\big(0,\wc a_{\ell}, b\big)\big|}+\frac 1{|G_\ell(0,  \frac\mu{\phi_\ell\rq{}(b)}, b)|}\big)\frac {\tErr_\ell(b)}{A_\ell(b)}\Big)\Big\}^{1/2}\\
&&\\
&&= \frac{\big(1+(2\pi \phi''_\ell(b)\gs^2(b) \wc a^2)^2\big)^{1/2}}{\gs(b)\pi \sqrt 2} \times \\
&& \quad \Big\{-\ln \Big(1-\Big\{
\Big(1+\big(2\pi \phi''_\ell(b)\gs^2(b) \wc a^2\big)^2\Big)^{1/4}+
\Big(1+\big(2\pi \phi''_\ell(b)\gs^2(b) \big(\frac\mu{\phi_\ell\rq{}(b)}\big)^2 \big)^2\Big)^{1/4}\Big\}\frac {\tErr_\ell(b)}{A_\ell(b)}\Big)\Big\}^{1/2}. 
\end{eqnarray*}
%Using the fact $-\ln(1-t)< 2t$ for $0<t<\frac12$ and 
%Suppose $\vep_1, \vep_3, \tau_0$ are small enough such that $\tErr_\ell(b)\le \frac 14 |G_\ell(0)|{A_\ell(b)}$. Then 
Suppose $\vep_1, \vep_3, \tau_0$ are small enough such that
$$
\Big\{
\Big(1+\big(2\pi \phi''_\ell(b)\gs^2(b) \wc a^2\big)^2\Big)^{1/4}+
\Big(1+\big(2\pi \phi''_\ell(b)\gs^2(b) \big(\frac\mu{\phi_\ell\rq{}(b)}\big)^2 \big)^2\Big)^{1/4}\Big\}\frac {\tErr_\ell(b)}{A_\ell(b)}\le c_0 
$$
for some $0< c_0<1$. Then applying the fact $-\ln(1-t)< \frac 1{1-c_0}t$ for $0<t\le c_0$ again,  we have   
\begin{eqnarray*}
\Bd_{1, \ell}\hskip -0.6cm &&< \frac{\big(1+(2\pi \phi''_\ell(b)\gs^2(b) \wc a^2)^2\big)^{1/2}}{\gs(b)\pi \sqrt 2} \times \\
&& \quad \frac 1{\sqrt{1-c_0}} \Big\{
\Big(1+\big(2\pi \phi''_\ell(b)\gs^2(b) \wc a^2\big)^2\Big)^{1/4}+
\Big(1+\big(2\pi \phi''_\ell(b)\gs^2(b) \big(\frac\mu{\phi_\ell\rq{}(b)}\big)^2 \big)^2\Big)^{1/4}\Big\}^{1/2}\sqrt{\frac {\tErr_\ell(b)}{A_\ell(b)}}. 
\end{eqnarray*}
Also, in this case, we have 
\begin{eqnarray*}
&&\Bd_{2, \ell}<  
\tErr_\ell(b)+\frac{\sqrt 2 I_1}{\sqrt{1-c_0}}\big(1+(2\pi \phi''_\ell(b)\gs^2(b) \wc a^2)^2\big)^{1/2} \times \\
&& \quad\Big\{
\Big(1+\big(2\pi \phi''_\ell(b)\gs^2(b) \wc a^2\big)^2\Big)^{1/4}+
\Big(1+\big(2\pi \phi''_\ell(b)\gs^2(b) \big(\frac\mu{\phi_\ell\rq{}(b)}\big)^2 \big)^2\Big)^{1/4}\Big\}^{1/2}\sqrt{\tErr_\ell(b)A_\ell(b)}. 
\end{eqnarray*}
From the above estimate for $\Bd_{2, \ell}$, we can see that a smaller $\gs(b)$ overall results in a smaller  
error bound for 
%$\Bd_{1, \ell}$ and 
$\Bd_{2, \ell}$.   

Note that 
$$
G_\ell(0, a, b)=\frac {1}{\sqrt{1-i2\pi \phi''_\ell(b)a^2\gs^2(b)}}. 
$$
Thus component recovery formulas in \eqref{comp_xk_est_2nd1} and \eqref{comp_xk_est_2nd1_real} are respectively 
\begin{equation}
\label{comp_xk_est_2nd1_gaussian}
 x_\ell(t) \approx \sqrt{1-i2\pi \phi''_\ell(b)\wc a_\ell^2\gs^2(b)}\; \wt W_{x}(\wc a_\ell, b)
\end{equation}
and  
\begin{equation}
\label{comp_xk_est_2nd1_real_gaussian}
 x_\ell(t) \approx 2{\rm Re}\Big(\sqrt{1-i2\pi \phi''_\ell(b)\wc a_\ell^2\gs^2(b)}\; \wt W_{x}(\wc a_\ell, b)\Big), 
\end{equation}
for a real-valued $x(t)$.  
\hfill $\blacksquare$
\end{example}

Next we provide the proof of Theorem \ref{theo:main_2nd}, which is similar to that of Theorem \ref{theo:main_1st}.

\bigskip

{\bf Proof  of  Theorem \ref{theo:main_2nd}(a)}.  Clearly $\cup _{k=1}^K \cH_{b, k}\subseteq \cG_b$. 
Next we show  $\cG_b\subseteq \cup _{k=1}^K \cH_{b, k}$.  Let $a \in \cG_b$. 
Assume $a\not \in \cup _{k=1}^K \cH_{b, k}$. That is $(a, b)\not \in \cup _{k=1}^K O_k$.  
Then by \eqref{def_Ok2_ineq}, 
we have  
$$
|G_k\big(\gs(b)(\mu-a\phi'_k(b)), a, b\big)|\le \tau_0.
$$ 
Hence, by \eqref{CWT_approx} and \eqref{rem0_est_all_2nd}, we have 
\begin{eqnarray*}
|\wt W_x(a, b)|\hskip -0.6cm &&\le \sum_{k=1}^K \big|x_k(b) G_k\big(\gs(b)(\mu-a\phi'_k(b)), a, b\big)\big| + |\err_0|
\\
&&\le \sum_{k=1}^K A_k(b)\tau_0  +M(b)\Pi_1(b)
=M(b)(\tau_0  +\Pi_1(b))\le \wt \ep_1, 
\end{eqnarray*}
a contradiction to the assumption $|\wt W_x(a, b)|>\wt \ep_1$. Thus $(a, b)\in O_\ell$ for some $\ell$. 
This shows that $a \in \cH_{b, \ell}$. Hence $\cG_b=\cup _{k=1}^K \cH_{b, k}$. 
Since $O_k, 1\le k\le K$ are not overlapping, we know  $\cH_{b, k}, 1\le k\le K$ are disjoint.  

To show that $\cH_{b, \ell}$ is non-empty, 
%$\phi'_\ell(b)\in \cG_\ell$, 
we need only to show $\frac\mu{\phi_\ell\rq{}(b)}\in \cG_b$.
From \eqref{est_xk_CWT_2nd} with $a=\frac \mu{\phi'_\ell(b)}$, we have 
\begin{eqnarray}
\label{est_xk_CWT2}
\big|\wt W_x(\frac \mu{\phi'_\ell(b)}, b)\big|\ge \big|x_\ell(b) G_\ell\big(0, \frac\mu{\phi_\ell\rq{}(b)}, b\big)\big|-\tErr_\ell(b)
 =A_\ell(b)\big|G_\ell\big(0, \frac\mu{\phi_\ell\rq{}(b)}, b\big)\big| - \tErr_\ell(b).  
\end{eqnarray}
Then by the definition of $\tErr_\ell(b)$, we have 
\begin{eqnarray*}
\big|\wt W_x(\frac\mu{\phi_\ell\rq{}(b)}, b)\big|\hskip -0.6cm &&>  A_\ell(b) \big|G_\ell\big(0, \frac\mu{\phi_\ell\rq{}(b)}, b\big)\big|- M(b)\tau_0  -M(b)\Pi_\ell(b)\\
&& \ge  \nu(b) h_0(b)- M(b)\big(\tau_0 +\Pi_1(b)\big) \ge \wt \ep_1.  
\end{eqnarray*}
Hence we conclude that $\frac\mu{\phi_\ell\rq{}(b)}\in \cG_b$. Thus all statements in  Theorem \ref{theo:main_2nd}(a) hold. 
\hfill $\square$

%\bigskip 

{\bf Proof  of  Theorem \ref{theo:main_2nd}(b)}. 
By the definition of $\wc a_\ell$ and \eqref{est_xk_CWT_2nd}, we have 
\begin{equation}
\label{est_xk_CWT2_both}
\big|\wt W_x(\frac \mu{\phi'_\ell(b)}, b)\big| \frac{\big|G_\ell\big(0,\wc a_{\ell}, b\big)\big|}{|G_\ell(0,  \frac\mu{\phi_\ell\rq{}(b)}, b)|}\le \big|\wt W_x(\wc a_\ell, b)\big| \le 
\big|x_\ell(b) 
G_\ell\big(\gs(b)(\mu -\wc a_\ell \; \phi'_\ell(b)),  \wc a_\ell, b\big)\big|+ \tErr_\ell(b)
\end{equation}
This and  \eqref{est_xk_CWT2} lead to  
\begin{equation*}
%\label{est_xk_CWT2_both}
A_\ell(b)\big|G_\ell\big(0, \wc a_\ell, b\big)\big|  -\frac{\big|G_\ell\big(0,\wc a_{\ell}, b\big)\big|}{|G_\ell(0,  \frac\mu{\phi_\ell\rq{}(b)}, b)|} \tErr_\ell(b)\le A_\ell(b) \big|G_\ell \big(\gs(b)(\mu -\wc a_\ell \; \phi'_\ell(b)),  \wc a_\ell, b \big)\big|+ \tErr_\ell(b), 
\end{equation*}
that is 
$$
0<\big|G_\ell\big(0, \wc a_\ell, b\big)\big|- \frac{\tErr_\ell(b)}{A_\ell(b)}-\frac{\big|G_\ell\big(0,\wc a_{\ell}, b\big)\big|}{|G_\ell(0,  \frac\mu{\phi_\ell\rq{}(b)}, b)|}\frac{\tErr_\ell(b)}{A_\ell(b)} \le  F_ {\wc a_\ell, \ell}\big(|\gs(b)(\mu -\wc a_\ell \; \phi'_\ell(b)) |\big).  
$$
%Then \eqref{phi_est_2nd} follows from the above inequality and that $ F_ {\wc a_\ell, \ell}(\xi)$ is decreasing for $\xi$ on $(0, \infty)$. 
Therefore we have 
$$
\gs(b)|\mu -\wc a_\ell \; \phi'_\ell(b))| \le F^{-1}_ {\wc a_\ell, \ell}\Big(\big|G_\ell\big(0, \wc a_\ell, b\big)\big|-\frac{\tErr_\ell(b)}{A_\ell(b)} - \frac{\big|G_\ell\big(0,\wc a_{\ell}, b\big)\big|}{|G_\ell(0,  \frac\mu{\phi_\ell\rq{}(b)}, b)|}\frac{\tErr_\ell(b)}{A_\ell(b)}\Big). 
$$
This proves \eqref{phi_est_2nd}. \hfill $\square$

%\bigskip 

{\bf Proof  of  Theorem \ref{theo:main_2nd}(c)}.
From \eqref{est_xk_CWT_2nd}, we have 
\begin{eqnarray*}
&&\big|\wt W_x(\wc a_\ell, b)- G_\ell\big(0, \wc a_\ell, b\big)x_\ell(b)\big|\\
&&\le \big|\wt W_x(\wc a_\ell, b)-x_\ell(b) G_\ell\big(\gs(b)(\mu -\wc a_\ell \; \phi'_\ell(b)), \wc a_\ell, b\big)\big|\\
&& \qquad +\big |x_\ell(b) G_\ell\big(\gs(b)(\mu -\wc a_\ell \; \phi'_\ell(b)), \wc a_\ell, b\big)-G_\ell\big(0, \wc a_\ell, b\big)x_\ell(b)\big|
\\
&&\le \tErr_\ell(b)+A_\ell(b) \Big| \int_\R g(t) e^{i\pi \phi_\ell''(b) (\wc a_\ell \; \gs(b)t)^2}\Big(e^{-i2\pi \gs(b)(\mu -\wc a_\ell \; \phi'_\ell(b) )t }-1\Big) dt \Big|\\
&&\le \tErr_\ell(b)+A_\ell(b) \int_\R |g(t)| \; 2\pi \gs(b)\big|\mu -\wc a_\ell \; \phi'_\ell(b)\big | |t| dt\\
&&\le \tErr_\ell(b)+A_\ell(b) 2\pi \gs(b)\big|\mu -\wc a_\ell \; \phi'_\ell(b)\big |\int_\R |g(t)|  |t| dt\\
&&\le   \tErr_\ell(b)+2\pi I_1 A_\ell(b) 
 F^{-1}_ {\wc a_\ell, \ell}\Big(\big|G_\ell\big(0, \wc a_\ell, b\big)\big|-\frac{\tErr_\ell(b)}{A_\ell(b)} - \frac{\big|G_\ell\big(0,\wc a_{\ell}, b\big)\big|}{|G_\ell(0,  \frac\mu{\phi_\ell\rq{}(b)}, b)|}\frac{\tErr_\ell(b)}{A_\ell(b)}\Big).
\end{eqnarray*}
This proves \eqref{comp_xk_est_2nd}. 
\hfill $\square$

{\bf Proof  of  Theorem \ref{theo:main_2nd}(d)}. Observe that, by Assumption 2, 
$$
|G_\ell(\xi, a, b)|=|\wb g(\xi, -\phi''_\ell(b) a^2 \gs^2(b))|\le |\wb g(0, -\phi''_\ell(b) a^2 \gs^2(b))|=|G_\ell(0, a, b)|. 
$$
Thus from \eqref{est_xk_CWT2} and \eqref{est_xk_CWT2_both}, we have 
\begin{equation}
\label{est_xk_CWT2_both2}
A_\ell(b)\big|G_\ell\big(0, \wc a_\ell, b\big)\big|  -\frac{\big|G_\ell\big(0,\wc a_{\ell}, b\big)\big|}{|G_\ell(0,  \frac\mu{\phi_\ell\rq{}(b)}, b)|} \tErr_\ell(b)\le \big|\wt W_x(\wc a_\ell, b)\big| 
\le A_\ell(b) \big|G_\ell \big(0,  \wc a_\ell, b \big)\big|+ \tErr_\ell(b).  
\end{equation}
\eqref{abs_IA_est_2nd} follows immediately from \eqref{est_xk_CWT2_both2}. This completes the proof of  Theorem \ref{theo:main_2nd}(d). 
\hfill $\square$

%\bigskip 

\section{Experimental results} 

In this section we provide some experimental results to illustrate our general theory. In our experiments, we use the Gaussian function $g(t)$ defined by \eqref{def_g} as the window function and we set $\tau_0=\frac 18$ for the {\it essential support} of $g(t)$. In addition, we let $\mu=1$. As we observe in Examples 1 and 2 that for either the sinusoidal signal-based model or the linear chirp-based model, a  smaller  $\gs(b)$ results in smaller error bounds  $\bd_{2, \ell}$ and $\Bd_{2, \ell}$ for component recovery. Thus we should choose a small $\gs(b)$. However if $\gs(b)$ is too small, components of $x(t)$ will not be separated in the time-scale plane, namely, the time-scale zones of $\wt W_{x_k}(a, b)$  and $\wt W_{x_{k-1}}(a, b)$ overlap. \cite{LCJ20} derives the optimal time-varying parameter $\gs(b)$:  for the  sinusoidal signal-based model,  $\gs(b)$ is 
\begin{equation}
\label{def_gs1}
  \gs_1(b):=\max_{2\le k\le K}\Big\{\frac \ga\mu\; \frac{\phi'_k(b)+\phi'_{k-1}(b)}{\phi'_k(b)-\phi'_{k-1}(b)}\Big\};
\end{equation}
while for the linear chirp-based model,  $\gs(b)$ is 
\begin{equation}
\label{def_gs2}
\gs_2(b):=\left\{
\begin{array}{ll}
\max\Big\{\frac\ga\mu, \frac{\gb_k(b)-\sqrt{\Upsilon_k(b)}}{2\ga_k(b)}: \; 2\le k\le K\Big\}, &\hbox{if $|\phi''_k(b)|+|\phi''_{k-1}(b)|\not=0$}, \\
& \\
 \max\Big\{\frac \ga\mu \frac{\phi'_k(b)+\phi'_{k-1}(b)}{\phi'_k(b)-\phi'_{k-1}(b)}: \; 2\le k\le K\Big\}, &\hbox{if $\phi''_k(b)=\phi''_{k-1}(b)=0$,}
\end{array}
\right.
\end{equation}
where 
\begin{eqnarray*}
\label{gak}
&&\ga_k(b):=2\pi \ga \mu (|\phi''_k(b)|+|\phi''_{k-1}(b)|)^2,\\
\label{gbk}
&&\gb_k(b):=\big(\phi'_k(b)|\phi''_{k-1}(b)|+\phi'_{k-1}(b)|\phi''_{k}(b)|\big) \big(\phi'_k(b)-\phi'_{k-1}(b)\big)+4\pi \ga^2  \big(\phi''_k(b)^2-\phi''_{k-1}(b)^2\big),\\
%\label{ggak}
%&&\gga_k(b)=\frac{\ga}\mu\Big\{
%\big(\phi'_k(b)|\phi''_{k-1}(b)|+\phi'_{k-1}(b)|\phi''_{k}(b)|\big)\big(\phi'_k(b)+\phi'_{k-1}(b)\big)+2\pi \ga^2  \big(|\phi''_k(b)|-|\phi''_{k-1}(b)|\big)^2\Big\}.
&&\Upsilon_k(b):=%\gb_k(b)^2-4\ga_k(b)\gga_k(b)\\&&=
\big(\phi'_k(b)|\phi''_{k-1}(b)|+\phi'_{k-1}(b)|\phi''_{k}(b)|\big)^2\Big\{\big(\phi'_k(b)-\phi'_{k-1}(b)\big)^2-16\pi \ga^2 \big(|\phi''_k(b)|+|\phi''_{k-1}(b)|\big)
\Big\}, 
\end{eqnarray*}
with $\ga$ defined by \eqref{def_ga}. In practice $\phi'_k(t), \phi''_k(b), 1\le k\le K$ are unknown and hence, one needs to develop some algorithms to  approximate  $\gs_1(b)$ and $\gs_2(b)$. Readers refer to \cite{LCJ20} for an algorithm to approximate $\gs_2(b)$. In this paper, to discover the performance of our proposed approach for signal separation, we will use  both $\gs_1(b)$ and $\gs_2(b)$ as $\gs(b)$. Next we consider two signals.

 %%\bigskip 
%%%%%%%%%%%%%%%%%%%the beginning of figure 1   %%%%%%%%%%%%%%%
\begin{figure}[H]
	\centering
	\begin{tabular}{ccc}
		\resizebox{2.1in}{1.4in}{\includegraphics{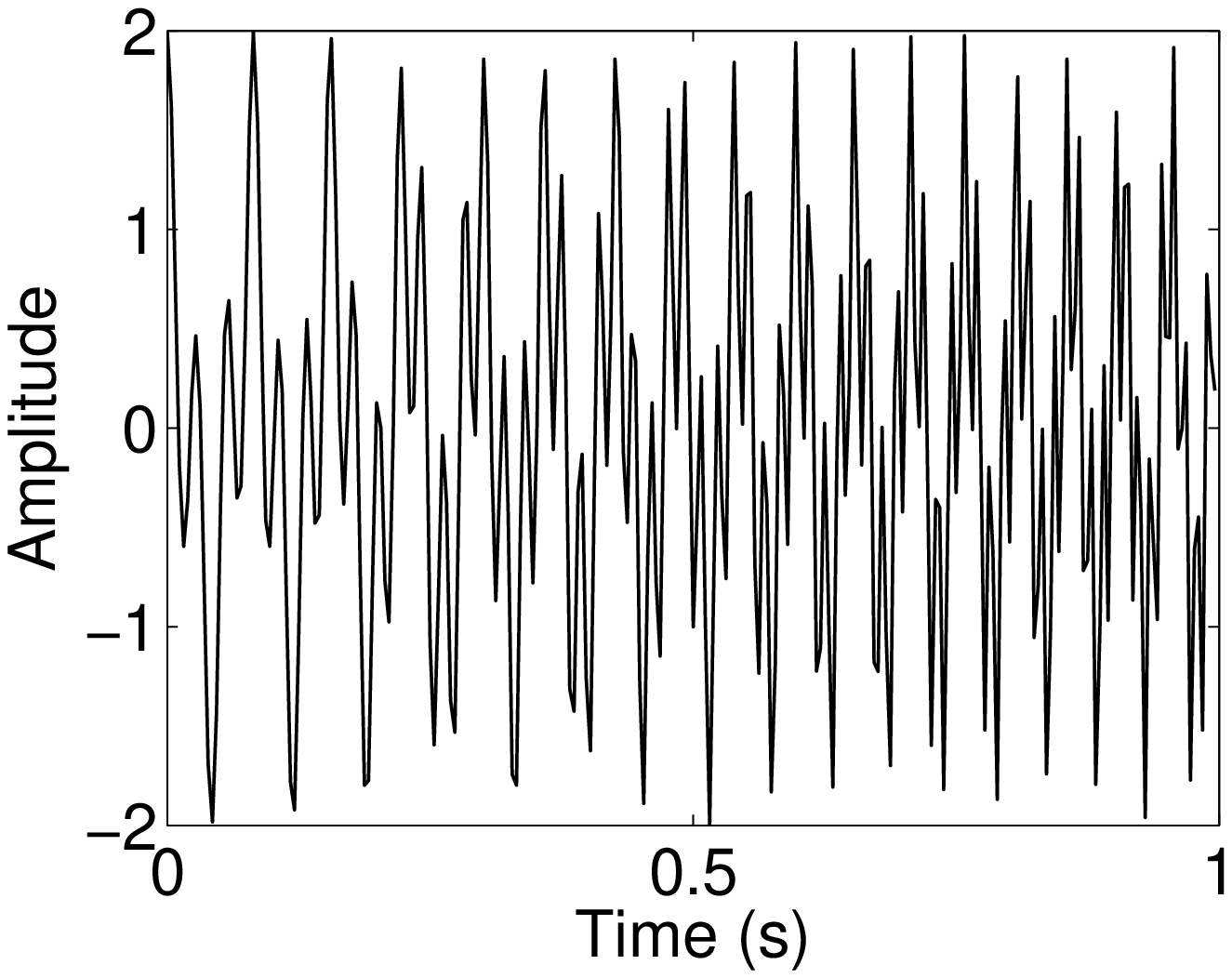}} \quad &
		\resizebox{2.1in}{1.4in}{\includegraphics{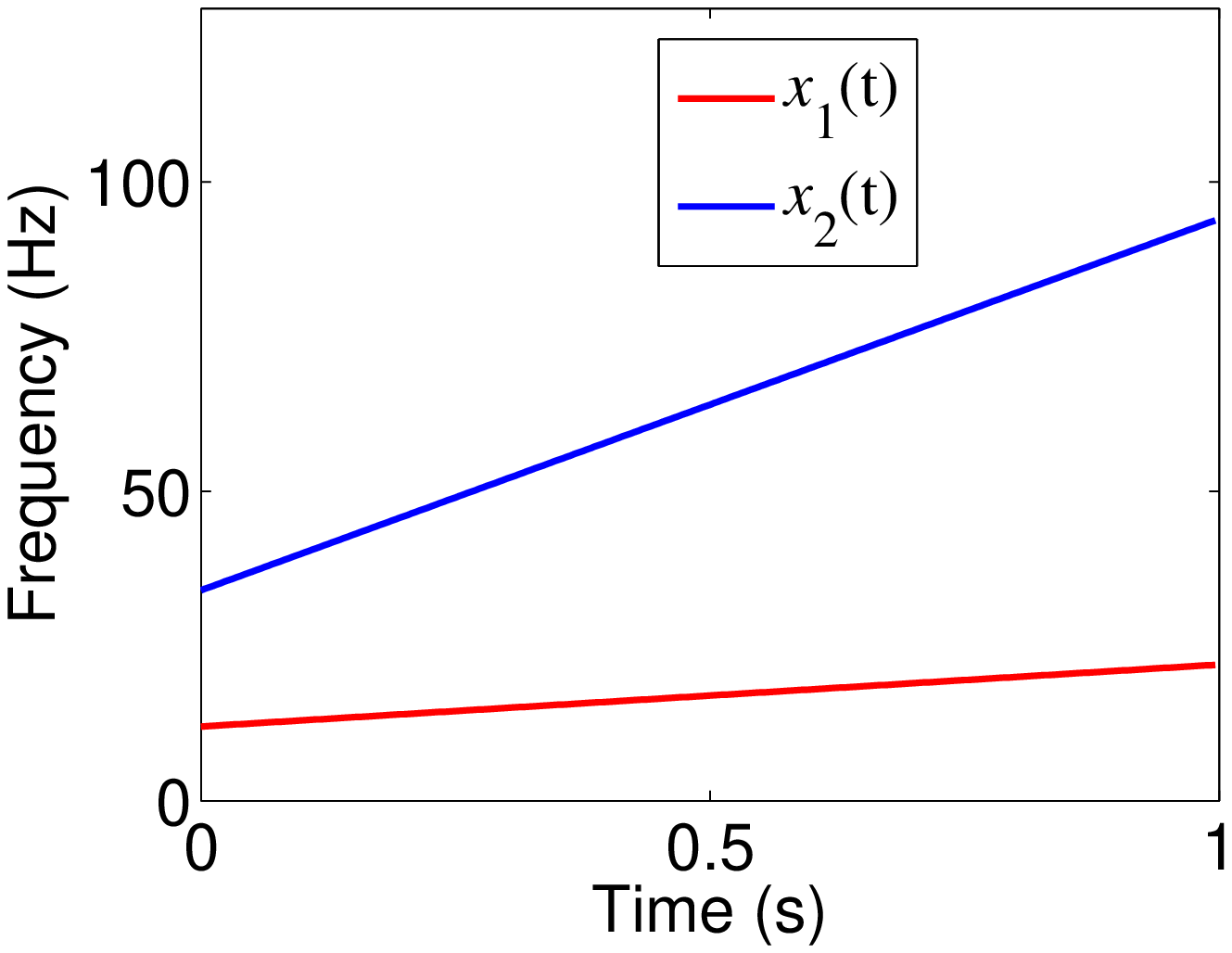}}\quad &
		\resizebox{2.1in}{1.4in}{\includegraphics{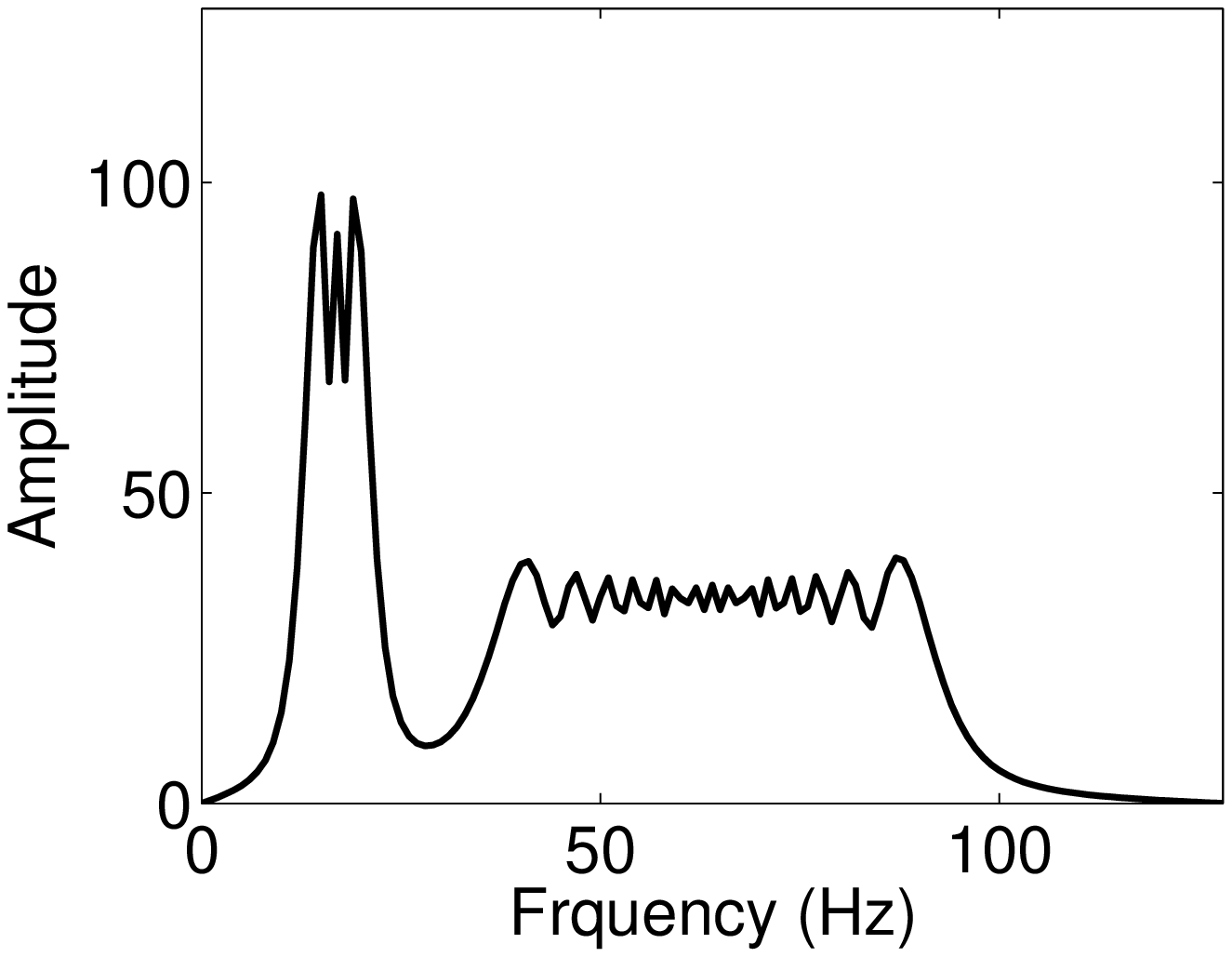}} \\ 
		\resizebox{2.1in}{1.4in}{\includegraphics{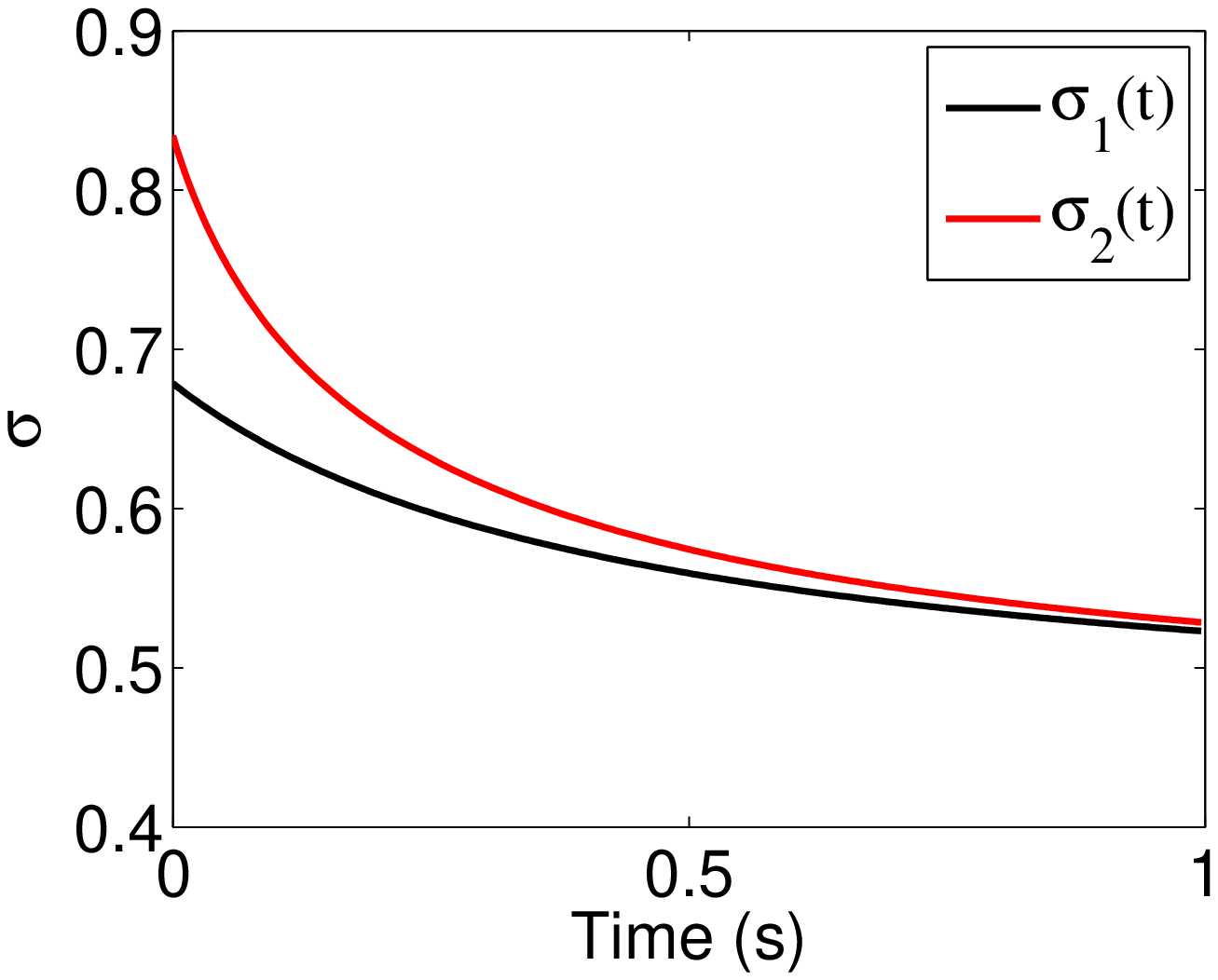}}\quad & 
		\resizebox{2.1in}{1.4in}{\includegraphics{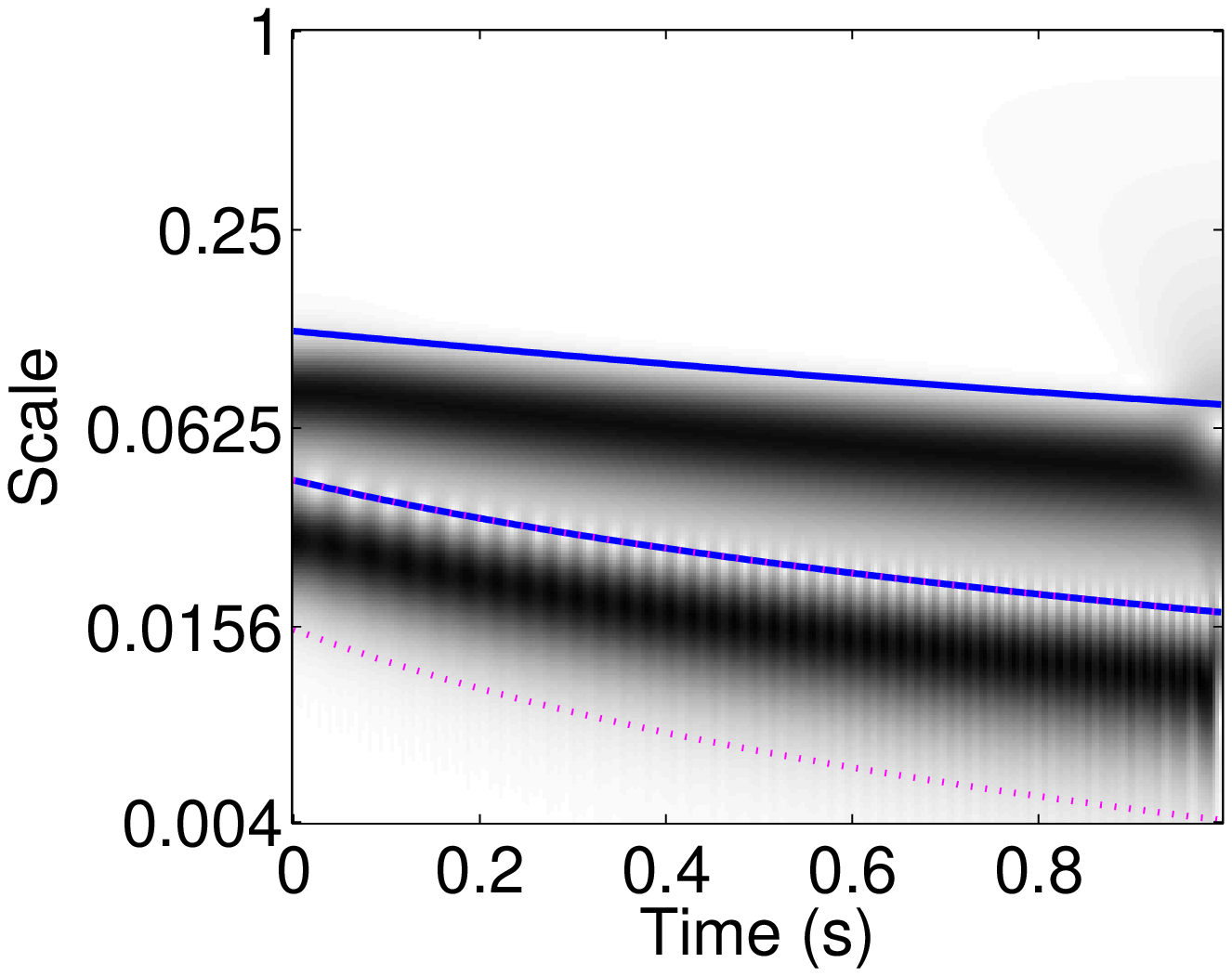}} \quad & 
		\resizebox{2.1in}{1.4in}{\includegraphics{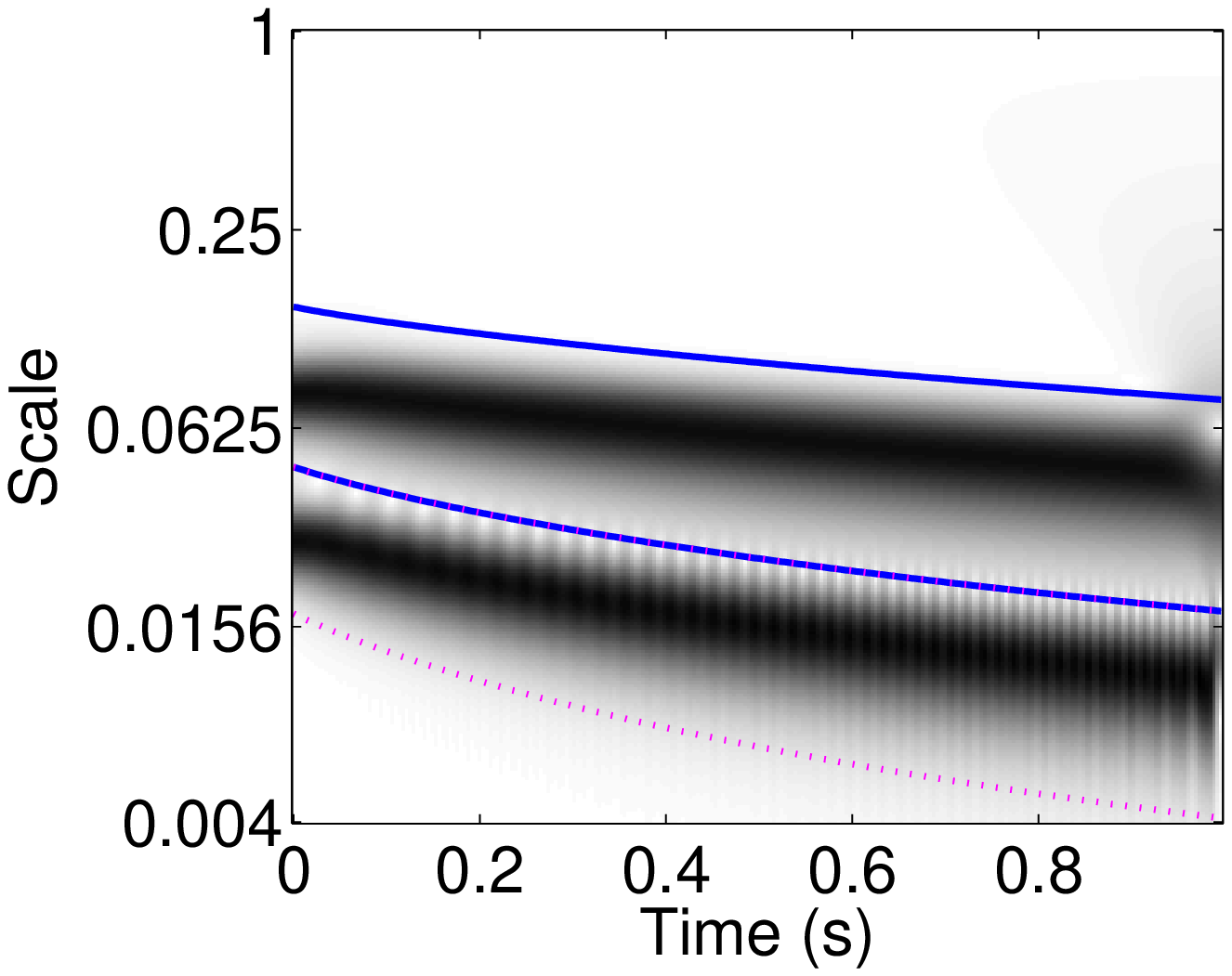}}
	\end{tabular}
	\caption{\small Example of two-component signal. 
		Top-left:  Waveform;  Top-middle: IFs; Top-right: Spectrum; 
		Bottom-left: Optimal parameter $\gs_1(b)$ and $\gs_2(b)$ with sinusoidal signal-based  and linear chirp-based models respectively; Bottom-middle: Adaptive CWLT with $\gs_1(b)$; Bottom-right: Adaptive CWLT with $\gs_2(b)$.}
	\label{fig:two_chirp_signal}
\end{figure}

\begin{figure}[H]
	\centering
	\begin{tabular}{cc}
		\resizebox{2.1in}{1.4in}{\includegraphics{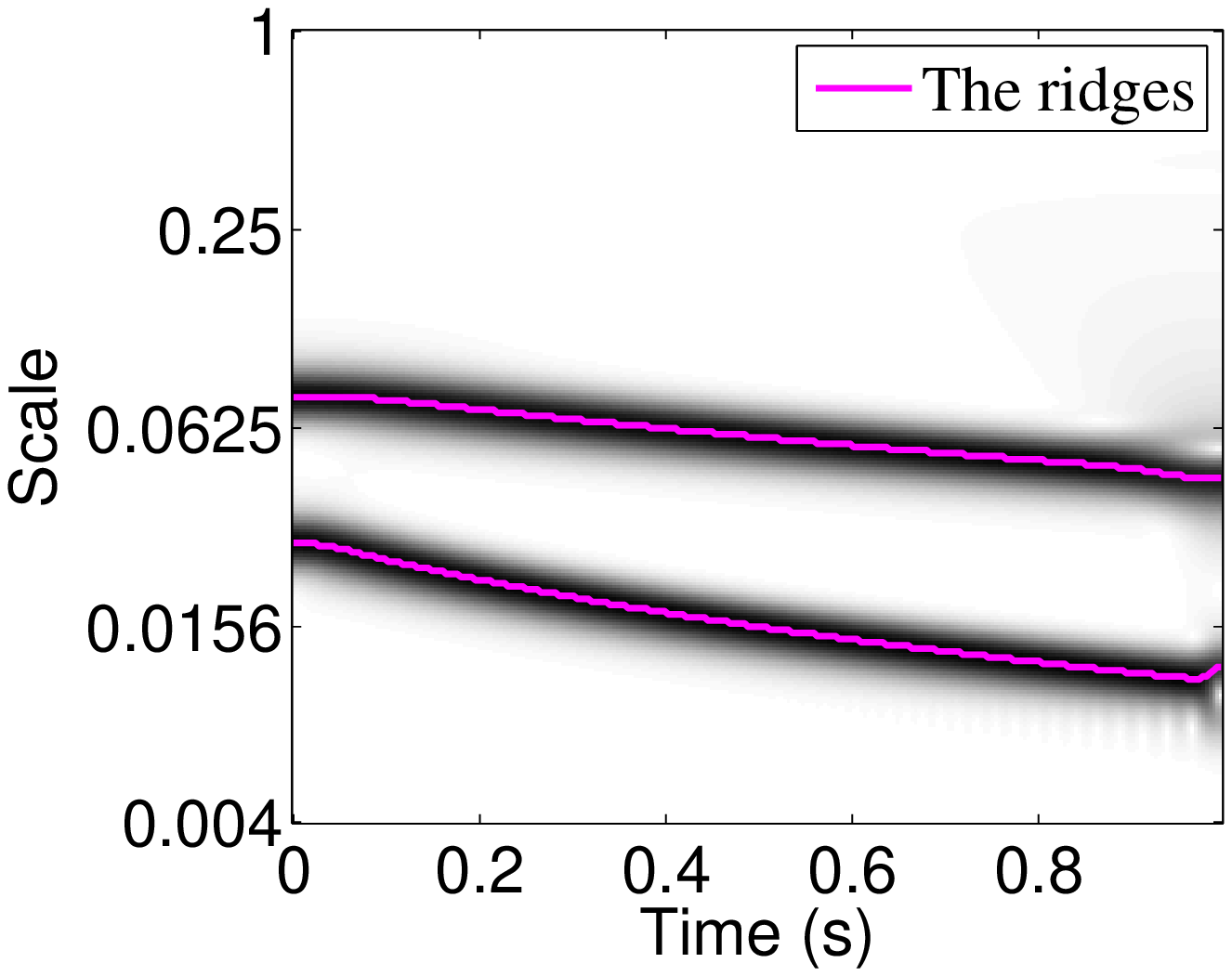}} \quad &
		\resizebox{2.1in}{1.4in}{\includegraphics{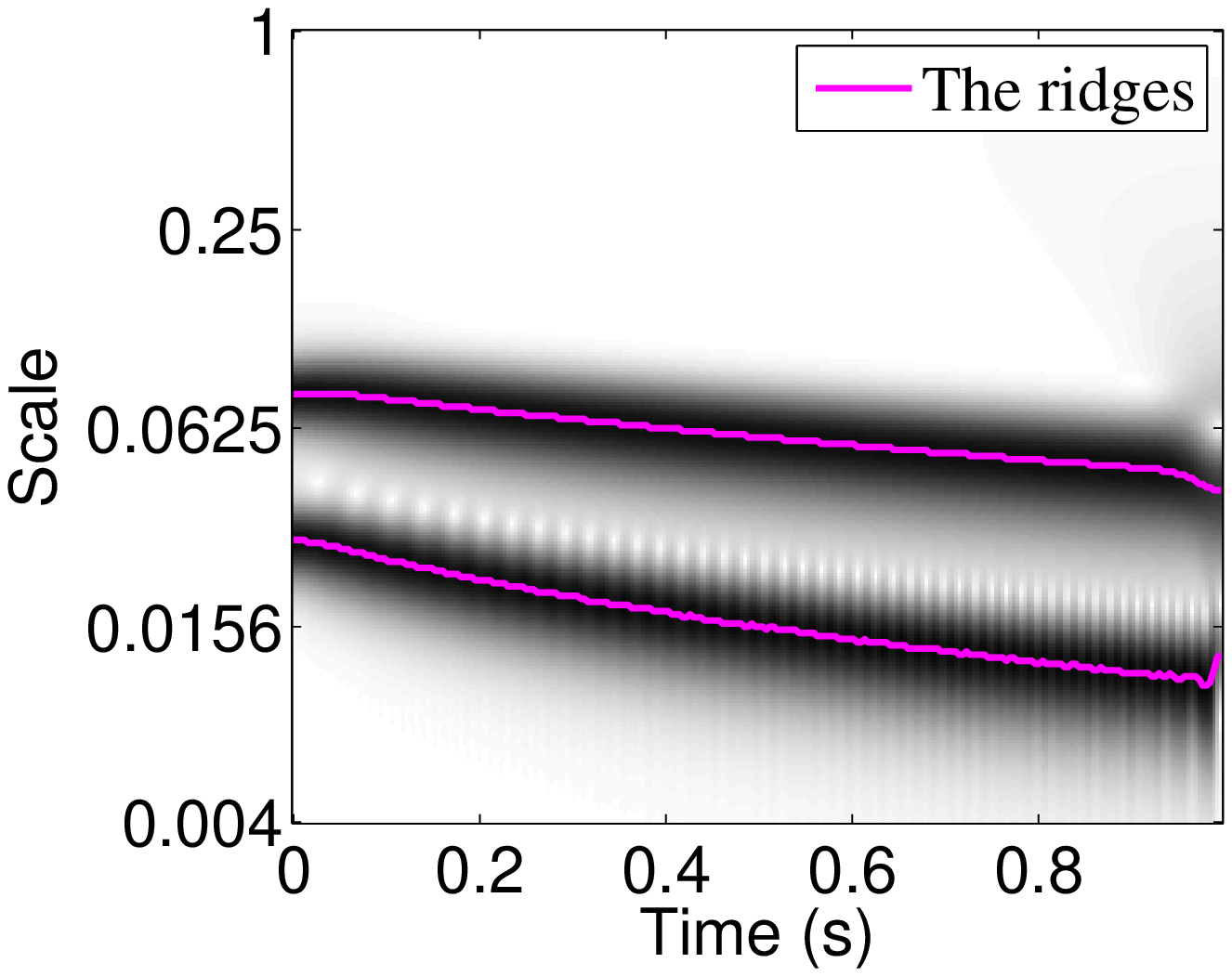}}\\
		\resizebox{2.1in}{1.4in}{\includegraphics{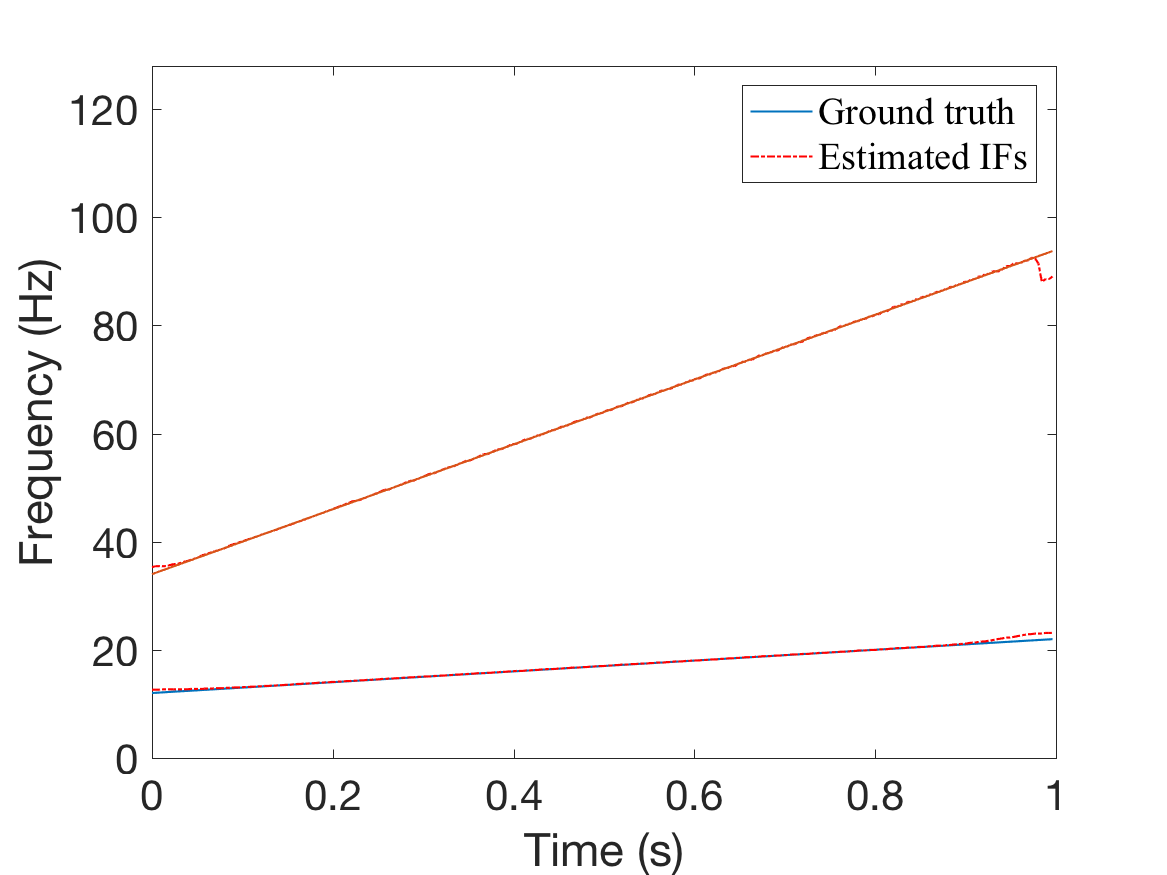}} \quad &
		\resizebox{2.1in}{1.4in}{\includegraphics{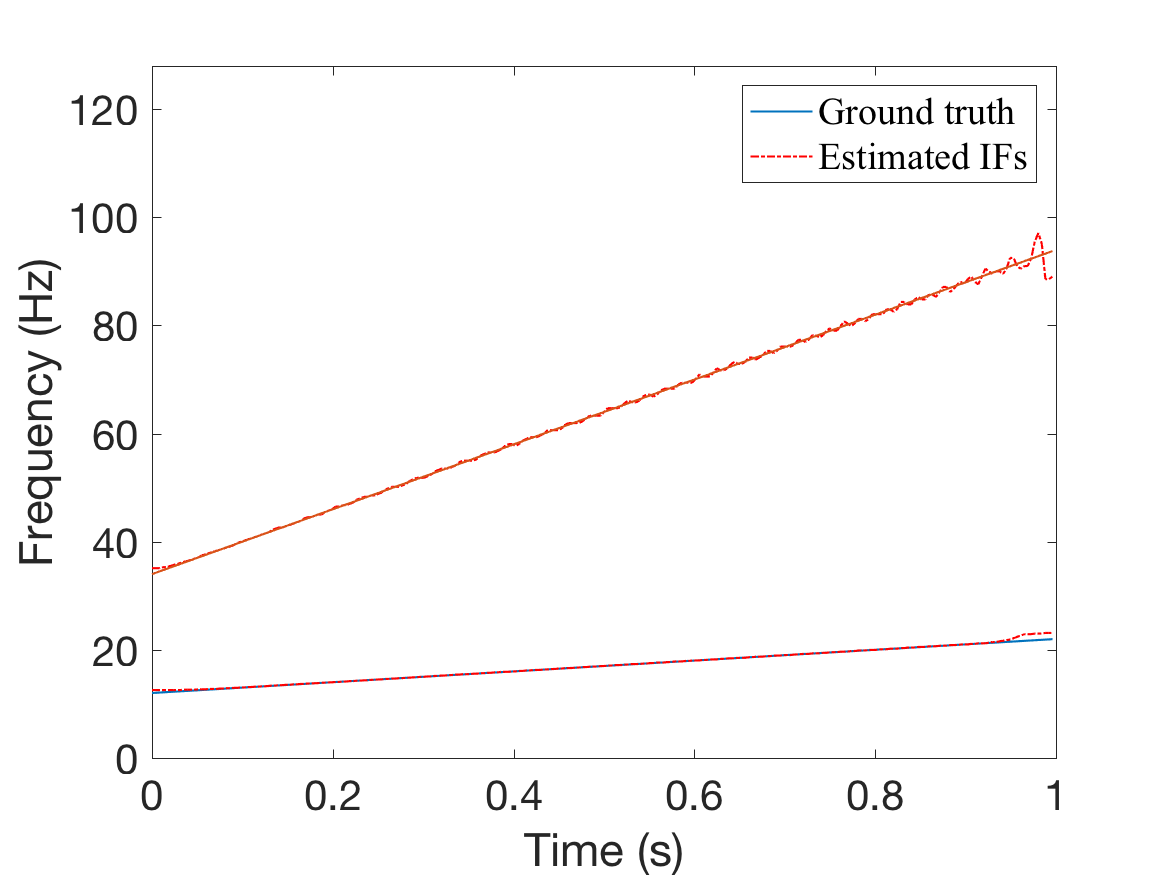}}  
	\end{tabular}
 	\caption{\small IF estimation results of  two-component signal. Top-left:  CWLT ridges %$\wh a_1(b), \wh a_2(b)$ 
 	with $\gs(t)=1$;   Top-right: Adaptive CWLT ridges %$\wc {a}_1(b), \wc{ a}_2(b)$ 
 	with $\gs(t)=\gs_2(t)$;  Bottom row: IF estimates with $\gs(t)=1$ (left panel) and with $\gs(t)=\gs_2(t)$ (right panel). }
 	%$1/\wh a_1(b), 1/\wh a_2(b)$; Bottom-right: IF estimates $1/\wc a_1(b), 1/\wc a_2(b)$.}
\label{fig:two_chirp_IF_estimation}
\end{figure}

First we consider a two-component linear chirp signal, 
\begin{equation}
\label{two_chirps_12_34}
%\begin{aligned}
x(t)=x_1(t)+x_2(t)
= \cos \left(2\pi(12t+5t^2)\right)+ \cos\left(2\pi(34t+30t^2)\right), \quad t\in [0, 1]. 
\end{equation}
The number of sampling points is $N=256$ and the sampling rate is 256Hz.
The IFs of  $x_1(t)$ and $x_2(t)$ are  $\phi'_1(t)=12+10t$ and $\phi'_2(t)=34+60t$, respectively.
Hence, the chirp rates of $x_1(t)$ and $x_2(t)$ are  $\phi_1''(t)=10$ and $\phi''_2(t)=60$, respectively. 
In Fig.\ref{fig:two_chirp_signal}, we show  the waveform of $x(t)$, IFs $\phi'_1$ and $\phi'_2$ and spectrum 
$|\wh x(\xi)|$. From its spectrum, we can say $x(t)$ is a wide-band non-stationary signal, which cannot be separated by Fourier transform.  In Fig.\ref{fig:two_chirp_signal}, we also show optimal parameters $\gs_1(b)$ and $\gs_2(b)$, and the adaptive CWLTs with $\gs_1(b)$ and $\gs_2(b)$.  The blue lines in the bottom-middle panel and bottom-right panel are the boundaries of the time-scale zones of $\wt W_{x_1}(a, b)$  and $\wt W_{x_2}(a, b)$ with $\gs(b)=\gs_1(b)$ and $\gs(b)=\gs_2(b)$ respectively.

\begin{figure}[H]
	\centering
	\begin{tabular}{cc}
		\resizebox{2.1in}{1.4in}{\includegraphics{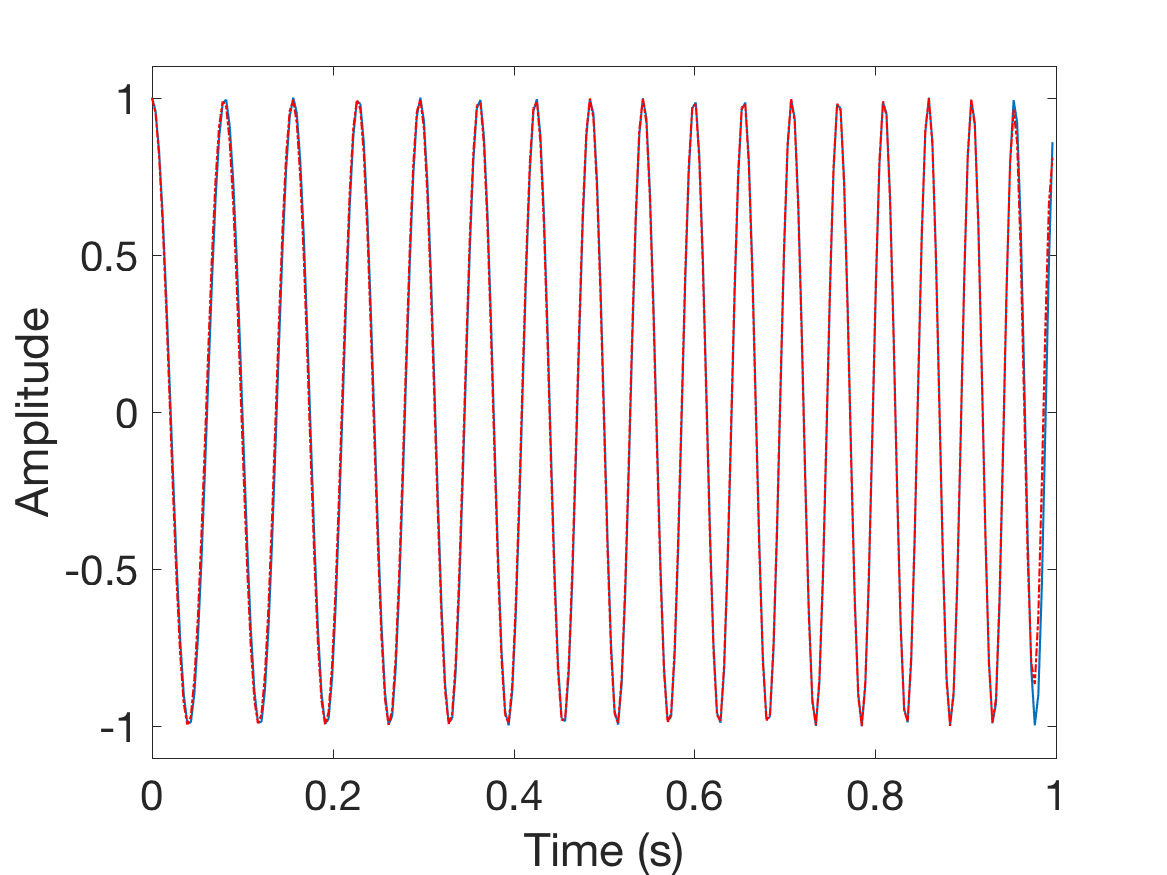}} \quad &
		\resizebox{2.1in}{1.4in}{\includegraphics{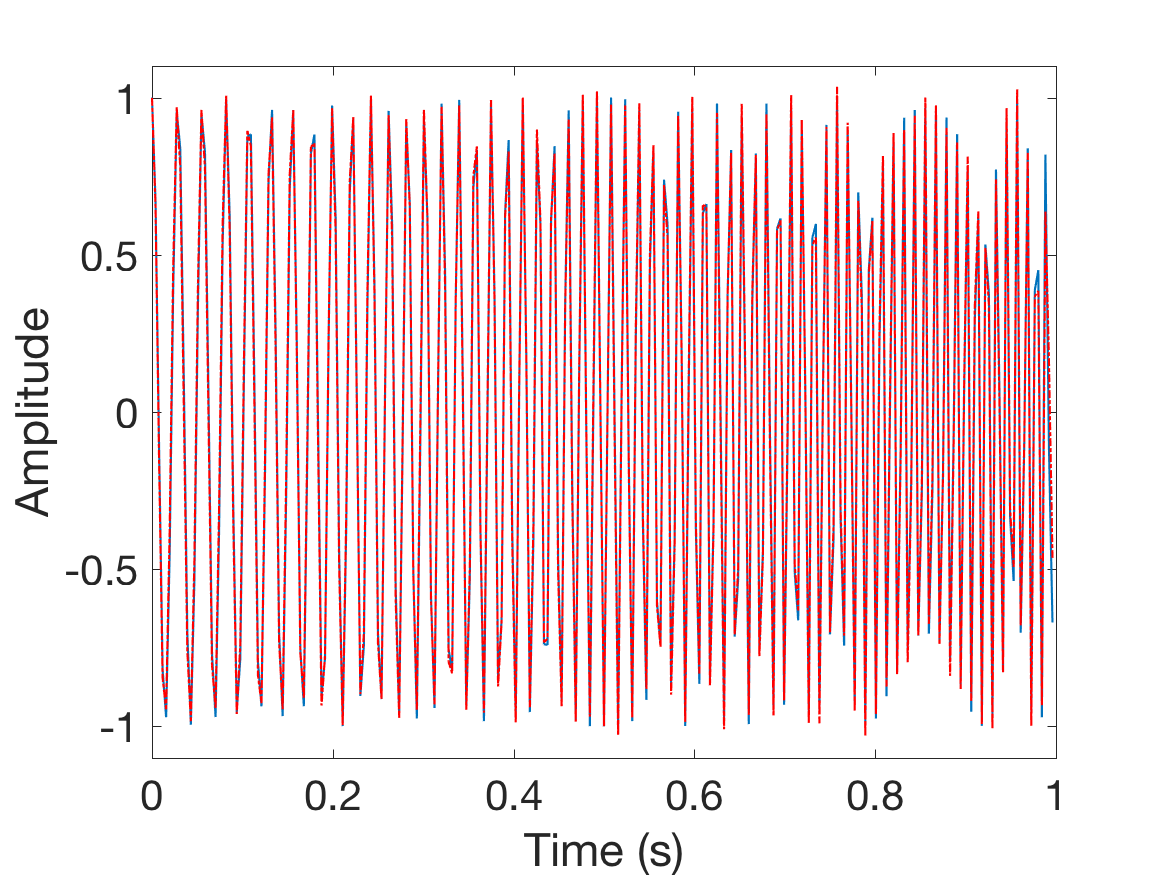}}\\
		\resizebox{2.1in}{1.4in}{\includegraphics{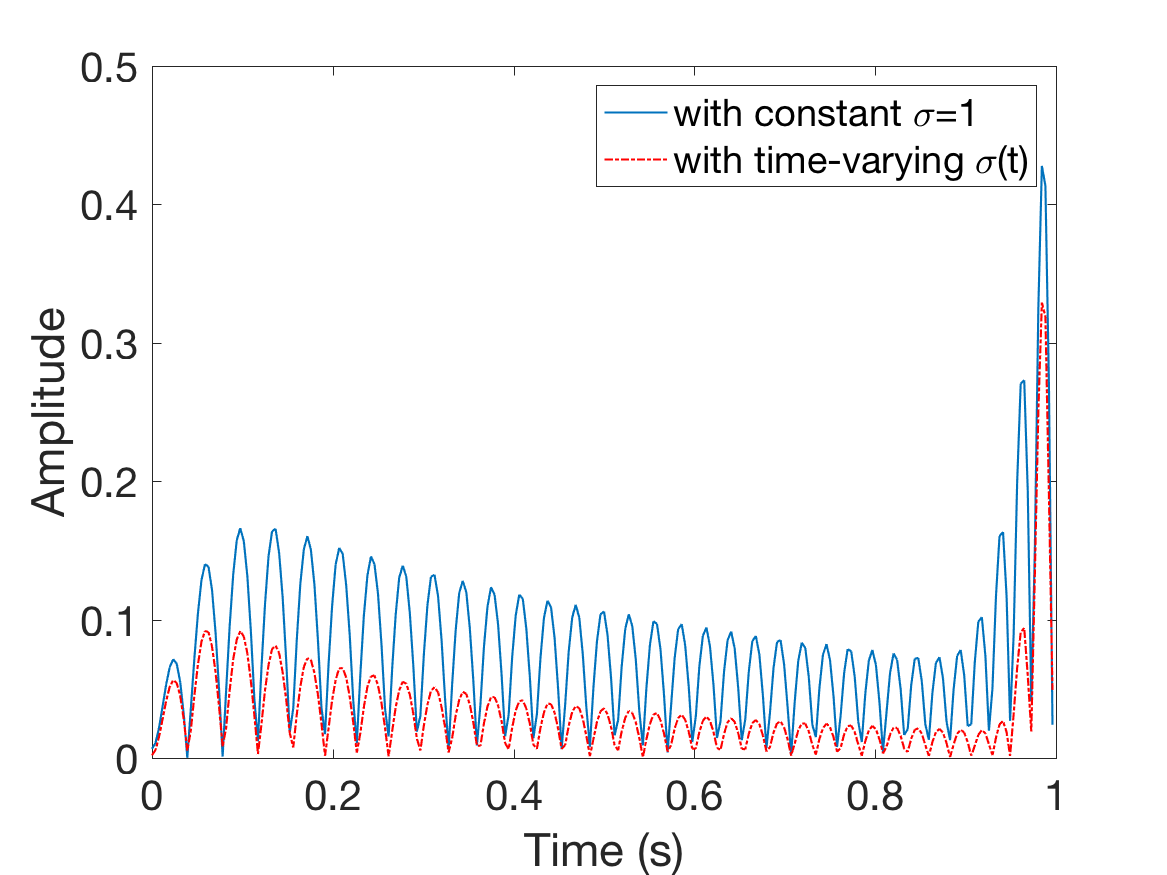}} \quad &
		\resizebox{2.1in}{1.4in}{\includegraphics{ACWT_x1error_twocomponent_gs_via_ci2_sample_1over51_Oct25_2020.png}}  
	\end{tabular}
	\caption{\small Recovery results of two-component signal. Top-left:  $x_1(t)$ and $\wt W_x(\wh a_1, t)$ (recovered $x_1(t)$) with $\gs(t)=\gs_2(t)$;  Top-right: $x_2(t)$ and $\wt W_x(\wh a_2, t)$ (recovered $x_2(t)$)  with $\gs(t)=\gs_2(t)$; Bottom row: Recovery error of $x_1(t)$ (left panel) and $x_2(t)$ (right panel)  with $\gs(t)=1$ and $\gs(t)=\gs_2(t)$.
     }
	\label{fig:two_chirp_recovered}
\end{figure}

In Fig.\ref{fig:two_chirp_IF_estimation}, we show the IF estimation results of the two-component signal $x(t)$ by our method. More precisely, in the top row of Fig.\ref{fig:two_chirp_IF_estimation} are $\wh a_1(t), \wh a_2(t)$
% ({\bf Lin, it is correct: $\wc a_k(t)$, not $\wh a_k(t)$?)} ({\bf Prof. Jiang: I use $\wh a_k(t)$})
with $\gs(t)=1$ (left panel) and %$\wc a_1(t), \wc a_2(t)$
with $\gs(t)=\gs_2(t)$ (right panel), 
while  $1/\wh a_1(t), 1/\wh a_2(t)$ with $\gs(t)=1$  and $\gs(t)=\gs_2(t)$ 
%and $1/\wc a_1(t), 1/\wc a_2(t)$ 
are presented in the bottom-left panel and bottom-right panel respectively. Clearly,  
for the choice of $\gs(t)=1$  or for $\gs(t)=\gs_2(t)$, 
$1/\wh a_1(t)$ and  $1/\wh a_2(t)$ %or $1/\wc a_1(t), 1/\wc a_2(t)$ 
give nice estimates to $\phi'_1(t)$ and $\phi'_2(t)$ respectively.

In Fig.\ref{fig:two_chirp_recovered}, we provide component recovery results of the two-component signal. 
In top-left panel, we show  $x_1(t)$ and $\wt W_x(\wh a_1, t)$, recovered $x_1(t)$, with $\gs(t)=\gs_2(t)$;  while in the top-right panel we present $x_2(t)$ and $\wt W_x(\wh a_2, t)$, recovered $x_2(t)$,  with $\gs(t)=\gs_2(t)$. In the bottom row of Fig.\ref{fig:two_chirp_recovered} are 
recovery error $|x_1(t)-\wt W_x(\wh a_1, t)|$  (left panel) and $|x_2(t)-\wt W_x(\wh a_2, t)|$  (right panel) with $\gs(t)=1$ (blue line) and $\gs(t)=\gs_2(t)$ (red dash-dotted line). Observe that  the recovery error with $\gs(t)=\gs_2(t)$ is much smaller than that with $\gs(t)=1$.

Next let look at the performance of the  linear chirp-based model in component recovery. In the top-left panel of Fig.\ref{fig:two_chirp_recovered_LFM}, we show the recovery error for $x_1(t)$ by  
the linear chirp-based model with $\gs=\gs_2(t)$ and ground truth $\phi^{\gp \gp}_1(t)$ in \eqref{comp_xk_est_2nd1_real_gaussian}, while the  recovery error for $x_2(t)$ is provided in the top-right panel of Fig.\ref{fig:two_chirp_recovered_LFM}. In practice, we do not know $\phi^{\gp \gp}_\ell(t)$. However, we may apply a numerical  algorithm to $\wh \mu/a_\ell(t_m)$ or  $\mu/\wc a_\ell(t_m)$ to obtain an estimate of  $\phi^{\gp\gp}_\ell(t_m)$, where 
$t_m, m=0,1, \cdots, $ are the sample points of the time $t$.
 Here we use a five-point formula for differentiation (see, e.g. \cite{BF11_book})  to  %$\wc \eta_\ell(t_m),  0\le m\le N-1$ to 
obtain an approximation to $\phi^{\gp\gp}_\ell(t_m)$. In the bottom row  of Fig.\ref{fig:two_chirp_recovered_LFM}, we show the recovery errors for $x_1(t)$ (left panel) and $x_2(t)$ (right panel) by linear chirp-based model with 
an estimated $\phi^{\gp\gp}_1(t), \phi^{\gp\gp}_2(t),$ in \eqref{comp_xk_est_2nd1_real_gaussian} (red dash-dotted line). 
From Fig.\ref{fig:two_chirp_recovered_LFM}, we see the linear chirp-based model leads more accurate component recovery. In particular the error for $x_1(t)$ is almost zero (except near the endpoints). 
\begin{figure}[H]
	\centering
	\begin{tabular}{cc}
		\resizebox{2.1in}{1.4in}{\includegraphics{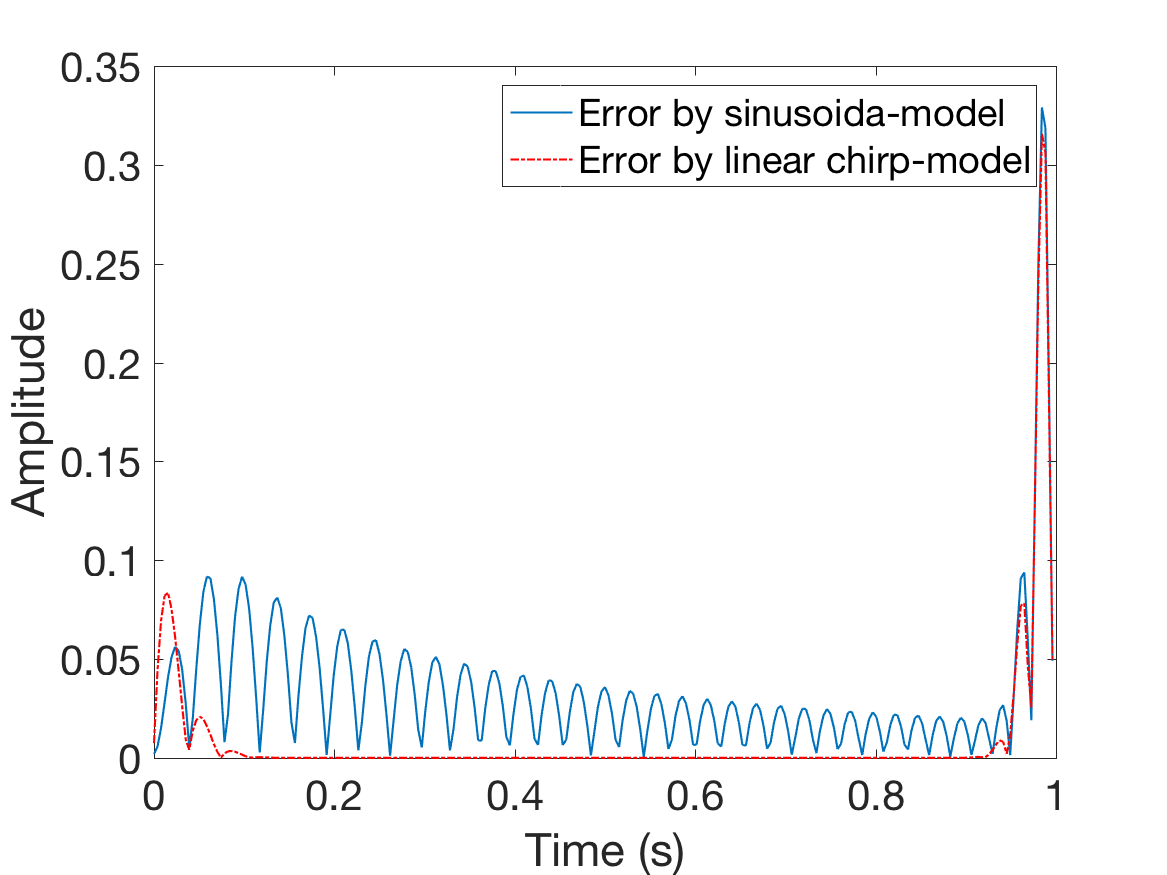}} \quad &
\resizebox{2.1in}{1.4in}{\includegraphics{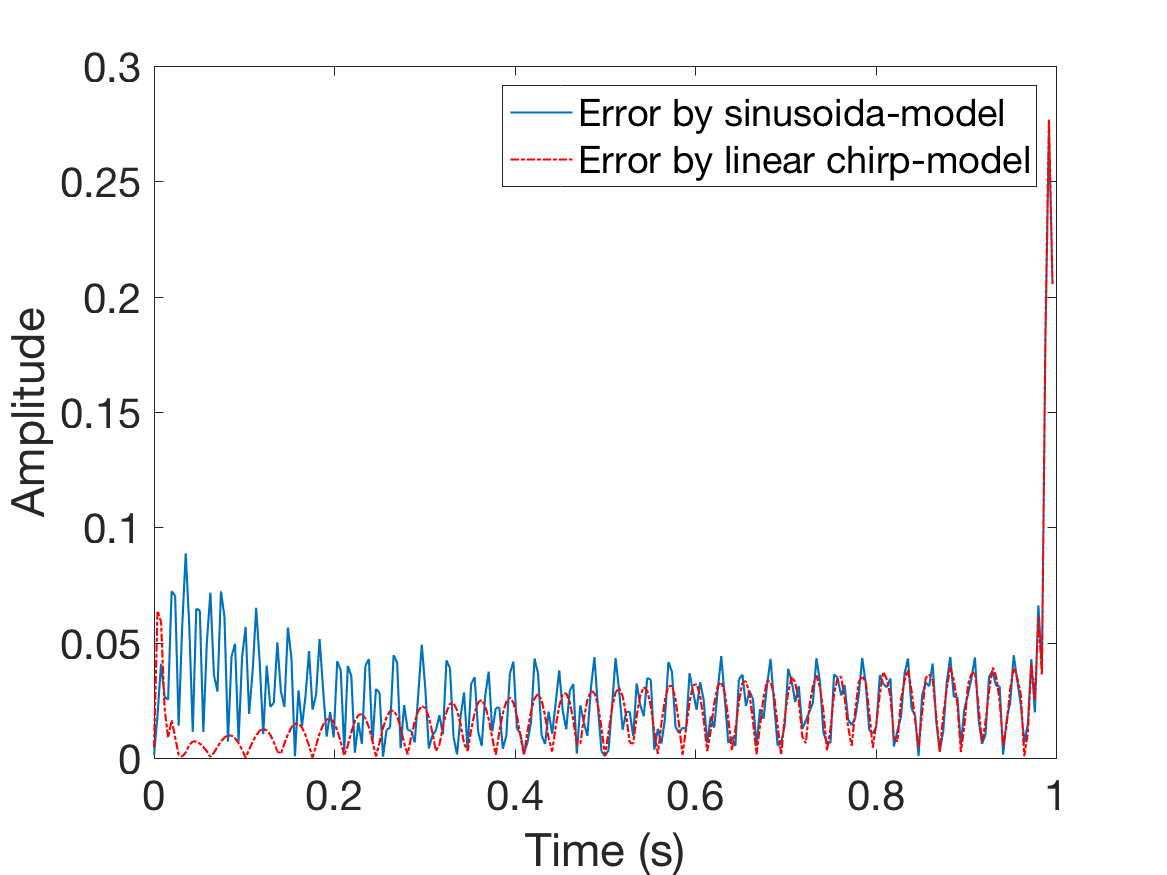}}
		\\
\resizebox{2.1in}{1.4in}{\includegraphics{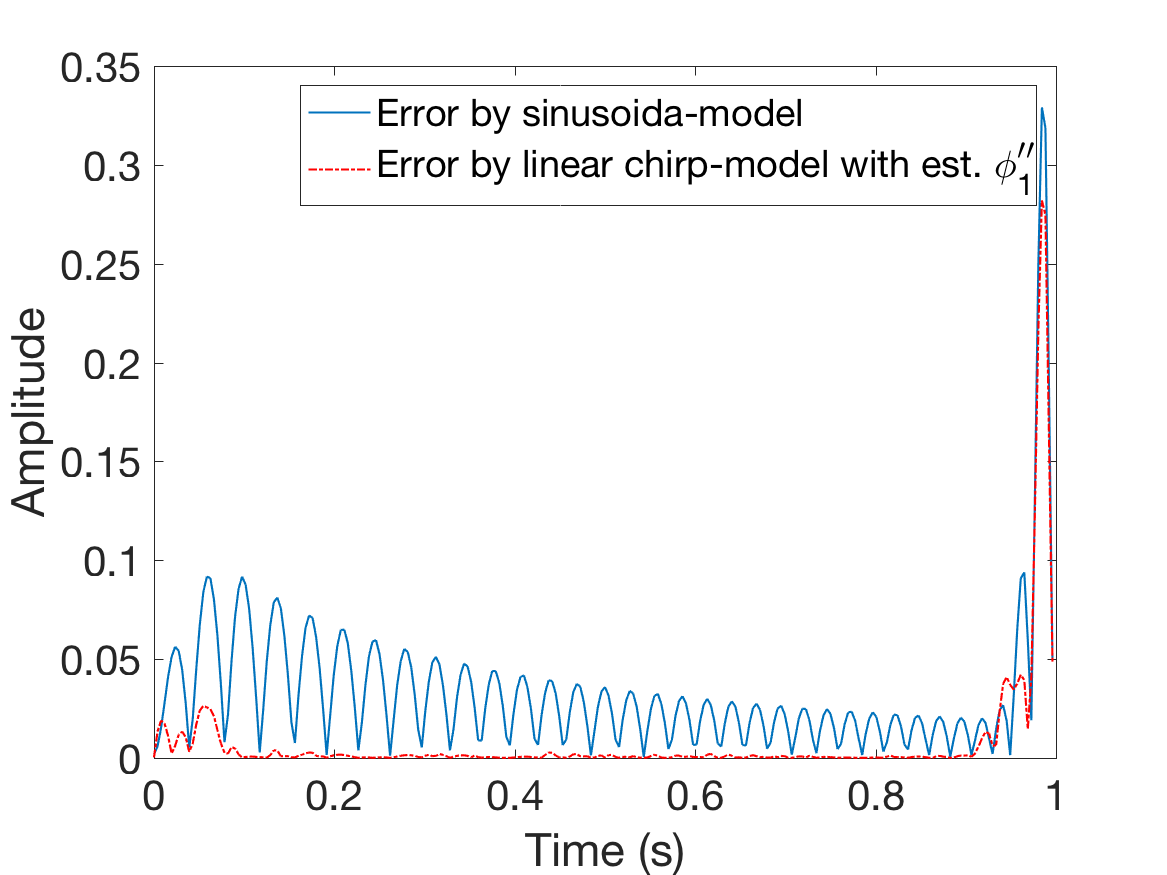}} \quad &\resizebox{2.1in}{1.4in}{\includegraphics{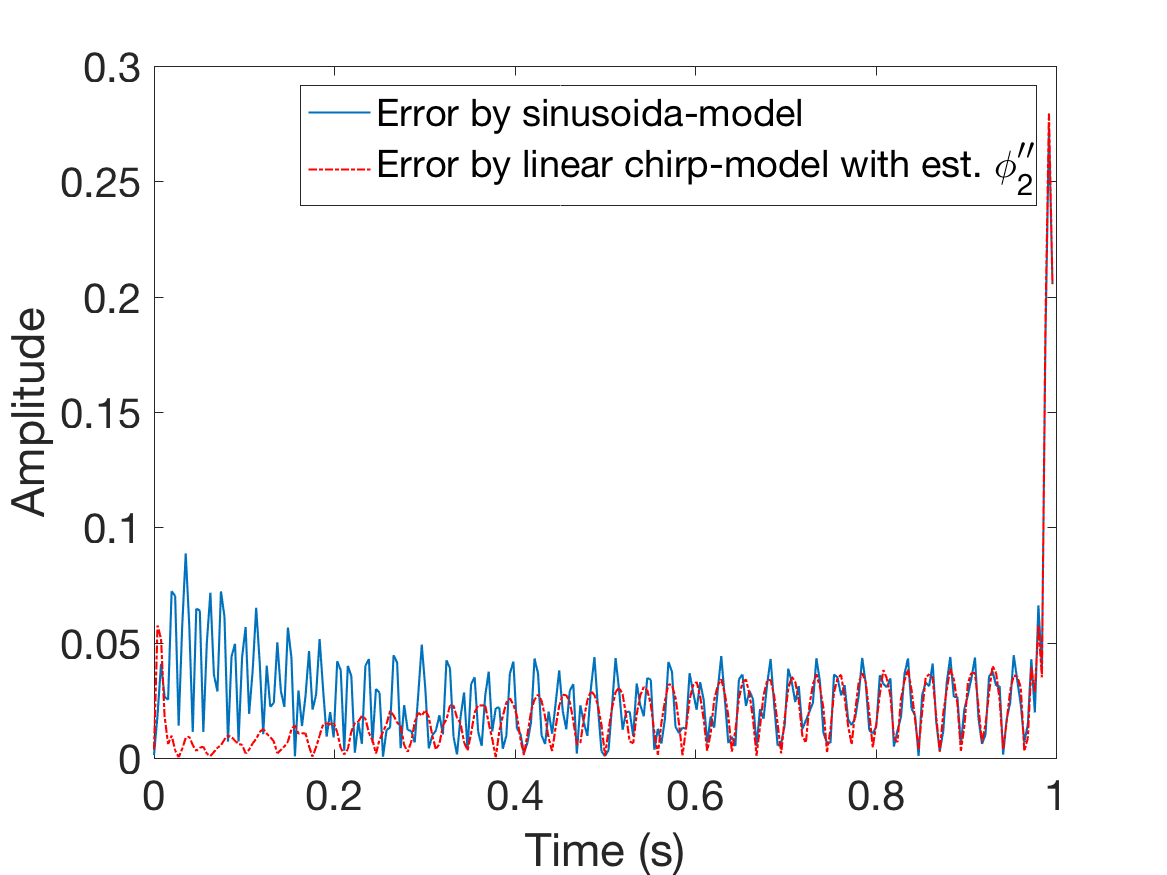}}
	\end{tabular}
	\caption{\small Recovery errors of two-component signal. Top row:  recovery errors by  sinusoidal signal-based model (blue line) and by linear chirp-based model with ground truth $\phi^{\gp\gp}_\ell(t)$ in \eqref{comp_xk_est_2nd1_real_gaussian} (red dash-dotted line) for  $x_1(t)$ (left panel) and for $x_2(t)$ (right panel); Bottom row:  recovery errors by  sinusoidal signal-based model (blue line) and by linear chirp-based model with 
an estimated $\phi^{\gp\gp}_\ell(t)$ in \eqref{comp_xk_est_2nd1_real_gaussian} (red dash-dotted line) for  $x_1(t)$ (left panel) and for $x_2(t)$ (right panel)
     }
	\label{fig:two_chirp_recovered_LFM}
\end{figure}

We also consider these two models in a noise environment. Two white Gaussian noises with SNR to be 20dB and 15dB respectively are added to the two-component signal  $x(t)$ given by \eqref{two_chirps_12_34}.  Fig.\ref{fig:two_chirp_recovered_LFM_noise20dB} and Fig.\ref{fig:two_chirp_recovered_LFM_noise15dB}
show the errors of component recovery. Again in these two cases the linear chirp-based model performs better.

In the following we consider a three-component signal with one harmonic and two nonlinear frequency modulation modes,
	\begin{equation*}
	\label{three_nonlinear_FM}
	\begin{array}{l}
	y(t)=y_1(t)+y_2(t)+y_3(t)
	\\ \qquad 
	%=\cos \left(6\pi t\right)+ \frac 23\cos \big(9.6\pi t+0.4\cos(3\pi t)\big) +\frac 12\cos \big(14.8\pi t+0.3\cos(3\pi t)\big), 
	=\cos \left(60\pi t\right)+ \frac 23\cos \big(96\pi t+4\cos(30\pi t)\big) +\frac 12\cos \big(148\pi t+3\cos(30\pi t)\big), 
	\quad t\in [0, 10]. 
	\end{array}
	\end{equation*}
%where $ t\in [0, 10]$.
 The number of sampling points is $N=512$, namely sampling rate is $F_s=51.2$ Hz.

\begin{figure}[H]
	\centering
	\begin{tabular}{cc}
		\resizebox{2.1in}{1.4in}{\includegraphics{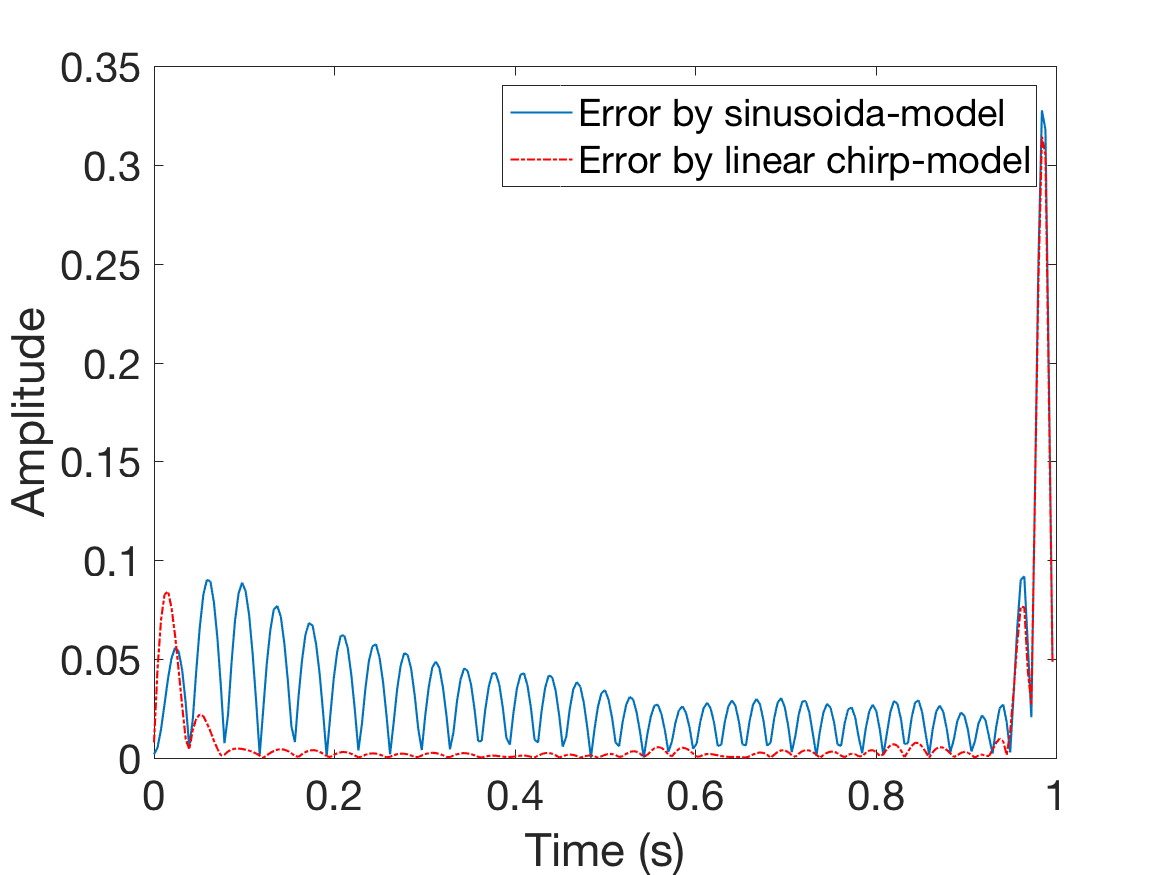}} \quad &
\resizebox{2.1in}{1.4in}{\includegraphics{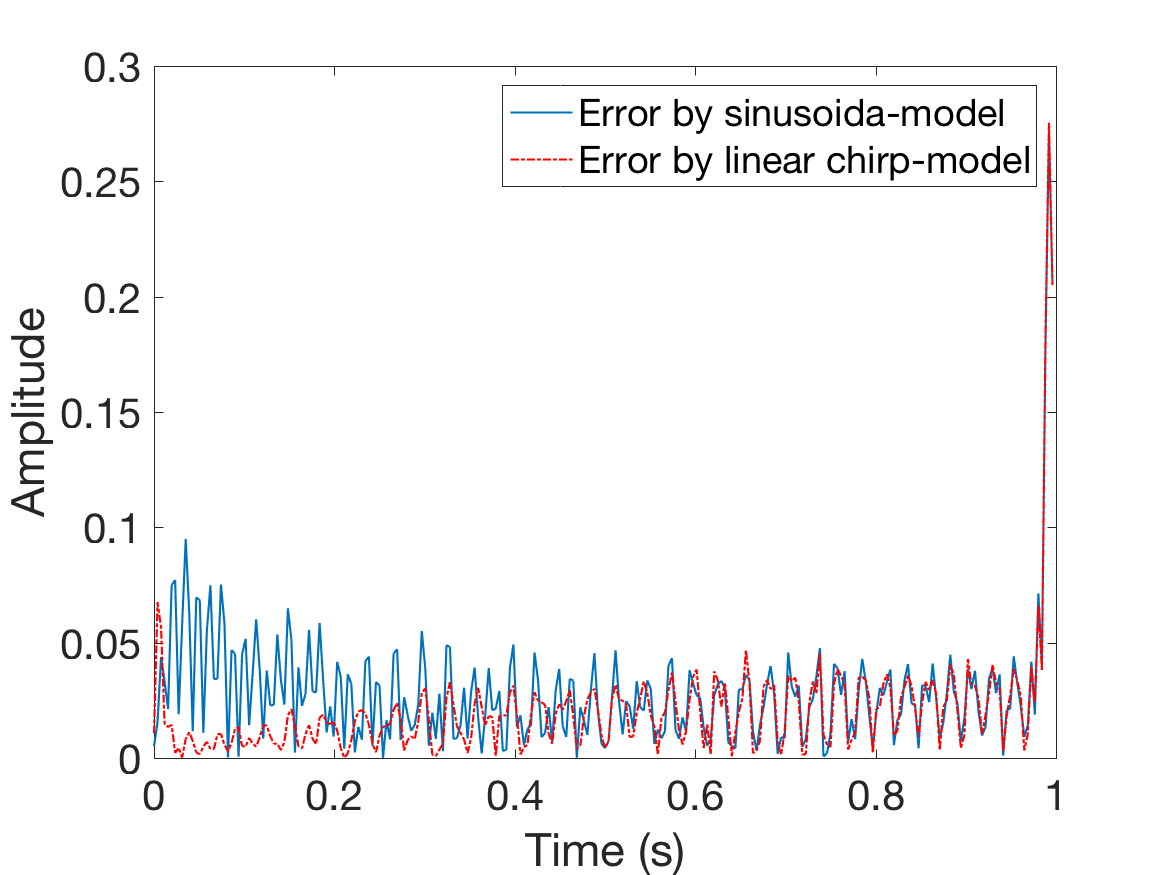}}
		\\
\resizebox{2.1in}{1.4in}{\includegraphics{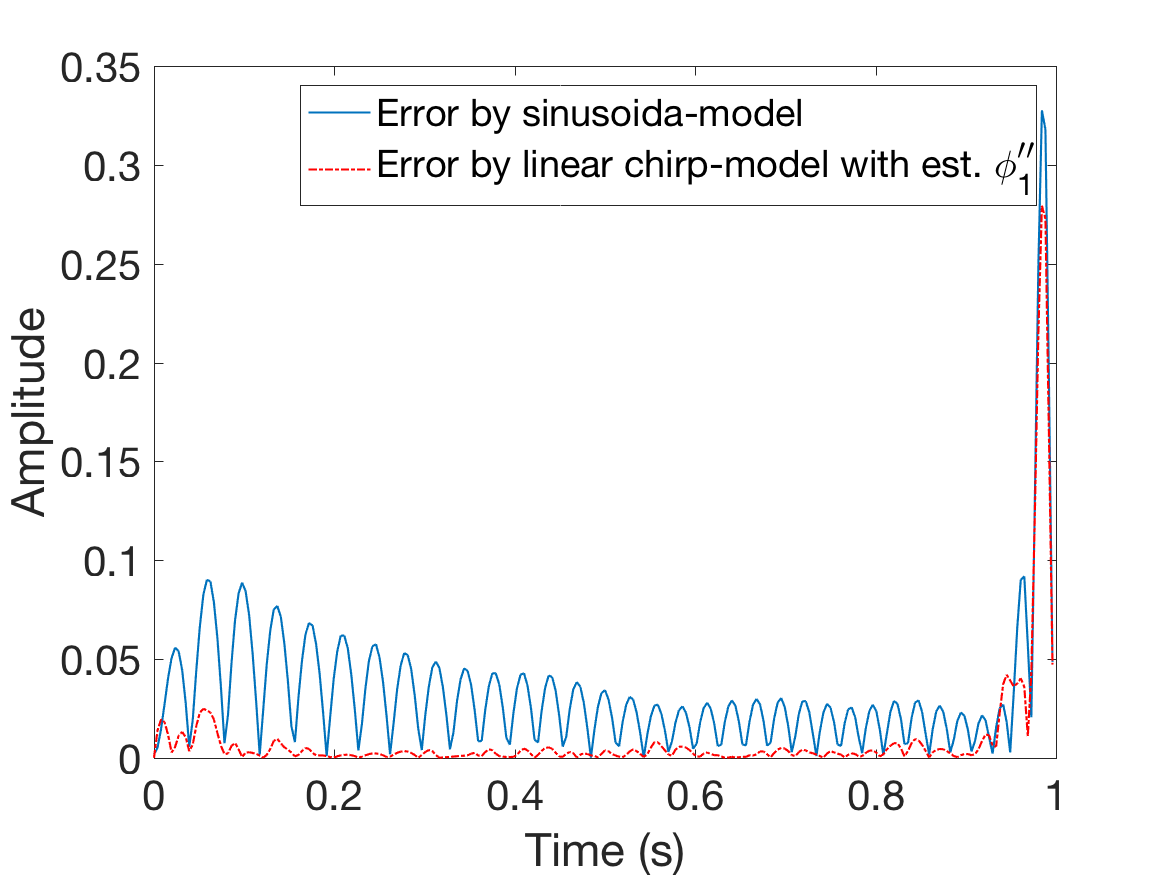}} \quad &\resizebox{2.1in}{1.4in}{\includegraphics{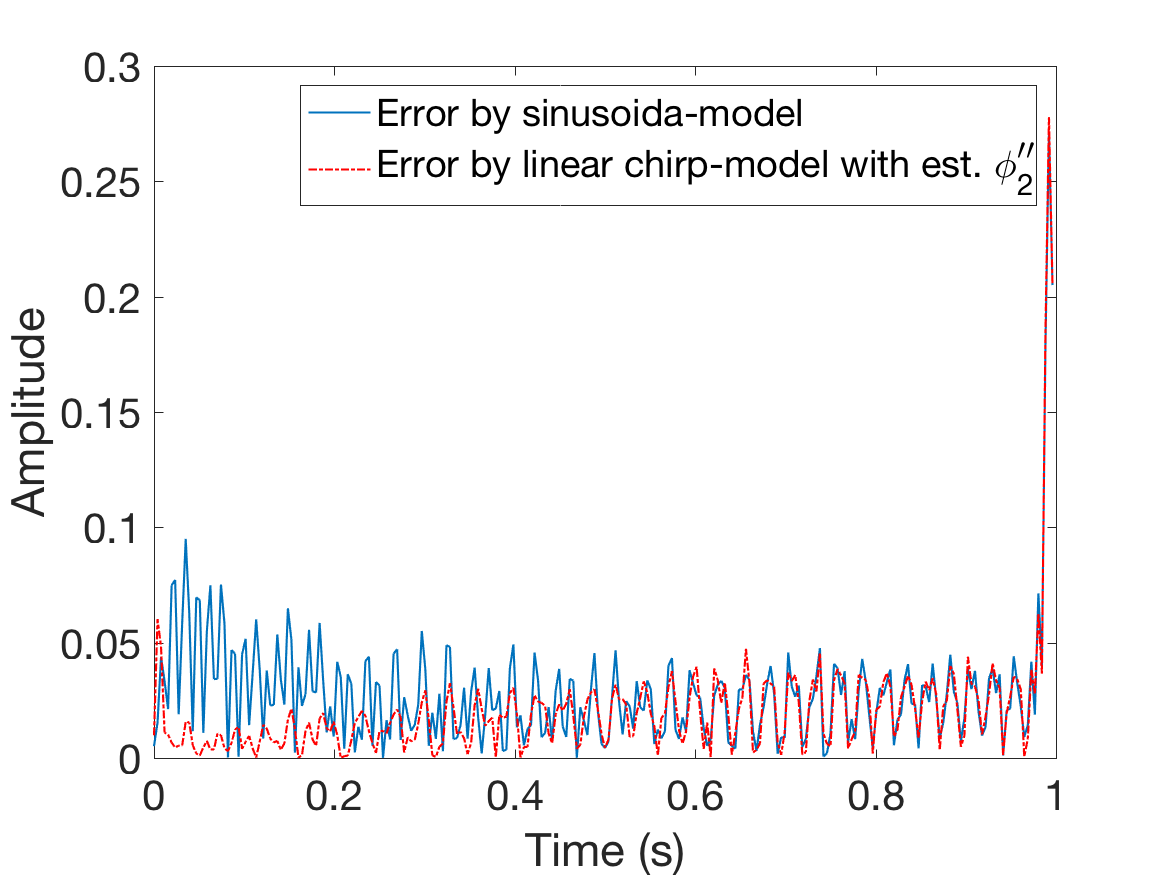}}
	\end{tabular}
	\caption{\small Recovery errors of two-component signal with noise SNR=20dB. Top row:  recovery errors by  sinusoidal signal-based model (blue line) and by linear chirp-based model with ground truth $\phi^{\gp\gp}_\ell(t)$ in \eqref{comp_xk_est_2nd1_real_gaussian} (red dash-dotted line) for  $x_1(t)$ (left panel) and for $x_2(t)$ (right panel); Bottom row:  recovery errors by  sinusoidal signal-based model (blue line) and by linear chirp-based model with 
an estimated $\phi^{\gp\gp}_\ell(t)$ in \eqref{comp_xk_est_2nd1_real_gaussian} (red dash-dotted line) for  $x_1(t)$ (left panel) and for $x_2(t)$ (right panel)
     }
	\label{fig:two_chirp_recovered_LFM_noise20dB}
\end{figure}

\begin{figure}[H]
	\centering
	\begin{tabular}{cc}
		\resizebox{2.1in}{1.4in}{\includegraphics{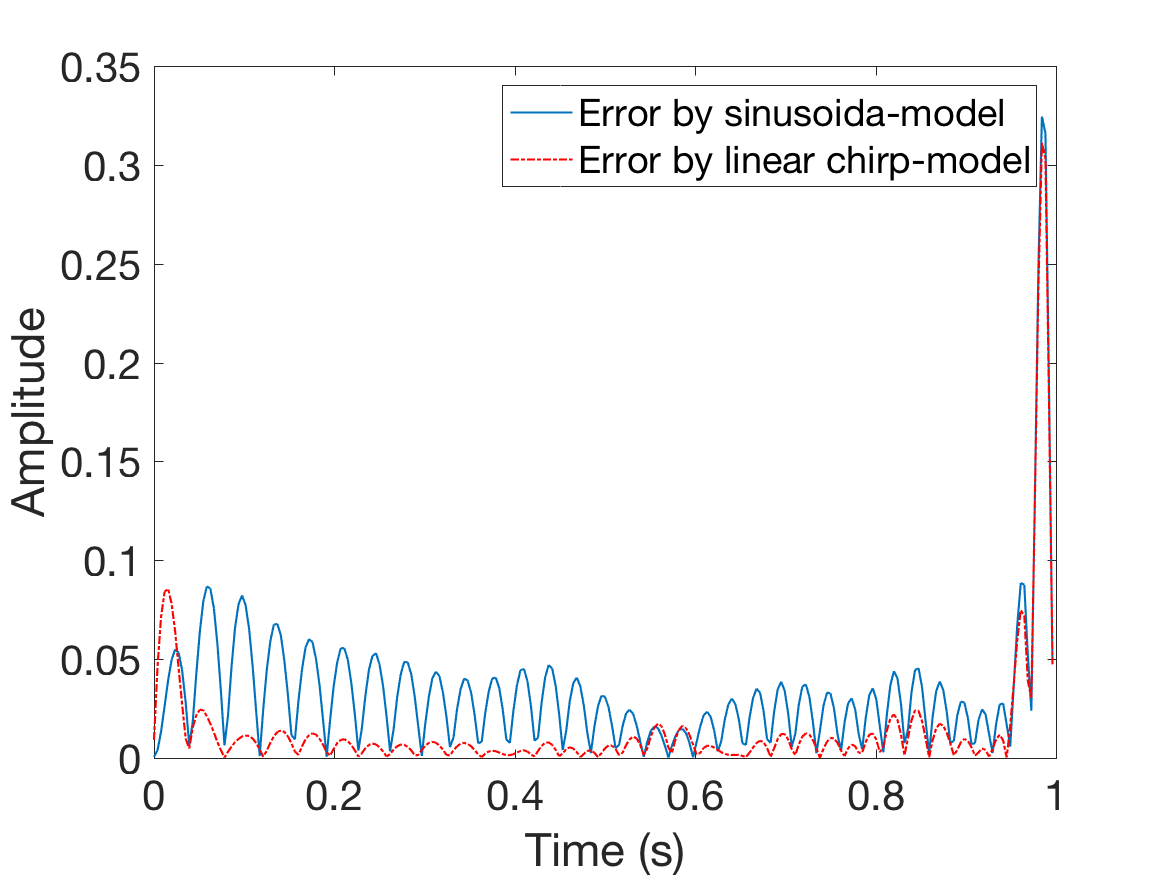}} \quad &
\resizebox{2.1in}{1.4in}{\includegraphics{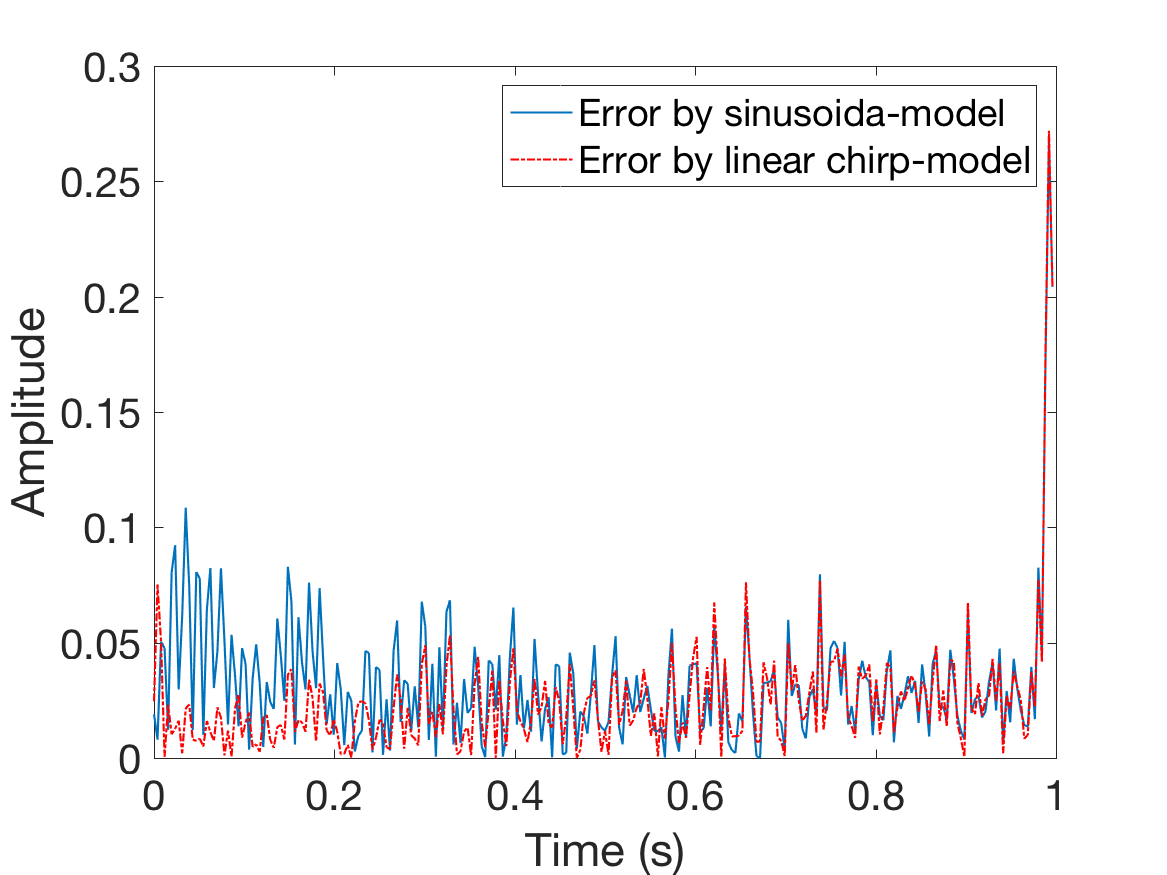}}
		\\
\resizebox{2.1in}{1.4in}{\includegraphics{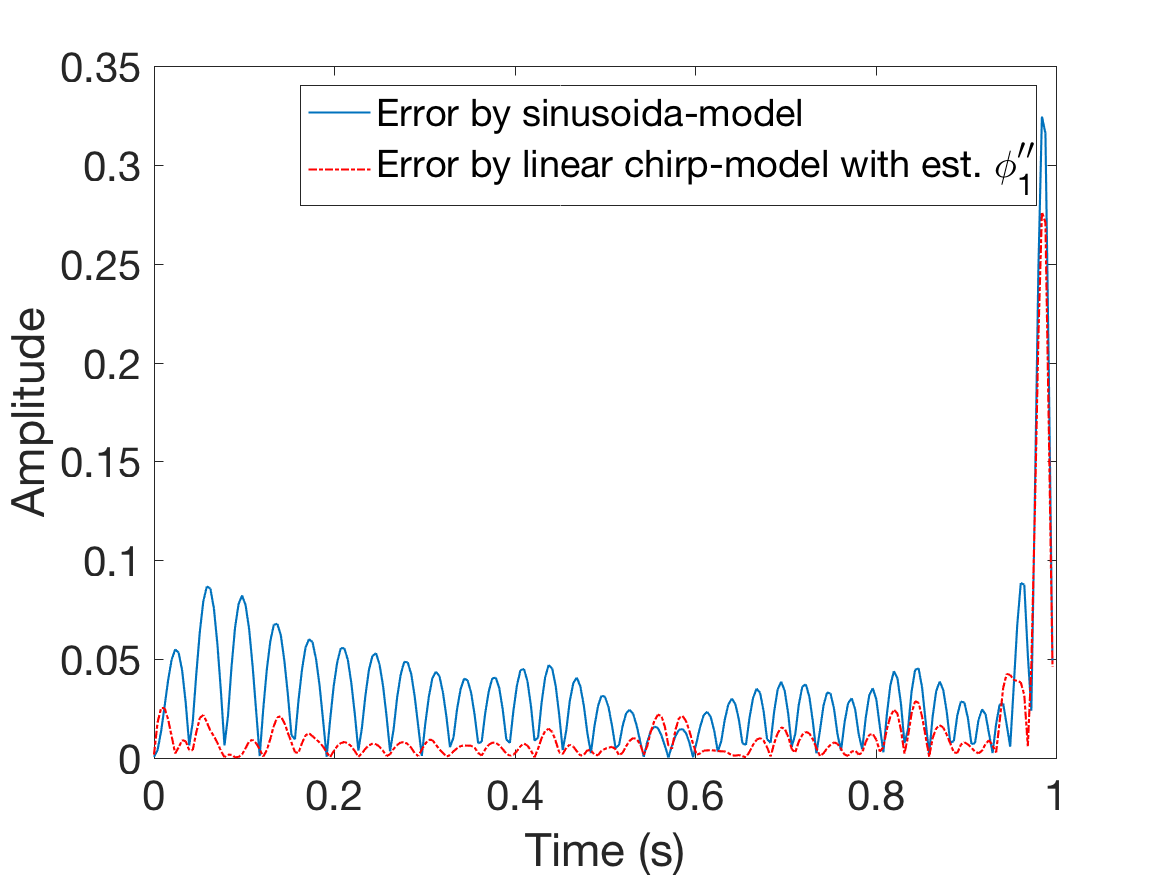}} \quad &\resizebox{2.1in}{1.4in}{\includegraphics{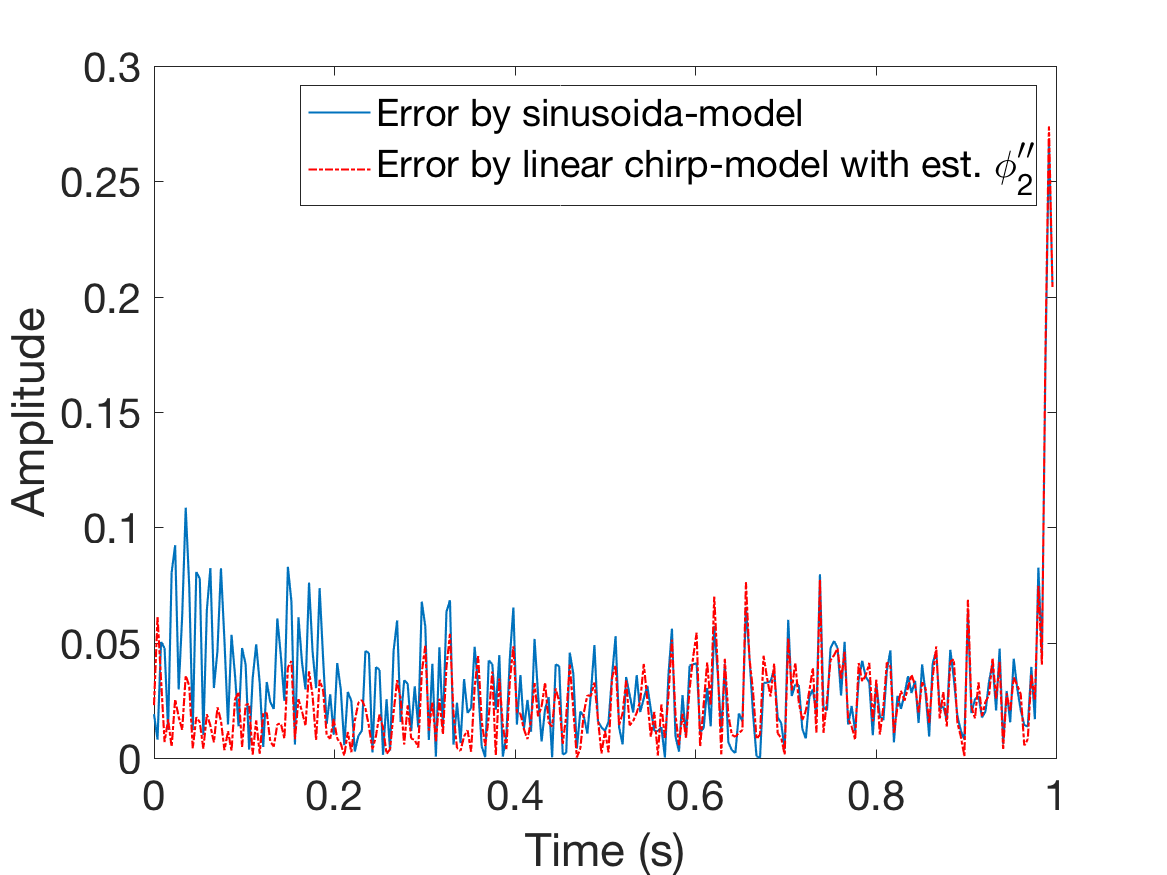}}
	\end{tabular}
	\caption{\small Recovery errors of two-component signal with noise SNR=15dB. Top row:  recovery errors by  sinusoidal signal-based model (blue line) and by linear chirp-based model with ground truth $\phi^{\gp\gp}_\ell(t)$ in \eqref{comp_xk_est_2nd1_real_gaussian} (red dash-dotted line) for  $x_1(t)$ (left panel) and for $x_2(t)$ (right panel); Bottom row:  recovery errors by  sinusoidal signal-based model (blue line) and by linear chirp-based model with 
an estimated $\phi^{\gp\gp}_\ell(t)$ in \eqref{comp_xk_est_2nd1_real_gaussian} (red dash-dotted line) for  $x_1(t)$ (left panel) and for $x_2(t)$ (right panel)
     }
	\label{fig:two_chirp_recovered_LFM_noise15dB}
\end{figure}

\clearpage
We notice that for this three-component signal and the two-component signal discussed above. There is no big difference in IF estimation and component recovery with  $\wh a_k(t)$ or $\wc a_k(t)$. 
%In the following we will use $\wc a_k(t)$ for most experiments.  
In the top row of Fig.\ref{fig:three_component_recovered}, we show 
$1/\wh a_1(t), 1/\wh a_2(t), 1/\wh a_3(t)$ with  $\gs(t)=2.35$ (left panel) and $1/\wc a_1(t), 1/\wc a_2(t)$, $1/\wc a_3(t)$ with  $\gs(t)=\gs_2(t)$ (right panel), which are IF estimates to  $\phi'_1(t)$,  $\phi'_2(t)$  and $\phi'_3(t)$ respectively. 
Here we choose $\gs(t)=2.35$ in that it is the average value of $\gs_2(t)$ for $y(t)$. 
In the top-right panel of Fig.\ref{fig:three_component_recovered}, we show recovery errors $|y_1(t)-\wt W_y(\wh a_1, t)|$ with 
 $\gs(t)=2.35$ and $\gs(t)=\gs_2(t)$. Since $y_1(t)$ is a harmonic mode, conventional CWT can recover it very well. Next we focuss 
 on $y_2(t)$ and $y_3(t)$.

\begin{figure}[th]
	\centering
	\begin{tabular}{ccc}	
	\resizebox{2.1in}{1.4in} {\includegraphics{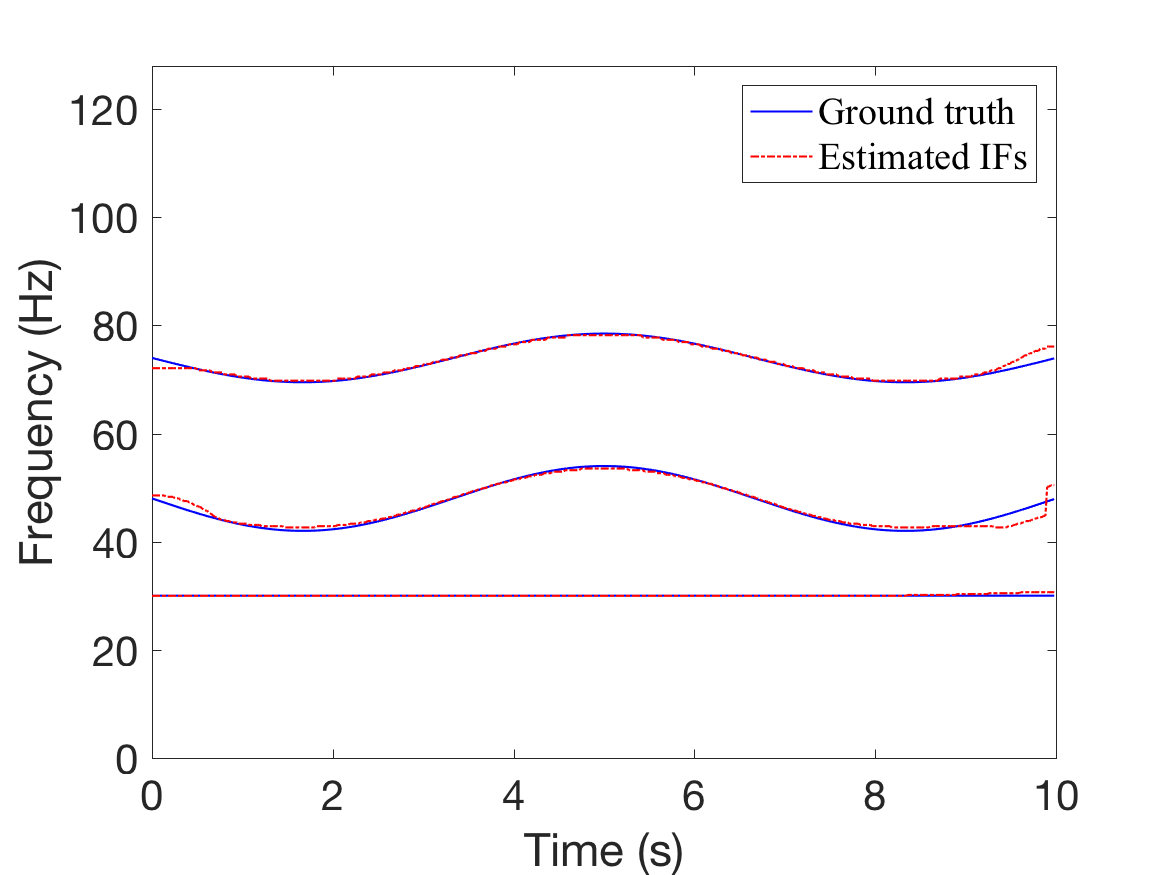}} &
	\resizebox{2.1in}{1.4in}{\includegraphics{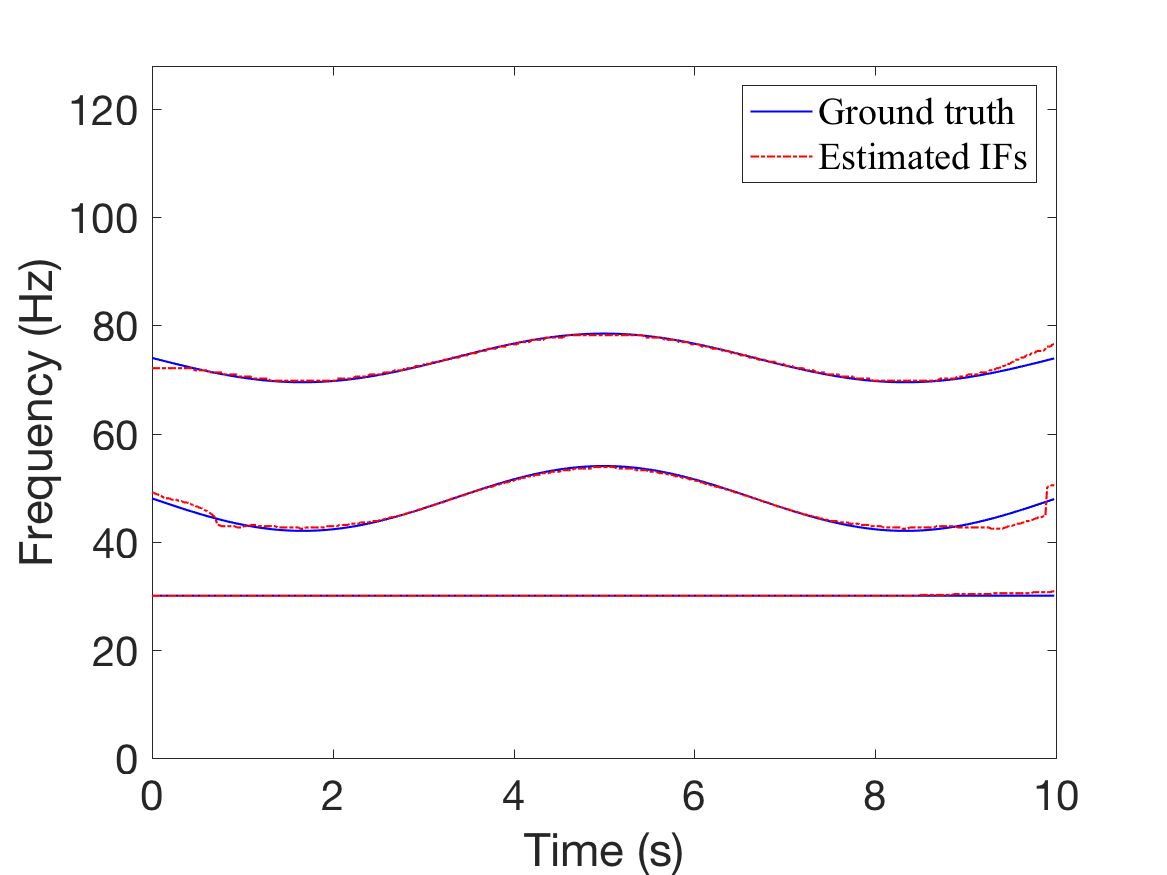}} 
	&
	\resizebox{2.1in}{1.4in}{\includegraphics{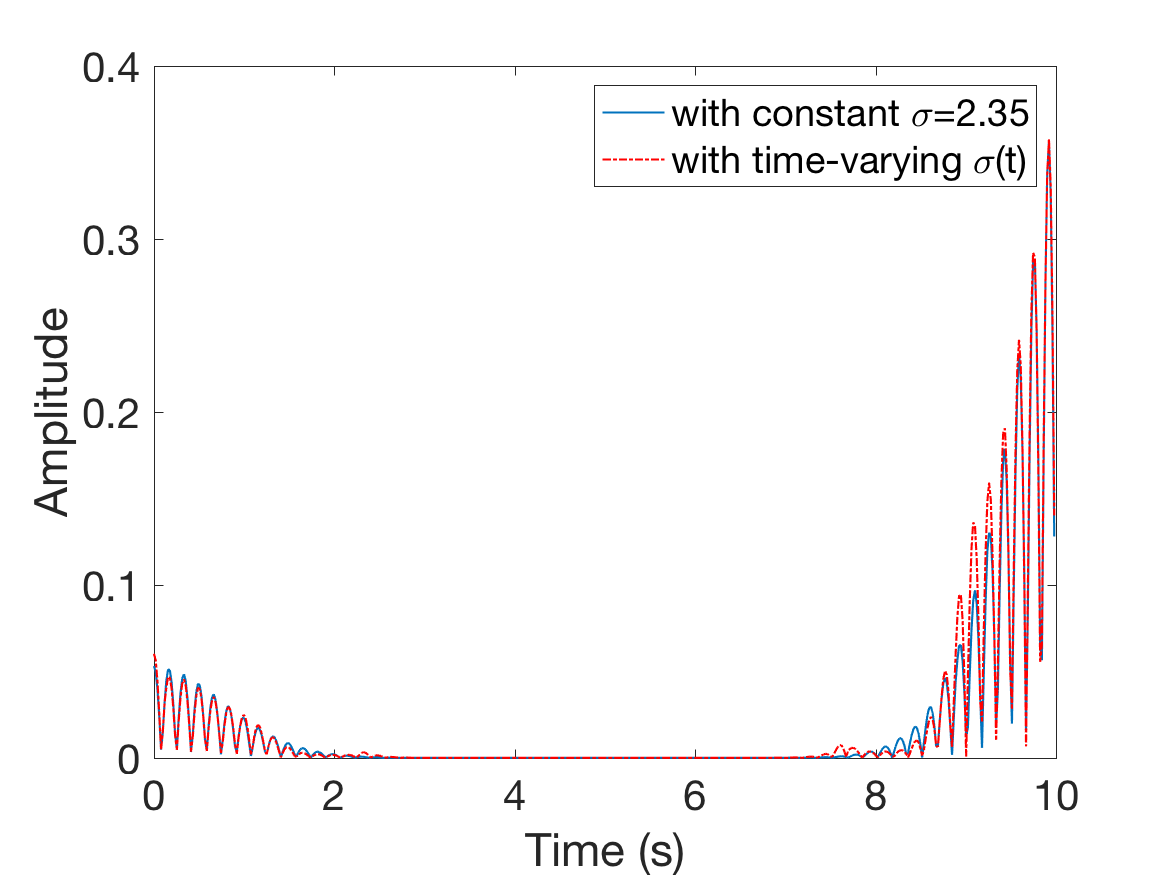}} \\
	
	%\resizebox{2.1in}{1.4in}{\includegraphics{ACWT_y2error_threecomponent_gs_via_ci2_sample_1over51_Oct25_2020.png}} &
	%\resizebox{2.1in}{1.4in}{\includegraphics{ACWT_y3error_threecomponent_gs_via_ci2_sample_1over51_Oct25_2020.png}}

	&	\resizebox{2.1in}{1.4in}{\includegraphics{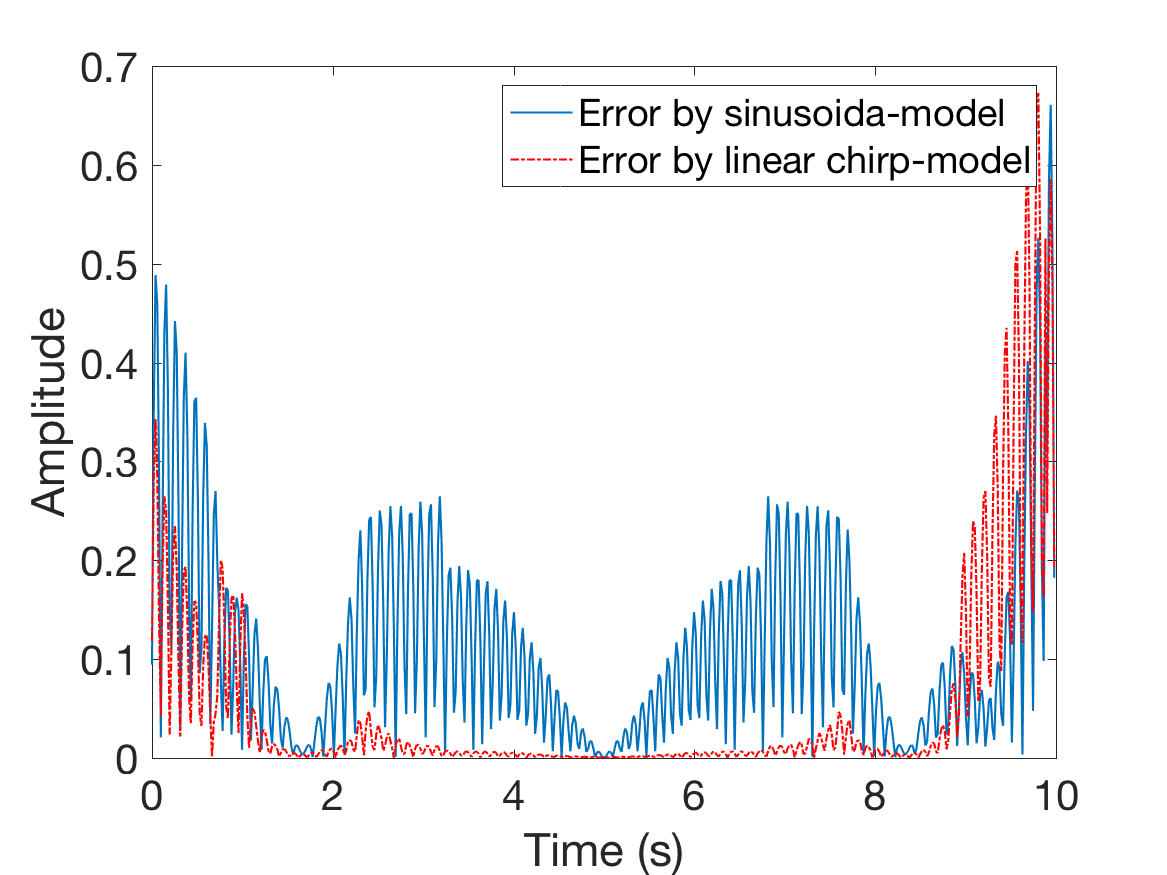}} &
	\resizebox{2.1in}{1.4in}{\includegraphics{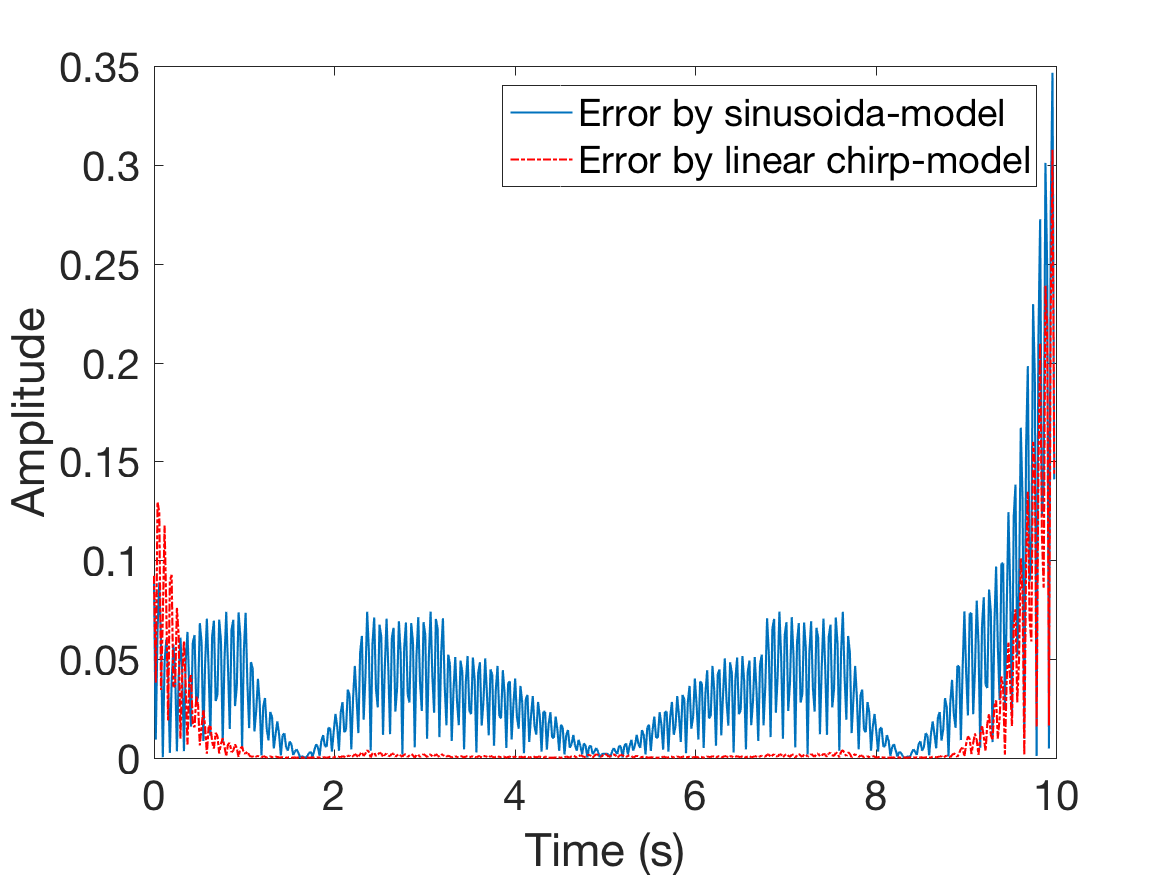}}
	\\
%		\resizebox{2.1in}{1.4in}{\includegraphics{ACWT_y1error_threecomponent_gsci2_sinu_via_LFMestr1_sample_1over51_Oct25_2020.png}} 
&
	\resizebox{2.1in}{1.4in}{\includegraphics{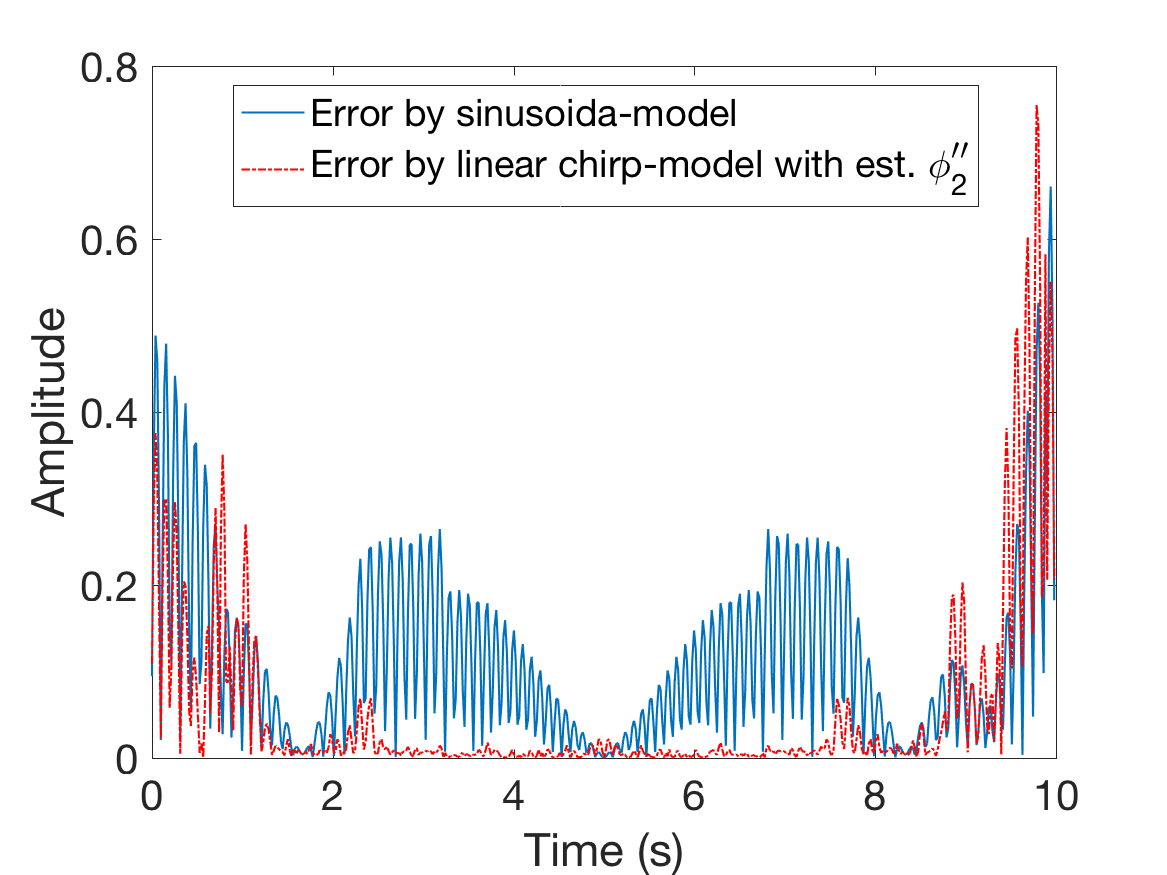}} &
	\resizebox{2.1in}{1.4in}{\includegraphics{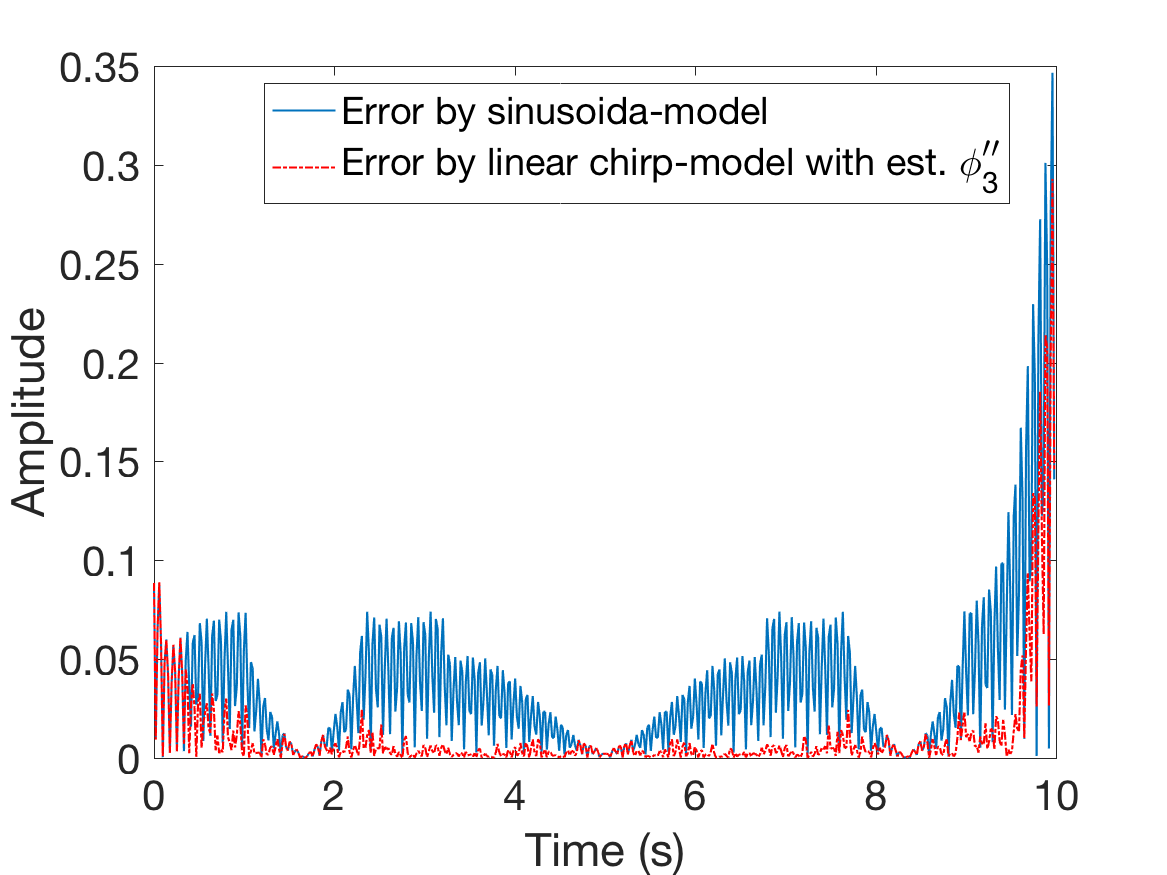}}
	
	\end{tabular}
	\caption{\small Recovery results of three-component signal $y(t)$. 
		Top row: recovered IFs of $y_1(t),   y_2(t), y_3(t)$ with $\gs(t)=2.35$ (top-left panel) and $\gs(t)=\gs_2(t)$ (top-middle panel), component recovery errors of $y_1(t)$  with $\gs(t)=2.35$ and $\gs(t)=\gs_2(t)$ (right panel); Middle row: recovery errors by  sinusoidal signal-based model (blue line) and by linear chirp-based model with ground truth $\phi^{\gp\gp}(t)$ in \eqref{comp_xk_est_2nd1_real_gaussian} (red dash-dotted line) for  $y_2(t)$ (left panel) and for $y_3(t)$ (right panel); 
Bottom row:  recovery errors by  sinusoidal signal-based model (blue line) and by linear chirp-based model with 
an estimated $\phi^{\gp\gp}(t)$ in \eqref{comp_xk_est_2nd1_real_gaussian} (red dash-dotted line) for  $y_2(t)$ (left panel) and for $y_3(t)$ (right panel).}
\label{fig:three_component_recovered}
\end{figure}

%\clearpage 

   In the middle row of Fig.\ref{fig:three_component_recovered}, we show the 
   recovery errors by  sinusoidal signal-based model (blue line) and by linear chirp-based model with ground truth $\phi^{\gp\gp}_2(t)$ in \eqref{comp_xk_est_2nd1_real_gaussian} (red dash-dotted line) for  $y_2(t)$, while the recovery errors for $y_3(t)$ 
by these two models are provided in the right panel. Using five-point formula for differentiation, we have approximations to $\phi^{\gp\gp}_\ell(t)$. With estimated  $\phi^{\gp\gp}_\ell(t)$ used in  \eqref{comp_xk_est_2nd1_real_gaussian},  the recovery errors by linear chirp-based model  (red dash-dotted line) for  $y_2(t)$ and $y_3(t)$ are shown in the bottom row of Fig.\ref{fig:three_component_recovered}. Again, linear chirp-based model leads more accurate signal separation. 
Here we use $\gs(t)=\gs_2(t)$. For $\gs=2.35$,  linear chirp-based model also performs much better than   sinusoidal signal-based model in component recovery.  

In addition, we use the relative root of mean square error (RMSE) to evaluate the errors of IF estimation and component reconstruction. The RMSE is defined as,
\begin{equation*}
%\label{def_RMSE}
{\rm RMSE}_\upsilon := \frac{1}{K} \sum_{k=1}^{K} \frac{\|\upsilon_k-\wh \upsilon_k\|_2} {\|\upsilon_k\|_2}, 
\end{equation*} 
where $\upsilon$ is a vector and $\wh\upsilon$ is an estimation of  $\upsilon$. Note that due to the errors near the endpoints, the RMSE is calculated over $j: \; N/8+1\le j\le 7N/8$, for all methods. We use ACWT (ACWTe respectively) to denote linear chirp-based model with ground truth $\phi^{\gp\gp}_\ell(t)$  (an estimated $\phi^{\gp\gp}_\ell(t)$ respectively) in \eqref{comp_xk_est_2nd1_real_gaussian}. 
The  RMSEs of  ACWT, ACWTe, CWT-based SST  (WSST) and the CWT-based 2nd-order SST (WSST2) and EMD are shown in Fig.\ref{fig:reconstruction_RMSE_three_comp}. 
The results in Fig.\ref{fig:reconstruction_RMSE_three_comp} demonstrate the correctness and efficiency of our proposed linear chirp-based method especially for signal reconstruction.

%%%%%%%%%%%%%%%%%%%the beginning of figure 7%%%%%%%%%%%%%%%
	\begin{figure}[H]
		\centering
		\hspace{-0.3cm}
		\begin{tabular}{cc}
			\resizebox{2.4in}{1.6in}{\includegraphics{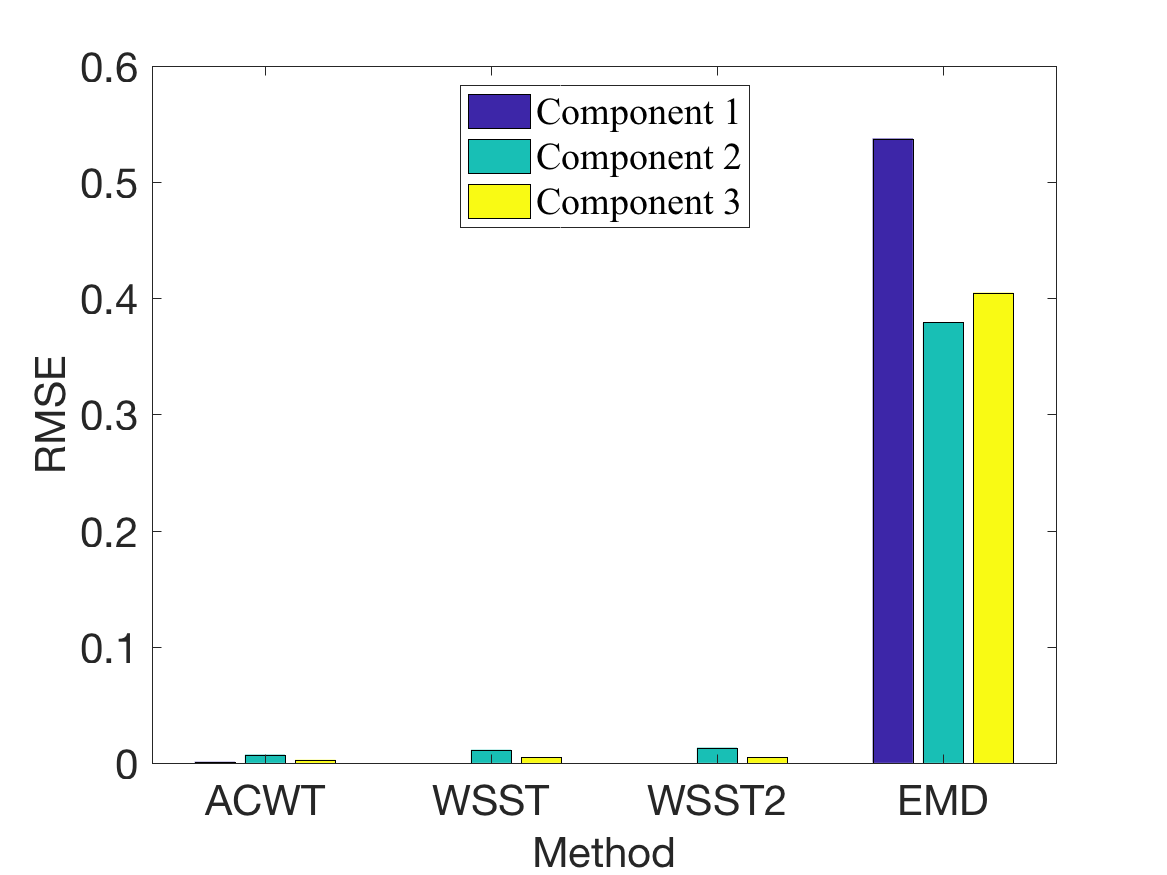}} 
			\resizebox{2.4in}{1.6in}{\includegraphics{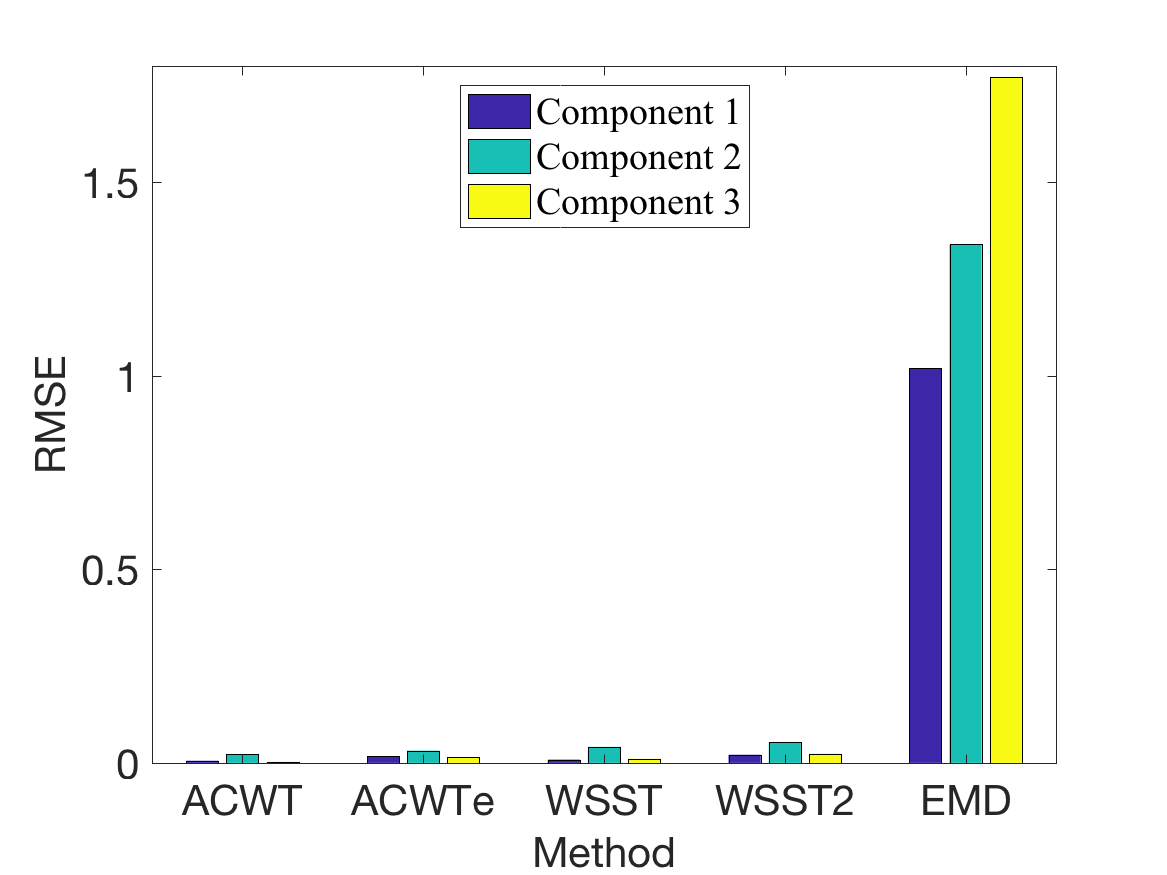}}
		\end{tabular}
		\caption{\small Performance comparisons of different methods for three-component signal $y(t)$. 
			Left:  RMSE of IF estimation; 
			Right: RMSE of component reconstruction.}
		\label{fig:reconstruction_RMSE_three_comp}
	\end{figure}
	%%%%%%%%%%%%%%%%the end of figure 7 %%%%%%%%%%%%%%%%%%%%%

\end{document}